\documentclass[12pt]{amsart}
\usepackage{amsmath, amsthm, amssymb}
\usepackage{fullpage}
\usepackage{amssymb, amsmath}
\usepackage{cancel}
\usepackage{enumitem}
\usepackage{hyperref}
\usepackage[utf8]{inputenc}
\usepackage[normalem]{ulem}

\newtheorem{theorem}{Theorem}
\newtheorem{proposition}{Proposition}

\newtheorem{remark}[proposition]{Remark}
\newtheorem{lemma}[proposition]{Lemma}

\newtheorem{definition}[proposition]{Definition}
\numberwithin{proposition}{section}
\numberwithin{equation}{section}
\usepackage{mathtools}
\mathtoolsset{showonlyrefs}

\usepackage{xcolor}
\usepackage[style=numeric,doi=false,url=false,isbn=false,maxbibnames=99]{biblatex}
\newcommand{\newsection}[1] {\section{#1}\setcounter{theorem}{0}
 \setcounter{equation}{0}\par\noindent}

\newcommand{\N}{{\mathbb N}}
\newcommand{\R}{{\mathbb R}}

\newcommand{\T}{{\mathbb T}}
\newcommand{\jD}{\langle D \rangle }
\newcommand{\NF}{N\!F}
\newcommand{\A}{\mathcal A}
\newcommand{\bmo}{\text{bmo}}

\begin{document}

\title{Enhanced lifespan bounds for  1D quasilinear Klein-Gordon flows}

\thanks{}

\author{}

\address{}

\author{Hongjing Huang}
\author{Mihaela Ifrim}
\author{Daniel Tataru}
\address{Department of Mathematics, University of California, Berkeley, CA 94720}
\email{tataru@math.berkeley.edu}

\begin{abstract}
In this article we consider one-dimensional scalar quasilinear Klein--Gordon equations with general  nonlinearities, on both $\R$ and $\T$.
By employing a refined modified-energy framework of Ifrim and Tataru, we investigate long time lifespan bounds for small data solutions.
Our main result asserts  that solutions with small initial data of size $\epsilon$ persist on the improved cubic timescale $|t| \lesssim \epsilon^{-2}$ and satisfy sharp cubic energy estimates throughout this interval. We also establish difference bounds on the same time scale. In the case 
of $\R$, we are further able 
to use dispersion in order to extend the lifespan to $\epsilon^{-4}$. This generalizes earlier results 
obtained by Delort, \cite{Delort1997_KG1D}  in the semilinear case.

\end{abstract}


\maketitle 

\setcounter{tocdepth}{1} 

\tableofcontents

\newsection{Introduction}

We study the Cauchy problem for one-dimensional quasilinear Klein--Gordon equations on both $\R$ and $\T$.
The quasilinear Klein–Gordon equation arises naturally in relativistic scalar field theory, where it governs the dynamics of scalar fields subject to relativistic invariance and nonlinear self-interactions \cite{Klein1926, Gordon1926}.
We consider two broad classes of scalar
 quasilinear Klein–Gordon equations with a strictly positive mass, $m  > 0$:

\begin{equation}\label{kg0}
- g^{\alpha \beta}(u) \partial_\alpha
\partial_\beta u + m u=f(u,\partial u),
\end{equation}
and
\begin{equation}\label{kg1}
- g^{\alpha \beta}(u,\partial u) \partial_\alpha
\partial_\beta u + m u  =f(u,\partial u),
\end{equation}
where $\partial u$ includes both $x$ and $t$ derivatives.

We assume that the source term $f$ is smooth in its arguments and at least quadratic at $(0,0)$, and the metric $g$ is real-valued, smooth, and satisfies $g^{\alpha\beta}(0, 0)= \eta^{\alpha\beta} $. Here $\eta^{\alpha\beta}$ denotes the Minkowski metric.  We prescribe the initial data in Sobolev spaces 
\[
(u(0), u_t(0)) = (u_0,u_1):=u[0] \in H^{s} \times H^{s-1}.
\]

For the above problems, we consider solutions under a small-data assumption
\begin{equation}\label{small}
\| u[0]\|_{H^{s} \times H^{s-1}} \leq \epsilon \ll 1.
\end{equation}
In this context, we note that classical results
yield local well-posedness 
in $H^{s} \times H^{s-1}$ for 
$s > \frac32$ in the case of 
\eqref{kg0} and $s > \frac52$ 
in the case of \eqref{kg1}, see e.g. \cite{Hormander1997-Lectures},
\cite{Kato1975-ARMA}.
For $\epsilon$-small initial data, these results directly provide a lifespan bound of order $\epsilon^{-1}$. Our goals in this work are twofold:

\begin{enumerate}[label=(\roman*)]
\item To develop cubic energy estimates for the nonlinear Klein-Gordon flows \eqref{kg0} and \eqref{kg1}, including at the leading order  energy level.
    \item To improve the lifespan bound for the solutions from $T\approx \epsilon^{-1}$,  to cubic lifespan i.e. $T\approx\epsilon^{-2}$ in the case of $\T$, respectively 
    $T\approx \epsilon^{-4}$ in the case of $\R$.
\end{enumerate}

\subsection{ The main results} Our first result on the existence of an enhanced lifespan interval for the solutions of the evolutions  \eqref{kg0} and \eqref{kg1}
applies equally on $\T$ and on $\R$, and is stated in the following theorem:

\begin{theorem}\label{t:enhanced}
Consider the equations \eqref{kg0} and \eqref{kg1} 
on 
either $\mathbb{T}$ or $\R$. Then 
\begin{itemize}
\item[a)] The solutions to \eqref{kg0}  with initial data satisfying \eqref{small}, with $s >  \frac{7}{2}  $, exist on the time interval $I^{\epsilon}:=[-c\epsilon^{-2} ,c\epsilon^{-2} ]$ and obey  bounds similar to \eqref{small},  where $c>0$ is a universal constant.
\item[b)] The solutions to  \eqref{kg1}  with initial data satisfying \eqref{small}, with $s > \frac{9}{2} $, exist on the time interval $I^{\epsilon}$ and obey  bounds similar to \eqref{small}.
\end{itemize}
\end{theorem}

In addition to establishing lifespan bounds on the cubic   $\epsilon^{-2}$ 
 time scale, we are also able to prove weak-Lipschitz bounds over the same time  interval: 
\begin{theorem}\label{t:weak-Lip}
    Let $u^1,u^2 \in C(I^{\epsilon}; H^{s}) \cap C^1(I^{\epsilon}; H^{s-1})   $ be  solutions to the same equation (either \eqref{kg0} or \eqref{kg1} as in Theorem~\ref{t:enhanced}) with initial data \( (u_0^j, u_1^j) \), \(j =1,2\). Then we have the Lipschitz difference bounds 
    \[
    \begin{aligned}
        \Vert u^1 [\cdot]- u^2 [\cdot ] \Vert_{L^\infty (H^1\times L^2)}
        \lesssim
        \Vert u^1[0] -u^2[0]
        \Vert_{H^1\times L^2}
    \end{aligned}
    \]
  holding over the same time interval  $I^{\epsilon} = [-c \epsilon^{-2}, c \epsilon^{-2}]$.
\end{theorem}

The only difference between the two parts of the above  results is that the equation \eqref{kg1} is considered in Sobolev spaces which are one derivative higher than for the equation \eqref{kg0}. In particular, \eqref{kg1} can be easily recast as a principally diagonal system of the form  \eqref{kg0} for $u$ and $\partial u$. Hence, in order to avoid repetitions, in this paper we only prove the results for the equation \eqref{kg1}.

 Of independent interest here are the energy estimates for \eqref{kg1}, which represent the main ingredient of the proof of the above theorems. To describe them it is useful to introduce appropriate notions of time dependent \emph{control parameters} as follows
 \begin{equation}
 \label{Ak}
\A_k := \|u\|_{L^\infty} + \|\partial u\|_{W^{k,\infty}}, \qquad k \geq 0.
 \end{equation}

We also note that the quantities \(\A_{k}\) control \( \|\partial^{\le {k+1}}u \|_{ L^\infty}\)  via the equations,
\begin{equation}
\|\partial^{\le {k+1}}u \|_{ L^\infty}
\lesssim_{\A_0} \A_k.
\end{equation}

Indeed, whenever \(u_{tt} \) appears, we use the evolution equation to eliminate it, rewriting \(u_{tt}\) in terms of spatial derivatives of \(u \) and nonlinearities, which are then estimated by \(\A_{k}\) (together with standard Moser-type bounds). For further details, we refer to Proposition~\ref{p:infty}.

Three of these parameters will play a role in our estimates, namely $\A_0$, $\A_2$ and $\A_3$. In terms of these, our typical \emph{cubic energy estimate }will have the form
\[
\frac{d}{dt} E \lesssim_{\A_2} \A_0 \A_3 E.
\]

 We state these  estimates in the next theorem:

 \begin{theorem}\label{t:eepropagation} For each $\sigma \geq 1$ there exists a modified cubic energy 
 $E^{\sigma} (u[t])$ 
 for the full equation \eqref{kg1}, so that we have
\begin{enumerate}[label=(\roman*)]
\item The norm  equivalence
\begin{equation}
 E^{\sigma} (u[t]) \approx_{\A_2}  \Vert u[t]\Vert _{H^{\sigma}\times H^{\sigma -1}}^2,
  \end{equation}
\item The cubic energy estimate
\begin{equation}
\left|\frac{d}{dt}   E^{\sigma} (u[t]) \right|
\lesssim_{\A_2} \A_0 \A_3
\,  E^{\sigma}(u[t]).
\end{equation}
\end{enumerate}
\end{theorem}

We also prove cubic bounds for the associated linearized equation,
which are at the heart of the proof of Theorem~\ref{t:weak-Lip}:

\begin{theorem}\label{t:eepropagation-lin} Let $u$ be a solution to the Klein-Gordon flow \eqref{kg1}, and let $v$ be a solution to the corresponding linearized equation around the solution $u$. Then there exists a modified cubic energy $E_{lin} (v[t])$   so that we have
\begin{enumerate}[label=(\roman*)]
\item The norm  equivalence
\begin{equation}
  E_{lin}(v[t]) \approx_{\A_2}  \Vert v[t]\Vert _{H^{1}\times L^2}^2,
  \end{equation}
\item The cubic energy estimate
\begin{equation}
\label{cubic ee}
\left|\frac{d}{dt}   E_{lin}(v[t])\right|
\lesssim_{ \ \A_2} 
\A_0
\A_3 \,  E_{lin}(v[t]).
\end{equation}
\end{enumerate}
\end{theorem}

We remark that the above theorems provide uniform propagation of energy estimates at every regularity level throughout the cubic time scale.  It is in the proof of these theorems that we employ the \emph{modified energy method}, originally introduced by Hunter, Ifrim, and Tataru in \cite{HunterIfrimTataruWong2015-PAMS}, and subsequently developed and extensively used by Ifrim and Tataru (often in collaboration) in several notable quasilinear settings, particularly in the analysis of water waves and minimal surface models \cite{HunterIfrimTataru2016CMP, IfrimTataru2016BSMF, IfrimTataru2017ARMA, HarropGriffithsIfrimTataru2017AnnPDE, AiIfrimTataru2024Invent, IfrimTataru2019AnalPDE}. This method can be seen as a robust adaptation of normal form methods in quasilinear contexts, where it asserts that 
it is more efficient to modify the energy rather than transform the equation(s).

\begin{remark}  In \eqref{cubic ee}, the regularity of the control norms varies substantially, forcing us to increase the regularity assumed on the initial data. One may naturally  ask whether this can be avoided by using the \emph{balanced energy method}, first introduced by Ai--Ifrim--Tataru in \cite{AiIfrimTataru2022}. However,
as it turns out, this method cannot be implemented in this setting; this is essentially because of the near resonant interactions which are present in the high frequency limit for Klein-Gordon, at least for generic nonlinearities which do not satisfy a null condition.
\end{remark}

With the cubic energy estimates in hand, we now return to the case of the real line, $\mathbb{R}$. There we have further 
dispersive tools at our disposal, 
which allow us to obtain a lifespan of order $\epsilon^{-4}$, at the expense of requiring  additional
regularity for the initial data:

\begin{theorem} \label{stri}
    Suppose that the initial data $u[0] \in H^{s}(\R)\times H^{s-1}(\R)$ 
    satisfy the following smallness bound
     \[
     \Vert
     u[0] \Vert_{H^{s} \times H^{s-1}}
     \leq \epsilon \ll 1.
     \]
    \begin{itemize}
        \item[(i)]  Assume that \( s \ge 6 \frac{1}{4}\) for equation~\eqref{kg0}, and \( s \ge 7 \frac{1}{4}\) for equation~\eqref{kg1}.
     Then there exists a solution $u \in C(I^{\epsilon}; H^{s}) \cap C^1(I^{\epsilon}; H^{s-1})  $
     defined on the time-interval 
     $I^{\epsilon} = [- c \epsilon^{-4}, c  \epsilon^{-4}]$, where $c>0$ is a universal constant.

\item[(ii)]
     Let $u^1,u^2 \in  C(I^{\epsilon}; H^{s}) \cap C^1(I^{\epsilon}; H^{s-1})   $ be  solutions to the same equation (either \eqref{kg0} or \eqref{kg1} as in (i)) with initial data \( u^j[0] \), \(j =1,2\). Then the solution map has Lipschitz difference bounds 
    \[
    \begin{aligned}
        \Vert u^1 [\cdot]- u^2[\cdot] \Vert_{L^\infty (H^1\times L^2)}
        \lesssim
        \Vert u^1[0]- u^2[0] 
        \Vert_{H^1\times L^2}
    \end{aligned}
    \]
  holding over the same time interval  $I^{\epsilon} = [-c \epsilon^{-4}, c \epsilon^{-4}]$.
    \end{itemize}
\end{theorem}

This result extends Delort's earlier result in \cite{Delort1997_KG1D}, which only applies to semilinear Klein-Gordon problems satisfying a null condition. 
Our method of proof also differs. 
In contrast to Delort’s use of $X^{s,b}$-type analysis, our approach combines modified energy estimates with Strichartz estimates. The latter have to be applied with care, since our problem is quasilinear.

At a conceptual level, our results would be relatively straightforward to establish if the nonlinearities were all at least cubic. Thus, it is the quadratic part of the nonlinearity which 
plays the leading role in the proof, and which is in particular relevant for the normal form computations.  We also note that, after dividing by $g^{00}$ we may normalize $g^{00}= -1$.

\medskip

The paper is organized as follows: 
\begin{enumerate}[label=(\roman*)]
\item In Section~\ref{s:notations} we introduce notations and the
paradifferential calculus needed throughout the paper. \item  We then compute the full formal normal forms for our equations, which will be needed for the construction of the modified energy functionals; these computations are carried out in Section~\ref{s:normal-form}.  This prepares us to construct the cubic energy
functional via the modified energy method, for a succession of related equations described in Section~\ref{s:para}. 
\item  After some preparatory estimates
in Section~\ref{s:normal-forms extras},  cubic modified energy functionals are successively constructed in  Section~\ref{s:cubic para ee} for the paradifferential equation, in Section~\ref{s:cubic ee} for the full equation, and then in Section~\ref{s:cubic linearized} for the linearized equation. 
\item Once the cubic energy estimates are available, we use them to prove an enhanced lifespan for the full equation, as detailed in Section~\ref{s:cubic lifespan}. We also derive weak-Lipschitz bounds via modified energy estimates for the linearized equations. 
\item  Finally, for the equations on $\R$ we combine the cubic energy estimates with Strichartz estimates in Section~\ref{s:Strichartz},  in order to obtain a quintic lifespan, i.e.,   $T\approx \epsilon^{-4}$.
\end{enumerate}

\subsection{Previous work}

 Local well-posedness for quasilinear hyperbolic problems in Sobolev spaces is classical.
Following Kato's theory for symmetric hyperbolic systems \cite{Kato1975-ARMA}, and further references in 
\cite{Hormander1997-Lectures,Alinhac2009-HyperbolicPDE},\cite{OzawaNakumura2001}
the local well-posedness in Sobolev spaces $H^{s} \times H^{s-1}$ holds for the  Klein--Gordon equations \eqref{kg0}, \eqref{kg1} provided that $s > \frac32$, respectively $s > \frac52$. 
This naturally raises the question of whether and when one can extend this local result to either a long-time  result or a global one.

The study of long-time well-posedness for a model problem such as the Klein-Gordon flow in \eqref{kg1}, potentially depends on a number of structural features, notably the spatial dimension, the nature of the nonlinearity, and the localization of the initial data. For dispersive equations, the analysis typically becomes more delicate in lower spatial  dimensions, since there are fewer directions in which waves can disperse. Historically, problems with smooth, small and localized initial data in $\R^n$ were considered first,
as they exhibit better decay and lifespan properties. But in the present paper our interest is in the case of non-localized data.

Following the historical perspective, this section is organized as follows. We first review results in $\R^n$ that require  initial data localization, and also possibly null conditions, and where vector field methods play the leading role.
We then turn our attention to results where no initial data localization is assumed. This is where normal form methods are of the essence, and in this context we will also  discuss the development of the \emph{modified energy method } first introduced by Ifrim and Tataru in \cite{HunterIfrimTataruWong2015-PAMS}.\\

In dimensions $d \ge 3$, the picture in the case of localized data is by now fairly well understood:
for quasilinear Klein--Gordon equations with smooth, small, localized data, it is well-known that the global-in-time solution exists, as has been proved independently by Klainerman~\cite{Klainerman1985-CPAM} and Shatah~\cite{Shatah1985}. Thus, in what follows we restrict our attention to previous results in one and two space dimensions.

In H\"ormander’s  analysis in  \cite{Hormander1997-Lectures}, the inhomogeneous Klein--Gordon Duhamel term gains additional decay compared to the standard $t^{-d/2}$ rate—namely, a logarithmic improvement in two dimensions and a full extra power of $t^{-1}$ in one dimension; this gain comes from sharper control of the Duhamel term (a stationary phase argument) leading  to almost global lifespan bounds $\epsilon \log T_{\epsilon} \to \infty$ in $d=2$ and cubic bounds $\epsilon^{2} T_{\epsilon} \to \infty$ in $d=1$. H\"ormander further conjectured that the optimal lifespan bounds must be much stronger, yielding global existence in  $d=2$ and almost-global existence in $d=1$.

Ozawa, Tsutaya and Tsutsumi, in \cite{OzawaTsutayaTsutsumi1996-MathZ}, were the first to prove H\"ormander's  conjecture in $d=2$ for  quadratic semilinear nonlinearities. Their result completed the partial results of Georgiev–Popivanov~\cite{GeorgievPopivanov1991-CPDE}, Kosecki~\cite{Kosecki1992-JDE}, and Simon–Taflin~\cite{SimonTaflin1993-CMP}.  In \cite{OzawaTsutayaTsutsumi1997-NonlinearWaves}, Ozawa, Tsutaya, and Tsutsumi subsequently extended the analysis to the quasilinear case, establishing a global existence theory together with a classical scattering result.

Continuing this line of inquiry, Moriyama, Tonegawa, and Tsutsumi~\cite{MoriyamaTonegawaTsutsumi1997-FE} investigated the one-dimensional Klein–Gordon model, allowing either cubic (or higher) nonlinearities or quadratic nonlinearities of semilinear type. In this setting they established an almost-global well-posedness result together with a classical scattering theory. 

The sharpness of the almost-global well-posedness result in~\cite{MoriyamaTonegawaTsutsumi1997-FE} was established by Yordanov~\cite{Yordanov1996-preprint}, and independently by Keel and Tao~\cite{KeelTao1999-AJM}. On the other hand, for certain classes of quadratic  quasilinear nonlinearities, global existence can in fact be obtained; this was proved by Moriyama in~\cite{Moriyama1997-DIE}.

Completing the above line of one dimensional results, in \cite{Delort2001-ASENS}, Delort proved that, for compactly supported initial data, one can find a \emph{null condition}, under which global existence can be proved. This is matched by the construction in \cite{Delort1999-AIHP} of approximate solutions blowing up at time $e^{A/\epsilon^2}$,  if the null condition is violated. Recently, in \cite{Delort2024AIH} Delort  also showed norm inflation for semilinear equation with cubic nonlinearities violating the null conditions.


\medskip

We now turn to the case where no localization for initial data is imposed. The first result of this type in $\R$ is due to Delort \cite{Delort1997_KG1D}, who obtained the lifespan of  $\epsilon^{-4} \log \epsilon ^{-6}$ for one dimensional semilinear equations, assuming a null condition introduced earlier by Kosecki \cite{Kosecki1992-JDE}.  His arguments were perturbative, using Bourgain type spaces adapted to the  Klein-Gordon flows.  Using a similar strategy, he proves almost global existence in higher dimensions \( d \geq 2\) ; see  \cite{DelortAJM}.
On the other hand, in the periodic case  Delort \cite{Delort1999-AIHP} showed  that, for small \emph{periodic} initial data, the semilinear Klein-Gordon equation has lifespan of $\epsilon^{-2}$ with $r=2$, and $\epsilon^{-(r-1)} \log\epsilon^{-(r-3)}$ for $r \geq 3$ for nonlinearities which are at least of order $r$. 
In the same paper, he also constructed examples showing optimality of this result (up to the logarithmic factors).

We also point out another line of results
that still applies in the quasilinear case, but under the assumption that the mass is outside a measure zero set. 
Precisely, in  \cite{Delort2012-Quasi}, Delort showed that there exists almost global solutions on $\T$. Further models also including potentials were considered in \cite{Delort2009TAMS}.


As it was understood over the years, enhanced lifespan results
in the nonlocalized data case 
are closely related to normal form methods, as first pointed out by Simon\cite{Simon1983-LettMP} and Shatah~\cite{Shatah1985}. However,
while normal forms can be applied directly for Klein-Gordon in the semilinear case, 
quasilinear problems bring forth an added difficulty which is that the associated normal form transformations are in general unbounded. A robust 
way to bypass this difficulty is the \emph{modified energy method}, which originates in the work of Hunter–Ifrim–Tataru–Wong \cite{HunterIfrimTataruWong2015-PAMS}, where it was used to obtain enhanced lifespan estimates for small solutions of the Burgers–Hilbert equation. 
Their key observation was that the loss of derivatives in normal forms in the the quasilinear setting can be avoided  by encoding the normal form corrections at the level of the energy functional itself, for carefully chosen energy functionals.

 This point of view was first carried out in \cite{HunterIfrimTataru2016CMP} for water waves equations, where cubic modified energy estimates were established, and was later refined in \cite{AiIfrimTataru2022} into the so-called \emph{balanced energy estimates}, in which the derivatives on the control parameters are more evenly distributed. 
 In the context of water waves \cite{HunterIfrimTataru2016CMP,IfrimTataru2016BSMF}, the modified energy method provides a robust framework to prove global well-posedness for small, spatially localized data. The same methodology has since been adapted to several other quasilinear models: capillary water wave systems in \cite{IfrimTataru2017ARMA}, gravity water waves in \cite{HarropGriffithsIfrimTataru2017AnnPDE}, and the minimal surface equation in \cite{AiIfrimTataru2024Invent}. These works show that the modified (and balanced) energy method can be used systematically to obtain long time well-posedness result for a broad class of quasilinear dispersive equations.

\subsection{Acknowledgments} 
The first author was partially supported by the NSF grant DMS-2348908, as well as by a Vilas Associate Fellowship.  
The second author was supported by the NSF grant DMS-2348908, by a Miller Visiting Professorship at UC Berkeley during the Fall semester of 2023,  by the Simons Foundation through a Simons Fellowship in the Spring semester of 2024, and by a Vilas Associate Fellowship.
The third author was supported by NSF grant DMS-2054975 and by a Simons Fellowship from the Simons Foundation.

\section{Notations and Preliminaries}\label{s:notations}

\subsection{Notations}
In this section, we introduce  notations that we will use in this work. First, we denote the linear Klein-Gordon operator by
\[
L_{KG} : = \partial_t^2 -\partial_x^2 +m.
\]

We also let $\Lambda_n$ denote the truncation operator that discards all terms of homogeneity different from $n$ in multilinear expressions.


\subsection{Function spaces} Here we review some of the function spaces and linear estimates used later in the paper. We briefly recall the standard 
 Littlewood-Paley decomposition in frequency
\begin{equation*}
1=\sum_{k\in \mathbf{N}}P_{k},
\end{equation*}
where the multipliers $P_k$ have smooth symbols localized at frequency $2^k$, with $P_0$ selecting the frequencies $|\xi| \lesssim 1$. For the dyadic portions of a function $f$ we will  use the notation $f_k := P_k f$, $f_{<k} := P_{<k} f$, so that its Littlewood-Paley decomposition reads
\[
f = \sum_{k \geq 0} f_k.
\]

A good portion of our analysis happens at the level of inhomogeneous
Sobolev spaces ${H}^{s}$, whose norm is given by
\begin{equation*}
\Vert f\Vert_{{H}^{s}}\sim \Vert ( \sum_{k \in \N}\vert 2^{ks} P_{k}f\vert^2 )^{1/2}  \Vert _{L^2}=
\| 2^{ks} P_k f \|_{L^2_{\alpha} \ell^2_k}.
\end{equation*}
We will also use  the (inhomogeneous) Littlewood-Paley
square function and its restricted version,
\begin{equation*}
\displaystyle S(f)(\alpha):=\bigg( \sum_{k\in {\mathbf N}} |P_k(f)(\alpha)|^2\bigg)^\frac{1}{2}, \qquad S_{>k}(f)(\alpha) := \left( \sum_{j > k} |P_j f|^2 \right)^\frac12.
\end{equation*}
The Littlewood-Paley inequality  is recalled below
\begin{equation}\label{lp-square}
   \displaystyle \|S(f)\|_{L^p({\mathbb R})}\simeq_{p} \|f\|_{L^p({\mathbb R})}, \qquad  1<p<\infty.
   \end{equation}
By duality this also yields the estimate
\begin{equation}
\label{useful}
\Vert \sum_{k\in \mathbf{N}}P_{k}f_{k}\Vert_{L^p}\lesssim
\Vert \sum_{k\in \mathbf{N}}(\vert f_{k}\vert ^2)^{1/2}\Vert_{L^p}, \qquad 1 < p < \infty,
\end{equation}
for any sequence of functions $\left\{ f_k\right\}_k\in  L^p_{\alpha} l^2_k$.

We will also use the local BMO space, usually denoted by $\bmo$, which agrees with $L^\infty$ at low frequency and with $BMO$ at high frequency.
For real $s$ we define the inhomogeneous spaces\footnote{These are the same as the inhomogeneous Triebel-Lizorkin spaces $F^{s}_{\infty,2}$.} $\bmo^s$
with norm
\[
\| u\|_{\bmo^s}  = \| \jD^s u\|_{\bmo}.
\]
\subsection{Coifman-Meyer and and Moser-type estimates}
\label{ss:coifman} 
With the basic Littlewood-Paley theory in hand, we move on to bilinear estimates,
where we begin with the
paraproduct  decomposition  of the product of two functions, is essential
\[
 f g = \sum_{k > l+4} f_{l} g_k +  \sum_{k > l+4} f_{k} g_{l} + \sum_{|k-l| \leq 4}
f_k g_l := T_f g + T_g f + \Pi(f,g),
\]
where we recall that we are using the inhomogeneous setup, and in particular
$k,l \geq 0$ above.

By a slight abuse of notation, in the sequel we will omit the frequency separation 
from our notations in bilinear Littlewood-Paley decomposition; for instance instead of the 
above formula we will use the shorter expression
\[
  f g = \sum_{k \in \N} f_{<k} g_k +  \sum_{k \in \N} f_{k} g_{<k} + \sum_{k \in \N} f_k g_k.
\]
Paraproducts may also be thought of as 
 belonging to the larger class of translation invariant bilinear operators.
 Such operators 
 \[
 f,g \to B(f,g)
 \]
 may be described by their symbols $b(\eta,\xi)$ in the Fourier space, by
 \[
 \mathcal F B(u,v)(\zeta) = \int_{\xi+\eta = \zeta} b(\eta,\xi) \hat f(\eta) \hat g(\xi) \, d \xi.
 \]

For more rigorous definitions,
let $\chi_{lh}(\xi_1,\xi_2)$, $\chi_{hh}(\xi_1,\xi_{2})$ be two non-negative bump functions, smooth on the dyadic scale,
\begin{align}
    \chi_{lh}(\xi_1,\xi_2) := \begin{cases}  1 ,& \text{when }   \langle \xi_1\rangle \leq\frac{1}{20} \langle\xi_2\rangle \\
    0 ,& \text{when } \langle\xi_1\rangle \geq\frac{1}{10} \langle\xi_2\rangle,
    \end{cases}
\end{align}
\begin{align}
    \chi_{hh}(\xi_1,\xi_2) := \begin{cases}  1 ,& \text{when }   \frac{1}{10} \leq \frac{\langle\xi_1\rangle}{\langle\xi_2\rangle} \leq 10\\
    0 ,& \text{when } \langle\xi_1\rangle \leq \frac{1}{20} \langle\xi_2\rangle \text{ or  } \langle\xi_2\rangle\leq \frac{1}{20} \langle\xi_1\rangle,
    \end{cases}
\end{align}
and such that $\chi_{lh}(\xi_1,\xi_2) + \chi_{lh}(\xi_2,\xi_1) + \chi_{hh}(\xi_1,\xi_2) =1$. The $\chi_{lh}(\xi_1,\xi_2)$ symbol selects the low--high portions of the bilinear terms and the $\chi_{hh}(\xi_1,\xi_2)$ selects the high--high portions of the bilinear terms. Associated to these symbols, we introduce corresponding projections bilinear operators.  

In particular the paraproduct $B(f,g) = T_f g$ is  defined to have symbol
\[
b(\xi_1,\xi_2) = \chi_{lh}(\xi_1, \xi_2+\frac12\xi_1),
\]
where the entry $\xi_2+\xi_1/2$ is the average of the $g$ input frequency 
and the output frequency. This corresponds exactly to using the Weyl calculus, which is a convenient choice
of quantization for our energy estimates.
We also remark that the set-up in this paper corresponds to inhomogeneous spaces, and in particular the interactions of two low ($\lesssim 1$) frequencies are placed in the high-high box.

\smallskip

We now discuss bilinear paraproduct bounds. Many of these bounds are relatively standard, like the classical Coifman-Meyer bounds 
and some of their generalizations; here are some classical references for the interested reader \cite{CoifmanMeyer1975, CoifmanMeyer1978, Meyer1991, Calderon1965}. Several other bounds are 
more customized for the work we pursue here, and for that we cite the work in \cite{HunterIfrimTataru2016CMP}.

Away from the exponents $1$ and $\infty$ one has a full set of estimates
\begin{equation}\label{CM}
\| T_f g\|_{L^r} + \| \Pi(f,g)\|_{L^r} \lesssim \|f\|_{L^p} \|g\|_{L^q}, \qquad \frac{1}r = \frac{1}{p} +
\frac{1}{q}, \qquad 1 < p,q,r < \infty.
\end{equation}
Corresponding to $q = \infty$ one also  has a $\bmo$ estimate 
\begin{equation}\label{CM-BMO}
\| T_f g\|_{L^p} + \| \Pi(f,g)\|_{L^p} \lesssim \|f\|_{L^p} \|g\|_{\bmo},  \qquad 1 < p < \infty,
\end{equation}
while for the remaining product term we have the weaker bound
\begin{equation}\label{CM+}
\| T_g f\|_{L^p}  \lesssim \|f\|_{ W^{s,p}} \|g\|_{\bmo^{-s}},  \qquad 1 < p < \infty, s > 0.
\end{equation}

Next we consider some similar product type estimates involving $bmo$ and
$L^\infty$ norms.

\begin{proposition}[\cite{HunterIfrimTataru2016CMP}]\label{p:bmo}
a) The following estimates hold:
\begin{equation}\label{bmo-bmo}
\| \Pi(u,v) \|_{\bmo} \lesssim \| u\|_{\bmo} \|v\|_{\bmo},
\end{equation}

\begin{equation}\label{bmo-infty} 
\| T_u v\|_{\bmo} \lesssim \| u\|_{L^\infty} \|v\|_{\bmo},
\end{equation}
\begin{equation}\label{bmo>infty}
\| T_u v\|_{\bmo} \lesssim \| u\|_{\bmo^{-\sigma}} \|v\|_{\bmo^\sigma}, \quad 
\qquad \sigma > 0.
\end{equation}

b) For $s > 0$ the space $L^\infty \cap \bmo^{s}$ is an algebra,
\begin{equation}\label{bmo-alg}
\| uv\|_{\bmo^{s}} \lesssim \| u\|_{L^\infty} \|v\|_{\bmo^s}+
 \| v\|_{L^\infty} \|u\|_{\bmo^s}.
\end{equation}

c) In addition, the following Moser estimate holds for  smooth functions $F$ vanishing at $0$:
\begin{equation}\label{bmo-moser}
\| F(u)\|_{\bmo^{s}} \lesssim_{\|u\|_{L^\infty}} \|u\|_{\bmo^s}.
\end{equation}
\end{proposition}

This Proposition is a paradifferential reformulation of 
results from \cite{HunterIfrimTataru2016CMP}, with the minor adaptation to the local BMO spaces. Its proof is a direct consequence 
of the square function characterization of $\bmo$, and is left
for the reader.

A more standard algebra estimate and the corresponding Moser bound is as follows:
 \begin{lemma}
Let $\sigma > 0$. Then $ H^\sigma \cap L^\infty$ is an algebra, and
\begin{equation}
\| fg\|_{H^\sigma} \lesssim \| f\| _{ H^\sigma} \|g\|_{L^\infty} +
\|f\|_{L^\infty}  \| g\| _{H^\sigma}.
\end{equation}
In addition, the following Moser estimate holds for a smooth function $F$ vanishing at 0:
\begin{equation}\label{bmo-hs}
\| F(u)\|_{ H^\sigma } \lesssim_{\|u\|_{L^\infty}} \|u\|_{ H^\sigma }.
\end{equation}
\end{lemma}

We do not include the proof  of these lemmas as they can can be found in \cite{HunterIfrimTataru2016CMP}. 
Also, we need to have the following commutator and associativity bounds:
\begin{lemma}[Para-commutators] \label{para-com}
We have
\[
\| T_f T_g -T_g T_f \|_{H^s \rightarrow H^{s + 2 }}
\lesssim
\| \jD f \|_{L^{\infty}} \|  \jD  g\|_{L^{\infty}}
\]
\end{lemma}

\begin{lemma}[Para-associativity]\label{l:para-assoc}
We have
\begin{equation}
\| T_f \Pi(v, u) - \Pi(v, T_f u)\|_{ H^{s +2 }} \lesssim 
\|\jD f \|_{L^{\infty}}\|\jD v\|_{L^{\infty}} \|u\|_{ H^{s}}
\end{equation}
\end{lemma}

The proofs can be found in \cite{AiIfrimTataru_AIHPC142}.

\smallskip

\subsection{Symbol classes and bilinear bounds}
 \begin{definition}
We say a bilinear symbol \( m(\xi_1,\xi_2) \)  belongs to the class \( S^n \)  if for every multi-index \( \alpha = (\alpha_1, \alpha_2 ) \), 
\[
|\partial^{\alpha_1}_{\xi_1} \partial^{\alpha_2}_{\xi_2} m(\xi_1,\xi_2)| 
\lesssim_{\alpha}
 \langle \xi \rangle^{n - |\alpha|}.
\]    
\end{definition}
 A special class of such operators, which we denote by $L_{lh}$, will play 
 an important role later in the paper:
 
 \begin{definition}\label{d:Llh}
 By $L_{lh}$ we denote translation invariant bilinear forms whose symbol
$\ell_{lh}(\xi_1,\xi_2)$ is  in $S^0$ and is supported in $\{\langle\xi_1\rangle \ll \langle \xi_2\rangle\}$.

\end{definition}

We remark that in particular the bilinear form $B(f,g) = T_f g$
is an operator of type $L_{lh}$, while $\Pi(f,g)$ is an operator of type $L_{hh}$.
The $L^p$ bounds and the commutator estimates for such bilinear form mirror exactly the similar bounds for paraproducts.

 \begin{definition}
    For \( m\geq 0 \) we say that a low--high symbol  $a(\xi_1,\xi_2) \in S^{m,n}$  if it is supported in the low--region \( |\xi_1| \le  \frac{1}{20} |\xi_2 | \) and admits a polyhomogeneous\footnote{This is a slightly imperfect use of the polyhomogeneous terminology, though less than immediately apparent since 
    in practice the symbols $b_j$ will always be also classically polyhomogeneous.}
    type expansion
   \[
a(\xi_1,\xi_2) = \sum_{j=0}^m 
\xi_1^j b_j(\xi_1,\xi_2), \qquad b_j \in S^{n}.
\]
 \end{definition}
A matching definition applies to high-low symbols.

\begin{lemma}
Let A(u,v) be a bilinear operator with symbol $a(\xi_1,\xi_2)$ in the anisotropic
Coifman–Meyer class $S^{m,n}$
and assume $a(\xi_1,\xi_2)$ is supported in the low--high region $\xi_1 \ll\xi_2 $.

Then for any $s > 0$, we have

\begin{equation} \label{B-CM}
    \Vert A(u,v) \Vert_{L^2} \lesssim \Vert u\Vert_{W^{m,\infty}} \Vert  v \Vert_{H^n},
\end{equation}

\begin{equation}\label{BMO-shift}
    \Vert A(u,v) \Vert_{L^2} \lesssim \Vert u\Vert_{\bmo^{m-s}} \Vert  v \Vert_{H^{n+s}}.
\end{equation}

\end{lemma}

\begin{proof}
For the symbol $a(\xi_1,\xi_2) \in S^{m,n}$, there exists a polyhomogeneous expansion $b_j(\xi_1,\xi_2) \in S^{n}$ such that 
\[
a(\xi_1,\xi_2) = \sum_{j=0}^m 
\xi_1^j b_j(\xi_1,\xi_2) ,  \qquad  b_j(\xi_1,\xi_2) \in S^{n}.
\]

 Consequently, the corresponding bilinear operator can be decomposed as
 \[
 A(u,v) = \sum_{j=0}^m B_j(\partial^j u,v) = \sum_{j=0}^m C_j(\partial^j u,\jD^{n}v),
 \]
 where  \( c_j \in S^0 \) for each \( j\).

 \smallskip
For \eqref{B-CM}, using the above decomposition and the standard Coifman--Meyer bound, we obtain
\[
\begin{aligned}
    \| A(u,v) \|_{L^2} \le \sum_{j=0}^m \| C_j(\partial^j u, \jD^{n} v) \|_{L^2} 
    \lesssim \| u \|_{W^{n,\infty}} \| v \|_{H^n}.
\end{aligned}
\]

\smallskip

Fix \( s>0\). By  Bernstein's inequality, we have
\[
\| P_{\le k} u  \|_{L^{\infty}} \lesssim  2^{ks}\| u\|_{\bmo^{-s}}.
\]
In particular, for  each dyadic piece \( k \ge 0\) we have
\[
\begin{aligned}
    \| \partial^j u_{\le k} \|_{L^{\infty}} \| \jD^{n-j} v_k \|_{L^2} 
    & \lesssim 2^{k(j+s)} \| u \|_{\bmo^{-s}}  2^{k(n-j)} \| v_k \|_{L^2} 
    \\
    & \lesssim 2^{k(n+s)} \| u \|_{\bmo^{-s}}  \| v_k \|_{L^2}.
\end{aligned}
\]
Summing in \( k\) and using almost orthogonality yields the desired estimate.

\end{proof}

\subsection{Multipliers and bilinear forms on real-valued functions}

\label{s:real}

For our work here, it is important to work with multipliers and bilinear forms which map real-valued functions to real-valued functions.
A multiplier $M(D)$ takes real-valued functions to real-valued functions
if 
\[
m(-\xi)= \overline{m(\xi)}.
\]
Equivalently, $m(\xi)$ must be the sum of a real, even function and a purely imaginary, odd function.

This carries over to bilinear translation invariant operators. 
A bilinear form $A$ takes real-valued pairs of functions to real-valued 
functions if its symbol satisfies 
\[
a(-\xi_1,-\xi_2) = \overline{a(\xi_1,\xi_2)},
\]
with a similar odd/even decomposition.
The class of these symbols forms an algebra with respect to addition and multiplication.

\section{Normal form analysis}
\label{s:normal-form}

Our objective in this section is to find a \emph{formal} normal form transformation of the following form
\begin{equation}\label{NFT}
 \mathbf{u} = u + A(u,u) + B(u_t, u_t) + C(u_t,u),
\end{equation}
and such that the corresponding $\mathbf{u}$-equation does not contain quadratic terms. Here we want the translation invariant bilinear forms $A$, $B$ and $C$ to map real-valued functions to real-valued functions, see the discussion in Section~\ref{s:real}.

For the normal form computation, only the quadratic terms are relevant. Accordingly,  we rewrite the equation \eqref{kg1} by
keeping the linear part on the left-hand side and moving the bilinear terms to the right-hand side, while temporarily neglecting the higher-order terms,
\begin{equation}\label{quadratic-KG}
\partial^2_t u - \partial_x^2 u + mu = Q_{00}(u_t,u_t) + Q_{01}(u_t,u) + Q_{11}(u,u).
\end{equation}

\noindent We assume that the symbols of $Q_{00}$ and $Q_{11}$ are symmetric in $(\xi_1,\xi_2)$, whereas no symmetry is imposed on the symbol of $Q_{01}$. 

More precisely, in view of the explicit structure of the equation~\eqref{kg1}, the associated bilinear operators $Q_{ij}$ have symbols that are polynomial in the frequency variables $(\xi_1, \xi_2)$ and belong to finite-dimensional spaces determined by the differential structure of the nonlinearities.

In particular, their symbols satisfy
\begin{equation} \label{eq:s-poly}
\left\{
    \begin{aligned}
q_{00} &\in \mathrm{span}\{1,\ i(\xi_1+\xi_2)\},\\
q_{01} &\in \mathrm{span}\{1,\ i\xi_1,\ i\xi_2,\ \xi_1\xi_2,\ \xi_2^2\},\\
q_{11} &\in \mathrm{span}\{1,\ i(\xi_1+\xi_2),\ \xi_1\xi_2,\ \xi_1^2+\xi_2^2,\ i(\xi_1^2\xi_2+\xi_1\xi_2^2)\}.
\end{aligned}
\right.
\end{equation}

Further, for  the analysis of  low--high interactions, we will need a precise  power series expansion expressed in powers of the high frequency.
For the symbols $q_{00}$ and $q_{11}$, which are symmetric with respect to the exchange
$(\xi_1,\xi_2)$, it suffices to consider the low--high frequency regime
$|\xi_1|\ll|\xi_2|$. The complementary high--low interactions then follow immediately
by symmetry, upon exchanging the roles of $\xi_1$ and $\xi_2$.
\[
q_{11}(\xi_1,\xi_2)=\sum_{k=0}^{2} q_{11}^{(k)}(\xi_1)\,\xi_2^{k},
\qquad
q_{00}(\xi_1,\xi_2)=\sum_{k=0}^{1} q_{00}^{(k)}(\xi_1)\,\xi_2^{k}.
\]
But since $q_{01}$ is not symmetric, we need to consider  expansion of the high-low and low--high separately. 
\[
q_{01}(\xi_1,\xi_2)=\sum_{k=0}^{2} q_{01}^{(k)}(\xi_1)\,\xi_2^{k},
\qquad
q_{01}(\xi_1,\xi_2)=\sum_{k=0}^{1} \Tilde{q}_{01}^{(k)}(\xi_2)\,\xi_1^{k}.
\]

Recalling that the linear Klein-Gordon operator is denoted by 
\[
L_{KG} = \partial^2_t - \partial_x^2 + m,
\]
we are now ready to state our main normal form result:

\begin{proposition}[Normal form transformation] \label{nfl}

 There exists a unique normal form transformation as in \eqref{NFT}
such that the quadratic nonlinearities in \eqref{quadratic-KG} are removed, i.e., 
    \[
   - \Lambda_2( L_{KG} ( A(u,u) + B(u_t,u_t) + C(u_t,u) ) )
   = 
Q_{00}(u_t,u_t) + Q_{01}(u_t,u) + Q_{11}(u,u).
    \]
  Furthermore, the symbols of the bilinear forms $A(u,u)$, $B(u_t,u_t)$
  and $C(u_t,u)$ can be described as follows:
\begin{enumerate}[label=(\roman*)]
\item   In the low--high case we have the following symbol expansions:
\begin{equation}\label{nf-hl}
\begin{aligned}
    a_{lh}(\xi_1,\xi_2) = & a_0(\xi_1) \xi_2  + a_1(\xi_1) + O((1+|\xi_1|^4) \xi_2^{-1}), \qquad \,  a_0 \in S^2, \ \ a_1 \in S^{3},
\\
    b_{lh}(\xi_1,\xi_2) = & b_0(\xi_1)  + b_1(\xi_1) \xi_2^{-1} + O( (1+|\xi_1|^3)\xi_2^{-2}), \qquad b_0 \in S^1, \ \ 
   b^1 \in S^{2},
\\
     c_{lh}(\xi_1,\xi_2) = & c_0^1(\xi_1) \xi_2  + c_1^1(\xi_1) + O( (1+|\xi_1|^3)\xi_2^{-1}), \qquad \  \, c_0^1 \in S^1, \ \ c_1^1 \in S^2,
\\
     c_{hl}(\xi_1,\xi_2) = & c_0^2(\xi_2)  + c_1^2(\xi_2) \xi_1^{-1} + O( (1+|\xi_2|^4)\xi_1^{-2}), \qquad c_0^2 \in S^2, \ \ c_1^2 \in S^3.
\end{aligned}
\end{equation}

\item    For the high--high interactions 
we have the structural decompositions
\begin{equation}\label{eq:hh-structure}
\begin{aligned}
a_{hh}(\xi_1,\xi_2) &= (\xi_1+\xi_2)\, \widetilde a_{hh}(\xi_1,\xi_2) + a_{hh}^{(0)}(\xi_1,\xi_2),\\
b_{hh}(\xi_1,\xi_2) &= (\xi_1+\xi_2)\, \widetilde b_{hh}(\xi_1,\xi_2) + b_{hh}^{(0)}(\xi_1,\xi_2),\\
c_{hh}(\xi_1,\xi_2) &= (\xi_1+\xi_2)\,  \widetilde c_{hh}(\xi_1,\xi_2) + c_{hh}^{(0)}(\xi_1,\xi_2),
\end{aligned}
\end{equation}
where
\[
 \widetilde a_{hh},\,a_{hh}^{(0)}\in S^{2} ,\qquad
 \widetilde b_{hh},\,b_{hh}^{(0)}\in S^{0},\qquad
 \widetilde c_{hh},\,c_{hh}^{(0)}\in S^{1}.
\]
In particular,
\[
a_{hh}\in (\xi_1+\xi_2)S^{2}+S^{2},\qquad
b_{hh}\in (\xi_1+\xi_2)S^{0}+S^{0},\qquad
c_{hh}\in (\xi_1+\xi_2)S^{1}+S^{1}.
\]

\end{enumerate}
\end{proposition}

As noted earlier, the bilinear forms $A$, $B$, $C$ map 
real-valued functions to real-valued functions. This property carries over to each of the terms in the expansions in \eqref{nf-hl}. 
The explicit Taylor coefficients in \eqref{nf-hl} play a significant role in our analysis, and are provided after  the proof of this proposition; see \eqref{taylor-0}  and \eqref{taylor-1}.  
The error terms in \eqref{nf-hl} can be placed in appropriate symbol classes, namely $S^{4,-1}$, $S^{3,-2}$, $S^{3,-1}$
respectively $S^{4,-2}$.

\begin{proof}

Using the definition of the formal normal form variable given in \eqref{NFT}, we expand $L_{KG} \mathbf{u}  $ as
\[
\begin{aligned}
L_{KG} \mathbf{u} 
= (\partial_t^2 - \partial_x^2 + m)\bigl(u + A(u,u) + B(u_t, u_t) + C(u_t, u)\bigr),
\end{aligned}
\]
with the goal of examining how the linear Klein--Gordon evolution acts on each component of the normal form.
Since $A,B,C$ are translation-invariant bilinear operators, derivatives distributes by Leibniz rule. 
In particular, we have
\[
\begin{aligned}
L_{KG}\Big(A(u,u)+B(u_t,u_t)+C(u_t,u)\Big)
&= 2A(L_{KG}u,u)+2A(u_t,u_t)-2A(u_x,u_x)- m A(u,u)\\
&\quad +2B(L_{KG}u_t,u_t)+2B(u_{tt},u_{tt})-2B(u_{tx},u_{tx})\\
&\quad -m B(u_t,u_t) +C(L_{KG}u_t,u)+C(u_t,L_{KG}u)\\
&\quad +2C(u_{tt},u_{t}) -2C(u_{tx},u_{x}) -m C(u_t,u).
\end{aligned}
\]

Substituting the equation for \( u\)  and extracting the quadratic part, we obtain the following partially decoupled algebraic $3\times 3$ system:
\[
\left\{
\begin{aligned}
  & Q_{11} (u,u)= 2A(u_x,u_x) -2B(u_{tt},u_{tt}) +m A(u,u) ,\\
 &  Q_{00}(u_t,u_t)  = 2B(u_{tx},u_{tx}) +mB(u_t,u_t)  -2A(u_t,u_t), \\
  &  Q_{01}(u_t,u)  = 2C(u_{tx}, u_x) -  2C(u_{xx}, u_t) +  m C(u_t,u) + 2m C(u,u_t).
\end{aligned}
\right.
\]
We now expand the term $B(u_{tt},u_{tt})$. 
Since we only need the bilinear contributions, which arise from products of linear terms, 
it suffices to retain the linear terms in $u_{tt}$:
\[
\Lambda_2(B(u_{tt},u_{tt})) = B((u_{xx} -mu) , (u_{xx} -mu)).
\]

At the level of symbols, this yields the following system for the symbols $a$ and $b$:
\[
\left\{
\begin{aligned}
q_{11}(\xi_1,\xi_2) 
&= (m-2\xi_1 \xi_2 )a(\xi_1,\xi_2) -2(\xi_1^2 + m)(\xi_2^2 +m)b(\xi_1,\xi_2), \\
  q_{00}(\xi_1,\xi_2) & = (m-2\xi_1\xi_2) b(\xi_1,\xi_2) -2a(\xi_1,\xi_2) .
\end{aligned}
\right.
\]
We proceed similarly for $C$, with the distinction that its symbol is not symmetric. Thus we obtain a system for the symbols $c(\xi_1,\xi_2)$ and $c(\xi_2,\xi_1)$, consisting of an equation and its symmetric counterpart
\[
\left\{
\begin{aligned}
     q_{01} (\xi_1,\xi_2) & = (-2\xi_1 \xi_2 + m  ) c(\xi_1,\xi_2) + 2 (\xi_2^2 + m) c(\xi_2,\xi_1), \\
        q_{01} (\xi_2,\xi_1) & = (-2\xi_1 \xi_2 + m ) c(\xi_2,\xi_1) + 2 ( \xi_1^2 + m) c (\xi_1,\xi_2) .
\end{aligned}
\right.
\]
Solving this system yields the following solution

\[
\left\{
\begin{aligned}
  a(\xi_1,\xi_2)&= \frac{(m-2\xi_1\xi_2) q_{11} + 2(\xi_1^2+m)(\xi_2^2 + m)q_{00}}{(m -2\xi_1\xi_2)^2  - 4(\xi_1^2 + m)(\xi_2^2 +m)} ,\\
  b(\xi_1,\xi_2) & =
  \frac{ 2q_{11} +(m-2\xi_1\xi_2) q_{00} }{  (m -2\xi_1\xi_2)^2  - 4(\xi_1^2 + m)(\xi_2^2 +m)},
  \\
  c(\xi_1,\xi_2) & = \frac{( m -2\xi_1 \xi_2   ) q_{01} 
    (\xi_1,\xi_2)  -  2 (\xi_2^2 + m) q_{01}  (\xi_2,\xi_1) 
    }{(m -2\xi_1\xi_2)^2  - 4(\xi_1^2 + m)(\xi_2^2 +m)}.
\end{aligned}
\right.
\]
Each of the symbols above has the same denominator; we denote the joint denominator by
\[
\Delta(\xi_1,\xi_2):=(m-2\xi_1\xi_2)^2-4(\xi_1^2+m)(\xi_2^2+m),
\]
and for convenience we rewrite as follows 
\[
\Delta(\xi_1,\xi_2):= -4m( \xi_1^1 + \xi_2^2 + \xi_1\xi_2) -3 m^2. 
\]
This form allows one to  easily observe that we are dealing with  an elliptic symbol of order two. We remark in 
particular that  $\Delta(\xi_1,\xi_2)$ is nowhere zero, which corresponds to the fact that this problem has no quadratic resonant interactions. 

\bigskip

\emph{a) The low--high symbol expansion.}
We now compute the expansions for the symbols with respect to the high frequency in the low--high case, i.e. in the region $|\xi_1|\le c|\xi_2|$ (with $c>0$ fixed sufficiently small). For this we consider $\frac{1}{\Delta}$, which admits an asymptotic expansion as follows:
\begin{align}
    \label{eq:invDelta_lh}
\frac{1}{\Delta} = - \frac{1}{4m} \xi_2^{-2} (1 - \frac{\xi_1}{\xi_2}+ O( (1 + |\xi_1|^2) \xi_2^{-2} ) ).
\end{align}
Now we successively consider the symbols $a,b,c$.

\medskip

\noindent\textbf{Expansion of $a$.}
Recalling that
\[
a(\xi_1,\xi_2)=\frac{(m-2\xi_1\xi_2)q_{11}(\xi_1,\xi_2)+2(\xi_1^2+m)(\xi_2^2+m)q_{00}(\xi_1,\xi_2)}{\Delta(\xi_1,\xi_2)},
\]
and using \eqref{eq:invDelta_lh}, we obtain the expansion
\[
a(\xi_1,\xi_2)=a_0(\xi_1)\,\xi_2+a_1(\xi_1)+O\!\left((1+|\xi_1|^{4})\xi_2^{-1}\right),
\]
where 
\begin{equation}\label{a-lh}
\left\{
\begin{aligned}
a_0(\xi_1):= & \ \frac{1}{2m} (  \xi_1 q_{11}^{(2)} - (\xi_1^2 +m) q_{00}^{(1)}),
\\
a_1(\xi_1):= & \ -\frac{1}{4m}\Bigl(m q_{11}^{(2)}(\xi_1)-2\xi_1 q_{11}^{(1)}(\xi_1)+2(\xi_1^2+m)q_{00}^{(0)}(\xi_1) \Bigr)
-\xi_1 a_0(\xi_1).
\end{aligned}
\right.
\end{equation}

\medskip

\noindent\textbf{Expansion of $b$.}
Recalling that
\[
b(\xi_1,\xi_2)=\frac{2q_{11}(\xi_1,\xi_2)+(m-2\xi_1\xi_2)q_{00}(\xi_1,\xi_2)}{\Delta(\xi_1,\xi_2)},
\]
and using \eqref{eq:invDelta_lh}, we obtain the expansion
\begin{equation*}
b(\xi_1,\xi_2)
=
b_0(\xi_1)
+ b_1(\xi_1)\xi_2^{-1}
+O\bigl((1+|\xi_1|^{3})\xi_2^{-2}\bigr),
\end{equation*}
where
\begin{equation}\label{b-lh}
\left\{
\begin{aligned}
b_0(\xi_1):= & \frac{1}{2m} (\xi_1 q_{00}^{(1)} -  q_{11}^{(2)} ) ,
\\
b_1(\xi_1):= & -\frac{1}{4m} (2q_{11}^{(1)}(\xi_1)+m q_{00}^{(1)}(\xi_1)-2\xi_1 q_{00}^{(0)}(\xi_1))
-\xi_1 b_0(\xi_1) .
\end{aligned}
\right.
\end{equation}


\medskip
\noindent\textbf{Expansion of $c$ in the low--high region.}
Recalling that
\[
c(\xi_1,\xi_2)=\frac{(m-2\xi_1\xi_2)q_{01}(\xi_1,\xi_2)-2(\xi_2^2+m)\,q_{01}(\xi_2,\xi_1)}{\Delta(\xi_1,\xi_2)},
\]
and using \eqref{eq:invDelta_lh}, we obtain the expansion
\[
c_{lh}(\xi_1,\xi_2)
= c_0^1(\xi_1) \xi_2  + c_1^1(\xi_1) + O( (1+|\xi_1|^3)\xi_2^{-1}),
\]
where
\begin{equation} \label{c-lh}
    \left\{
    \begin{aligned}
        c_0^1(\xi_1) 
        &:= 
        \frac{1}{2m} 
        ( q_{01}^{(2)}\xi_1 +  \Tilde{q}_{01}^{(1)}   ), \\
c_1^1(\xi_1) 
& :=
-\frac{1}{4m}\left( 
 m q_{01}^{(2)}(\xi_1)
-2\xi_1 q_{01}^{(1)}(\xi_1)
- 2 \Tilde{q}_{01}^{(0)}( \xi_1) 
\right)
-\xi_1 c_0^{\,1}(\xi_1).
    \end{aligned}
    \right.
\end{equation}

\medskip
\noindent\textbf{Expansion of $c$ in the high--low region.}
In the region $|\xi_2|\le c|\xi_1|$ we use the analogous expansion \eqref{eq:invDelta_lh}
 for \( c(\xi_1,\xi_2)\) and  obtain:
 \[
c_{hl}(\xi_1,\xi_2)= c_0^2(\xi_2)  + c_1^2(\xi_2) \xi_1^{-1} + O( (1+|\xi_2|^4)\xi_1^{-2}),
\]
where
\begin{equation}\label{c-hl}
\left\{ 
\begin{aligned}
     c_0^2(\xi_2) &:= 
     \frac{1}{2m} ( \Tilde{q}_{01}^{(1)}  \xi_2 + (\xi_2^2 + m) q_{01}^{(2)}), \\
c_1^2(\xi_2) &:=
-
\frac{1}{4m}\left(
m\Tilde{q}^{(1)}_{01}-2\xi_2 \Tilde{q}_{01}^{(0)}(\xi_2) -2(\xi_2^2+ m) q_{01}^{(1)}(\xi_2) 
\right)
-\xi_2 c_0^{\,2}(\xi_2).
\end{aligned}
\right.
\end{equation}

\bigskip

\emph{b) The high--high symbol properties:}

In the high--high region \( |\xi_1|\sim|\xi_2|\gg1 \), we have 
\[
|\Delta(\xi_1,\xi_2)|\sim m\,(|\xi_1|^2+|\xi_2|^2)\sim m\,\langle\xi\rangle^2.
\]
Since \( 1/\Delta\in S^{-2}  \)  on  \( \mathrm{supp}  \chi_{hh} \), it suffices to work at the numerator level, where the cancellation occurs. Accordingly, we introduce the notation
\[
\begin{aligned}
 a^{\mathrm{num}}(\xi_1,\xi_2):=a(\xi_1,\xi_2)\Delta(\xi_1,\xi_2),&\quad 
 b^{\mathrm{num}}(\xi_1,\xi_2):=b(\xi_1,\xi_2)\Delta(\xi_1,\xi_2),\\
 c^{\mathrm{num}}(\xi_1,\xi_2):=&c(\xi_1,\xi_2)\Delta(\xi_1,\xi_2).
 \end{aligned}
\]

It remains to show
\[
 a^{\mathrm{num}}\in(\xi_1+\xi_2)S^4+S^4,\qquad
 b^{\mathrm{num}}\in(\xi_1+\xi_2)S^2+S^2,\qquad
 c^{\mathrm{num}}\in(\xi_1+\xi_2)S^3+S^3.
\]
These are polynomials, so all that is needed is to factor out a $\xi_1+\xi_2$ in the highest degree term, where it suffices to simply examine the contribution of the highest degree term 
in \eqref{eq:s-poly}.


\medskip
\noindent \textbf{Factorization for \(a(\xi_1,\xi_2)\) in the high--high region.}
 Using the explicit formula,
\[
\begin{aligned}
     a^{\mathrm{num}}(\xi_1,\xi_2)
 & = a(\xi_1,\xi_2)\Delta(\xi_1,\xi_2) \\
 & = (m - 2\xi_1 \xi_2)\,q_{11}(\xi_1,\xi_2)
 +2(\xi_1^2 + m)(\xi_2^2+m)\,q_{00}(\xi_1,\xi_2).
\end{aligned}
\]
 We now examine the contribution of the highest degree term. By the definition in \eqref{eq:s-poly}, the highest degree term of \( q_{11}\) is a multiple of 
\( i(\xi_1\xi_2^2 + \xi_1^2 \xi_2) \), while the highest degree term of \(q_{00}\) is a multiple of  \( i(\xi_1 + \xi_2)\). Hence the highest degree contribution contain a factor \( i(\xi_1 + \xi_2)\)  with the remaining factors being of orders 2 and 0, respectively.
Since \((m-2\xi_1\xi_2)\in S^2\), \((\xi_1^2+m)(\xi_2^2+m)\in S^4 \), we obtain \(a^{\mathrm{num}}(\xi_1,\xi_2)\in (\xi_1 +\xi_2) S^4 + S^4\).

\medskip
\noindent \textbf{Factorization for \(b(\xi_1,\xi_2)\) in the high--high region.}
The argument is similar. 

 Using the explicit formula,
\[
\begin{aligned}
     b^{\mathrm{num}}(\xi_1,\xi_2)
 & = b(\xi_1,\xi_2)\Delta(\xi_1,\xi_2) \\
 & = 2 \,q_{11}(\xi_1,\xi_2)
 +(m- 2\xi_1\xi_2) \,q_{00}(\xi_1,\xi_2).
\end{aligned}
\]
 We now examine the contribution of the highest degree term. By the definition in \eqref{eq:s-poly}, the highest degree term of \( q_{11}\) is a multiple of 
\( i(\xi_1\xi_2^2 + \xi_1^2 \xi_2) \), while the highest degree term of \(q_{00}\) is a multiple of  \( i(\xi_1 + \xi_2)\). Hence the highest degree contribution contain a factor \( i(\xi_1 + \xi_2)\)  with the remaining factors being of orders 2 and 0, respectively.
Since \((m-2\xi_1\xi_2)\in S^2\), we obtain \(b^{\mathrm{num}}(\xi_1,\xi_2)\in (\xi_1 +\xi_2) S^2 + S^2\).

\medskip
\noindent \textbf{Factorization for \(c(\xi_1,\xi_2)\) in the high--high region.}

The argument is again analogous. Using the explicit formula,
\[
\begin{aligned}
     c^{\mathrm{num}}(\xi_1,\xi_2)
 & = c(\xi_1,\xi_2)\Delta(\xi_1,\xi_2) \\
 & = (m - 2\xi_1 \xi_2)\,q_{01}(\xi_1,\xi_2)
 -2(\xi_2^2 + m)\,q_{01}(\xi_2,\xi_1).
\end{aligned}
\]
 We now examine the contribution of the highest degree term. By the definition in \eqref{eq:s-poly}, there exists constants 
 \( \lambda_{1}\), \( \lambda_2\) such that the highest degree contribution of \( c^{\mathrm{num}}\) is
 \[
 \begin{aligned}
     & ( - 2\xi_1 \xi_2)(\lambda_1 \xi_1\xi_2 + \lambda_2 \xi_2^2 ) -2\xi_2^2 ( \lambda_1 \xi_1\xi_2 + \lambda_2 \xi_1^2) \\
     = & -2\lambda_1 \xi_1\xi_2^2(\xi_1 +\xi_2) - 2 \lambda_2 \xi_1 \xi_2^2 (\xi_1 +\xi_2).
 \end{aligned}
 \]
 From the above computations, we see that we can pull  out a factor \( \xi_1 + \xi_2\) from the highest degree contribution. Since \( \xi_1 \xi_2^2 \in S^3\), we obtain 
 \(c^{\mathrm{num}}(\xi_1,\xi_2)\in (\xi_1 +\xi_2) S^3 + S^3
 \).
\end{proof}

The purpose of the normal form transformation is to eliminate the quadratic terms in the equation.
However, the computations above show that, at leading order,
\[
a(\xi_1,\xi_2)\approx \xi_1\xi_2^2+\xi_1^2\xi_2,\qquad
b(\xi_1,\xi_2)\approx \xi_1+\xi_2,\qquad
c(\xi_1,\xi_2)\approx \xi_2^2+\xi_1\xi_2.
\]
Consequently the associated bilinear operators are
unbounded in the energy space, and one cannot expect the change of variables between $u$ and $\mathbf{u}$ to be a-priori invertible with uniform bounds.  This is precisely the point where we invoke the
modified energy method introduced by Ifrim and Tataru~\cite{HunterIfrimTataruWong2015-PAMS}.

Following this approach, we decompose the normal form operators into their low--high and high--high components, which play different roles
in our analysis of bilinear interactions.

\begin{description}
\item[(i)]\textbf{High--high interactions} can be thought of as perturbative,
and will be eliminated via a bounded normal form transformation.

\item[(ii)]\textbf{Low--high interactions} carry the main quasilinear structure and
are therefore absorbed into the modified energy functional.
\end{description}

\smallskip

\noindent  We begin with the first case \textbf{(i)}. For the high--high interactions it suffices to have the size and regularity of the normal form transformation, as described in the above proposition. However, for the low--high interactions, meaning case \textbf{(ii)}, this is no longer enough, and we need the full expressions of the coefficients in the 
expansions \eqref{nf-hl} for  $A_{lh}(u,w)$, $B_{lh}(u,w)$, \( C_{lh} (u_t,w)\), and \( C_{hl}(w_t,u)\).

We begin by computing the quadratic source terms of the full equation ~\eqref{kg1}.
More precisely, these bilinear forms can be expressed in terms of 
the derivatives at $(0,0)$ of $f$ and $g$ as follows:
\[
\left\{
\begin{aligned} 
   &  Q_{00} (u_t,u_t)
     = 2 g^{0 1 }_{u_t}  u_tu_{tx}
      + f_{u_t u_t} u_t u_t,
     \\ 
   &  Q_{01} (u_t,u)
      = g^{11}_{u_t} u_t u_{xx} 
      + 2 g^{01}_{u}  uu_{tx}
      + 2 g^{01}_{u_x}  u_x u_{tx}
      + f_{u_t u} u_t u,
  \\ 
   & 
     Q_{11} (u,u)
      =  g^{11}_u u u_{xx}
      + g^{11}_{u_x} u_x u_{xx}
      + f_{u u} u u  .
\end{aligned}
\right.
\]
This yields the symbols \(q_{00} \), \( q_{11} \), and \( q_{01} \). Note that \( q_{00} \) and \( q_{11} \) need to be symmetrized.
\begin{equation} \label{q-source}
    \left\{
\begin{aligned}
    q_{00}(\xi_1,\xi_2)
&= f_{u_tu_t}+ i\,g^{01}_{u_t}\,(\xi_1+\xi_2),
\\
q_{01}(\xi_1,\xi_2)
&= f_{u_tu}
+ 2 i\,g^{01}_{u}\,\xi_1
- 2\,g^{01}_{u_x}\,\xi_1\xi_2
- g^{11}_{u_t}\,\xi_2^{2},
\\
q_{11}(\xi_1,\xi_2)
&= f_{uu}
-\frac{g^{11}_{u}}{2}\,(\xi_1^{2}+\xi_2^{2})
-\frac{i\,g^{11}_{u_x}}{2}\,(\xi_1\xi_2^{2}+\xi_2\xi_1^{2}).
\end{aligned}
\right.
\end{equation}

 We then reorganize the symbols as follows:
 \begin{equation}\label{eq:q-coeffs}
\left\{
\begin{aligned}
q_{00}^{(0)}(\xi) 
& = f_{u_tu_t}+ i\,g^{01}_{u_t}\,\xi, 
&
\qquad 
q_{00}^{(1)}(\xi) 
&= i\,g^{01}_{u_t},\\
q_{11}^{(2)}(\xi) 
&= -\frac12 g^{11}_{u}-\frac{i}{2}g^{11}_{u_x}\,\xi, 
&\qquad q_{11}^{(1)}(\xi) 
&= -\frac{i}{2}g^{11}_{u_x}\,\xi^2,\\
q_{01}^{(2)}(\xi) 
&= -g^{11}_{u_t},
&\qquad q_{01}^{(1)}(\xi) 
&= -2g^{01}_{u_x}\,\xi,\\
\tilde q_{01}^{(1)}(\xi) 
&= 2i g^{01}_{u}-2g^{01}_{u_x}\,\xi,
&\qquad \tilde q_{01}^{(0)}(\xi)
&= f_{u_tu}-g^{11}_{u_t}\,\xi^2.
\end{aligned}
\right.
\end{equation}

Using the expansions \eqref{a-lh}–\eqref{c-hl} together with the coefficients in \eqref{eq:q-coeffs}, we obtain:

\begin{equation} \label{taylor-0}
    \left\{
\begin{aligned}
a_0(\xi) &= -\frac{i}{2}\,g^{01}_{u_t}
-\frac{g^{11}_{u}}{4m}\,\xi
-\frac{i}{2m}\Bigl(g^{01}_{u_t}+\frac12 g^{11}_{u_x}\Bigr)\xi^2,
\\
b_0(\xi) & = \frac{g^{11}_{u}}{4m}
+\frac{i\xi}{2m}\Bigl(g^{01}_{u_t}+\frac12 g^{11}_{u_x}\Bigr),
\\
c_0^1(\xi) & = \frac{i}{m}\,g^{01}_{u}
-\frac{g^{11}_{u_t}+2g^{01}_{u_x}}{2m}\,\xi,
\\
c_0^2 (\xi) & =-\frac{g^{11}_{u_t}}{2}
+\frac{i}{m}\,g^{01}_{u}\,\xi
-\frac{2g^{01}_{u_x}+g^{11}_{u_t}}{2m}\,\xi^2.
\end{aligned}
\right.
\end{equation}

We note that we have the following identities:

\begin{equation} \label{lead-symbol}
    \left\{
    \begin{aligned}
       &  c_0^1(\xi) \xi -c_0^2(\xi)  = \frac{1}{2}  g^{11}_{u_t},
\\
& a_0(\xi) \xi + b_0(\xi) (\xi^2 +m)  = \frac{1}{4} g_u^{11} + \frac{i}{4} g_{u_x}^{11} \xi .
    \end{aligned}
    \right.
\end{equation}

\smallskip

Now we compute \( a_1\), \( b_1\), \( c^{\, 1}_1\) and \( c_1^{\, 2}\). Using the symbol coefficient~\eqref{eq:q-coeffs}, we obtain:
\begin{equation}\label{taylor-1}
\left\{
\begin{aligned}
a_1(\xi) 
& = 
\frac{g^{11}_{u}}{8} - \frac{f_{u_tu_t}}{2}
+\frac{i}{8}\,g^{11}_{u_x}\,\xi
+\frac{\xi^2}{4m}\Bigl(g^{11}_{u}-2f_{u_tu_t}\Bigr),
\\
b_1(\xi) 
& =  -\frac{i}{4}\,g^{01}_{u_t}
+\frac{\xi}{4m}\Bigl(2f_{u_tu_t}-g^{11}_{u}\Bigr) ,
\\
c_1^{\,1}(\xi)
& = \frac{g^{11}_{u_t}}{4}+\frac{f_{u_tu}}{2m}
-\frac{i}{m}\,g^{01}_{u}\,\xi ,
\\
c_1^{\,2}(\xi)
& = -\frac{i}{2}g_u^{01}
+\frac12\Big(g_{u_t}^{11}-g_{u_x}^{01}+\frac{f_{u_tu}}{m}\Big)\xi
-\frac{i}{m}g_u^{01}\xi^2.
\end{aligned}
\right.    
\end{equation}

We note that the parity properties of \( a_0\), \( a_1\), \( b_0\), \( b_1\), \( c_0^1\), \( c_1^1\), \( c_0^2\), and \( c_1^2\) are inherited from those of \( a\), \(b\), and \(c\). 
In particular, the imaginary parts of
$a,b,$ and $c$ are even, while their real parts are odd.
Consequently, our expansions imply that \(a_0\) , \(b_1 \) , \( c_0^1 \), and \( c_1^2\) are imaginary even and real odd, whereas \( a_1 \), \(b_0\) , \(c_0^2\) , and \(c_1^1\) are imaginary odd and real even.

\section{The paradifferential expansion of the equation}
\label{s:para}

Our goal here is to develop the 
paradifferential expansion for the 
quasilinear Klein-Gordon model \eqref{kg1}, which, for convenience, we recall below:
\begin{equation}\label{eq-full}
- g^{\alpha \beta}(u,\partial u) \partial_\alpha
\partial_\beta u + m u  =f(u,\partial u).
\end{equation} 

To analyze the quasilinear Klein--Gordon equation, it is convenient to also consider the associated linearized equation around a given solution $u$.
Indeed, the infinitesimal difference of two nearby nonlinear solutions satisfies the linearized equation, which underlies the stability and uniqueness estimates. 
Moreover, the paradifferential expansions developed below will be carried out at the level of the corresponding linearized operator.
 Further, the energy estimates of the linearized equation are useful for stability estimates on the cubic time-scale.

We  begin with the associated linearized equation
\[
\begin{aligned}
    - g^{\alpha \beta}(u,\partial u ) \partial_\alpha
\partial_\beta v - ( g^{\alpha \beta}_u  v + g_{p_{\gamma}}^{\alpha \beta} \partial_{\gamma}v) \partial_{\alpha} \partial_{\beta} u  + m v = f_u v + f_{p_{\gamma}} \partial_{\gamma}v.
\end{aligned}
\]
It can be further rewritten in  a more convenient form, namely
\begin{equation} \label{eq-lin}
    - g^{\alpha \beta}(u,\partial u ) \partial_\alpha
\partial_\beta v  + m v
=
F^{\gamma,lin} \partial_\gamma v + F^{lin} v,
\end{equation}
where the coefficients on the right 
are given by
\[
\left\{
\begin{aligned}
& F^{\gamma,lin} :=  \ g^{\alpha \beta}_{p_{\gamma}}(u,\partial u) \partial_\alpha \partial_\beta u + f_{p_{\gamma}}(u,\partial u),
\\
& F^{lin} :=  \  g_{u}^{\alpha\beta}(u,\partial u) \partial_\alpha \partial_\beta  u + f_{u}(u,\partial u).
\end{aligned}
\right.
\]


At a more naive level, one may view both of these evolutions as nonlinear perturbations of the linear Klein-Gordon flow, and write them in the form
\begin{equation}
L_{KG} u = N(u),
\end{equation}
and, similarly, for the linearized equation
\begin{equation}
L_{KG} v = N^{lin}(u)v ,
\end{equation}
where the nonlinearity can be further 
split into a quadratic term and  cubic and higher terms,
\[
N(u) = N^{[2]}(u)+ N^{[3]}(u),
\]
and similarly for the linearized equation
\[
N^{lin}(u)v = N^{lin,[2]}(u)v+ N^{lin,[3]}(u)v.
\]
For our purposes, it will be useful to describe the quadratic 
components more explicitly; we will decompose  them in the form
\begin{equation}\label{N2}
N^{[2]}(u)= Q_{11}(u,u) + Q_{00}(u_t,u_t) + Q_{01}(u_t,u),
\end{equation}
and, correspondingly, we will do the same for $N^{lin,[2]}(u)v$: 
\[
N^{lin,[2]}(u)v = 2 Q_{11}(u,v) + 2Q_{00}(u_t,v_t) + Q_{01}(u_t,v) + Q_{01}(u,v_t).
\]

 However, a better strategy is to 
associate to both of these evolutions a corresponding linear 
paradifferential flow, capturing the principal quasilinear character, which has the form
\begin{equation}\label{para}
L_{KG}^{para} w = f,
\end{equation}
where $L_{KG}^{para}$ appears from rewriting  the linearized equation paradifferentially and extracting the main low--high parts.

Therefore,  we are now able to introduce the paradifferential operator  $L_{KG}^{para}$ as follows
\[
L^{para}_{KG} := -T_{g^{\alpha \beta}} \partial_\alpha \partial_\beta   - T_{F^{\gamma,lin}} \partial_{\gamma}  - T_{F^{lin} }  + m.
\]
For convenience, we use the shorthand notation
\[
T_{\Tilde{F}^{\gamma,lin}} \Tilde{\partial_{\gamma}}   := T_{F^{\gamma,lin}} \partial_{\gamma}  + T_{F^{lin} } ,
\] 
where the $T_ab$ notation is the standard paradifferential product. 
Consequently, $L^{para}_{KG} $  will then be expressed as 
\[
L^{para}_{KG} := -T_{g^{\alpha \beta}} \partial_\alpha \partial_\beta - T_{\Tilde{F}^{\gamma,lin}} \Tilde{\partial_{\gamma}}  + m .
\]

With this notation, the equation~\eqref{eq-full} can be reexpressed in paradifferential form as
\begin{equation}\label{para-KG}
  L^{\mathrm{para}}_{KG} u = N_{\mathrm{bal}}(u),
\end{equation}
where \( N_{bal}\) collects the balanced terms after subtracting the principal low--high part:
\[
N_{bal}(u) =: T_{\partial_\alpha \partial_\beta u  } g^{\alpha \beta}(u,\partial u) + \Pi( g^{\alpha \beta}(u,\partial u),\partial_\alpha \partial_\beta u ) + f(u,\partial u) -  T_{\Tilde{F}^{\gamma,lin}} \Tilde{\partial_{\gamma}} u.
\]

We can further separate the nonlinearity into a quadratic part and a cubic and higher order part, 
\begin{equation}\label{Nbal}
N_{bal}(u) = N_{bal}^{[2]}(u) + N_{bal}
^{[3]}(u),
\end{equation}
where the quadratic part plays the leading role and will be treated explicitly,
while the cubic and higher terms are perturbative. Precisely, the quadratic component can be described in terms of the expression in \eqref{N2} as 
\[
N_{bal}^{[2]}(u) := 
  Q^{hh}_{00}(u_t,u_t) 
+ Q^{hh}_{01}(u_t,u)
+ Q_{11}^{hh}(u,u).
\]

In a similar fashion, the linearized equation can be written as
\begin{equation}\label{para-KG-lin}
  L^{\mathrm{para}}_{KG} v = N^{lin}_{\mathrm{bal}}(u) v,
\end{equation}
where,
\[
\begin{aligned}
    N^{lin}_{\mathrm{bal}}(u) v
    = T_{ \partial_\alpha \partial_\beta v} g^{\alpha \beta }(u,\partial u ) 
    + \Pi( \partial_\alpha \partial_\beta v, g^{\alpha \beta }(u,\partial u ) )
    + T_{ \Tilde{\partial}_{\gamma} v} \Tilde{F}^{\gamma,lin}
    + \Pi(\Tilde{\partial}_{\gamma} v,\Tilde{F}^{\gamma,lin}  ).
\end{aligned}
\]
 We again separate the nonlinearity into a quadratic part and a cubic and higher order part, 
\[
N^{lin}_{bal}(u)v := N_{bal}^{lin,[2]}(u)v + N_{bal}
^{lin,[3]}(u)v.
\]
In contrast to the full equation, the quadratic part in the linearized equation also contains unbalanced terms, namely
\[
\begin{aligned}
    N_{bal}^{lin,[2]}(u)v  = & \; 2 Q^{hh}_{11}(u,v)
+ 2Q^{hh}_{00}(u_t,v_t) + Q^{hh}_{10}(u,v_t) 
+  Q^{hh}_{01}(u_t,v) 
\\
&
+  2Q^{hl}_{11}(u,v) + Q^{hl}_{10}(u,v_t)
+ Q^{hl}_{01}(u_t,v) +2Q^{hl}_{00}(u_t,v_t).
\end{aligned}
\]

We first record some preliminary estimates: \(L^{\infty}\) bounds for \(u_{tt}\)
 and \( H^{-1}\) estimates for \( v_{tt}\). These estimates arise when estimating the time derivative of the source terms in the analysis of the linearized equation.

 \begin{proposition}\label{p:infty}
     The following estimates hold for \( u_{tt}\):
\begin{equation} \label{eq:p2-utt}
\| \partial  u_{tt} \|_{W^{k-1,\infty}} \lesssim_{\A_0} \A_{k+1}  ,  
\end{equation}

\begin{equation} \label{eq:p2-utt-2}
\| \Lambda_{\ge 2}( \partial u_{tt}) \|_{W^{k-1,\infty}} \lesssim_{\A_0} \A_0 \A_{k+1} .   
\end{equation}
There also exists an expansion for \( v_{tt} = v_{tt}^{(-1)} + v_{tt}^{(0)}  \) such that

\begin{equation} \label{eq:vtt}
\begin{aligned}
     \| v_{tt}^{(-1)} \|_{H^{-1}} 
   &\lesssim_{\A_0}   \|v[t]\|_{H^1 \times L^2} 
   \\
   \| v^{(0)}_{tt} \|_{L^2}
    &\lesssim_{\A_0}  \A_1 \|v[t]\|_{H^1 \times L^2} .
\end{aligned}
\end{equation}
Similarly, there exists an expansion for \( v_{ttt} = v_{ttt}^{(-2)} + v_{ttt}^{(-1)} + v_{ttt}^{(0)}   \) such that 
\begin{equation} \label{eq:vttt}
\begin{aligned}
     \| \Lambda_{\geq 2 }v_{ttt}^{(-2)} \|_{H^{-2}}
   & 
   \lesssim_{\A_0}  \A_0 \|v[t]\|_{H^1 \times L^2} ,
   \\
   \|\Lambda_{\geq 2 } v^{(-1)}_{ttt} \|_{H^{-1}}
   &  \lesssim_{\A_0}  \A_1 \|v[t]\|_{H^1 \times L^2} ,
    \\
   \| \Lambda_{\geq 2 }v^{(0)}_{ttt} \|_{L^2}
   &  \lesssim_{\A_0}  \A_2 \|v[t]\|_{H^1 \times L^2}
   .
\end{aligned}
\end{equation}

 \end{proposition}
\begin{proof}
Using the equation for \( u_{tt}\), we write
\begin{equation} \label{eq:uttt}
\begin{aligned}
    \partial^k u_{tt} &= \ \partial^k (2 g^{01} \partial_x u_t + g^{11} \partial_x^2 u + f  -mu )
    \\
    & 
= \ \sum_{k_1 + k_2 = k} C_{k_1,k_2} 
\big( 
2  \partial^{k_1} g^{01}  \partial^{k_2} \partial_x u_t
+ \partial^{k_1} g^{11} \partial^{k_2}\partial_x^{2} u
\big)
+ \partial^k f  - m\partial^k u,
\end{aligned}
\end{equation}
where \(C_{k_1,k_2}  \) are the corresponding binomial coefficients.
 Whenever \(u_{tt}\) appears, we use the equation for \(u\) to replace \( u_{tt}\) by \(
 2g^{01} \partial_x u_t +
 g^{11}  \partial_x^2 u + f  -mu \). Since the equation is quasilinear, using the equation will preserve the highest number of derivatives applied to a factor. Hence, the worst contribution comes from the most unbalanced term, which has all derivatives fall on the highest-frequency factor. Such terms involve expressions of the form \( \partial_x^{k+1} \partial u\), which requires the control parameter \(  \A_{k+1}\),
 with a coefficient depending on $u$ and $\partial u$, which generates the implicit dependence on $\A_0$.

Then it remains to consider the intermediate terms, i.e., the terms that have more balanced derivatives, for which one can use  interpolation. 
Recalling the definition of the control parameters
\[
\begin{aligned}
    \A_{j }  
    & = \ \| u \|_{L^{\infty}} + \| \partial  u \|_{W^{j,\infty}}  ,
\end{aligned}
\]
the one-dimensional Gagliardo–Nirenberg inequality implies that for \( 0 \le j \le k+1\),
\[
\begin{aligned}
    \A_j \lesssim \
     \A_0 + \A_0^{1 - \frac{j}{k+1}} \A_{k+1}^{\frac{j}{k+1}} . 
\end{aligned}
\]
Consequently, for any \( n \ge 2\) and any integers \( \alpha_i \geq 0 \) satisfying \(\alpha_1 + \cdots \alpha_n = k+1 \), we obtain the multilinear bound
\begin{equation} \label{eq:n-inter}
\begin{aligned}
    \sum_{\alpha_1 + \cdots \alpha_n = k+1} \prod_{ i =1 }^{n} \A_{\alpha_i}  
    \lesssim \
    \A_0^{n-1}  \A_{k+1}.
\end{aligned}    
\end{equation}
In particular, all terms in \eqref{eq:uttt} can be estimated using \eqref{eq:n-inter}.

For the second estimates we additionally apply Moser-type estimates, which require control parameters \( \A_0 \A_{k+1} \). This concludes the proof for the first and the second estimate.

Next we estimate \( v_{tt}\). For these estimates, the paradifferential expansion is not particularly helpful, so we work directly with the original linearized equation~\eqref{eq-lin}. 
Noting that \( g^{00} = -1\), we obtain
\[
v_{tt} = \ 2 g^{0 1}(u,\partial u )
\partial_x v_t +    g^{1 1}(u,\partial u )
\partial_x^2 v
+
\Tilde{F}^{\gamma,lin} \Tilde{\partial}_\gamma v -mv .
\]
We then define
\[
\begin{aligned}
    v_{tt}^{(-1)} 
    &:= \;  
    \partial_x( 2 g^{01}(u,\partial u)  v_t + g^{11}(u,\partial u)  \partial_x v ) -mv,
    \\
v_{tt}^{(0)} & := \;  -\partial_x( 2 g^{01}(u,\partial u)  )  v_t 
 -\partial_x(  g^{11}(u,\partial u)  )  v_x   
+\Tilde{F}^{\gamma,lin} \Tilde{\partial}_\gamma v .
\end{aligned}
\]
Estimating the $u$ dependent coefficients in $L^\infty$,  we obtain a bound for \( v_{tt}^{(0)}\) 
in \(H^{-1}\) 
\[
\begin{aligned}
    \|  v_{tt}^{(-1)} \|_{H^{-1}}
    & \lesssim 
    \| \partial_x( 2 g^{01}(u,\partial u)  v_t + g^{11}(u,\partial u)  \partial_x v )  \|_{H^{-1}}  + \| mv \|_{L^2}
\\
    &
    \lesssim 
     \|   g^{\alpha 1}(u,\partial u ) \partial_\alpha v   \|_{L^2} 
 + \| mv \|_{L^2}
\\
&
\lesssim_{\A_0} \| v[t] \|_{H^1 \times L^2} .
\end{aligned}
\]
Similarly, we can control \( v_{tt}^{(0)}  \) in \( L^2\)
\[
\begin{aligned}
    \|   v_{tt}^{(0)} \|_{L^2}
    & 
    \lesssim
    \| \partial_x(g^{\alpha1}(u,\partial u)  ) \partial_\alpha v \|_{L^2}
    + \| \Tilde{F}^{\gamma,lin} \Tilde{\partial}_\gamma v  \|_{L^2} 
    \\
    &
    \lesssim_{\A_0} \A_1 \| v[t] \|_{H^1 \times L^2}.
\end{aligned}
\]

\smallskip

Finally, we expand \(v_{ttt}\) as follows
\[
\begin{aligned}
    v_{ttt} & =\; 
    2 g^{0 1}(u,\partial u )  \partial_x v_{tt} 
    +g^{1 1}(u,\partial u )   \partial_x^2 v_t
    + \partial_t ( 2g^{01}(u,\partial u )  ) \partial_x v_t
\\
& \quad \ 
+   \partial_t ( g^{11}(u,\partial u )  ) \partial_x^2 v
+  \partial_t (\Tilde{F}^{\gamma,lin})  \Tilde{\partial}_\gamma v + \Tilde{F}^{\gamma,lin} \Tilde{\partial}_\gamma \partial_t v - mv_t
\\
&
= \;  
 2 g^{0 1}(u,\partial u ) \partial_x \bigr( 
2 g^{0 1}(u,\partial u ) \partial_x v_t  
+g^{1 1}(u,\partial u ) \partial_x^2 v
+
\Tilde{F}^{\gamma,lin} \Tilde{\partial}_\gamma v -mv\bigl)  
\\
& \quad \ 
+ g^{11}(u,\partial u ) \partial_x^2 v_t 
+ \partial_t ( 2 g^{0 1}(u,\partial u )  )  \partial_x v_t
+ \partial_t (g^{1 1}(u,\partial u )  )  \partial_x^2 v 
\\
&  \quad \
+  \partial_t (\Tilde{F}^{\gamma,lin})  \Tilde{\partial}_\gamma v 
+ F^{1,lin} \partial_x v_t + F^{lin} v_t
\\
&
\quad \ +  F^{0,lin} \bigr( 2 g^{0 1}(u,\partial u ) 
\partial_x v_t
+
 g^{11}(u,\partial u ) 
\partial_x^2 v
+
\Tilde{F}^{\gamma,lin} \Tilde{\partial}_\gamma v -mv \bigl) - mv_t.
\end{aligned}
\]

We first regroup the terms by the number of derivatives falling onto \(v\)
\[
\begin{aligned}
    v_{ttt} = \mathfrak{c}_1^{\alpha}(u,\partial u) \partial_x^2 \partial_\alpha v 
    + 
    \mathfrak{c}_2^{\alpha}(u,\partial^{\le 2} u ) \partial_\alpha \partial_x v 
    + 
    \mathfrak{c}_3^{\alpha}(u, \partial^{\le 3} u) \partial_\alpha v + \mathfrak{c}_4(u, \partial^{\le 3} u)  v
    ,
\end{aligned}
\]
where the coefficient functions
$\mathfrak c^{\alpha}_1,\mathfrak c^{\alpha}_2,\mathfrak c^{\alpha}_3, \mathfrak c_4$ are used only in this
proof and depend smoothly on their arguments.
Up to nonlinear coefficients depending on $u$ and $\partial u$, each term in the expansion of $v_{ttt}$
contains at most four derivatives in total, counting derivatives of both \(u\) and \(v\). Since our decomposition---splitting
terms according to the number of derivatives falling onto $v$ and rewriting in
divergence form---does not alter the total derivative count, the coefficients
must have the claimed structure. They can be explicitly computed, however, for our purposes it suffices to record the bound
\[
\| \Lambda_{\ge 1} \mathfrak{c}_i\|_{L^{\infty}} \lesssim_{\A_0} \A_{i-1},\quad i=\overline{1,4}.
\]

We then rewrite the terms in divergence form.
More precisely, we write
\begin{equation}
v_{ttt}
= v_{ttt}^{(-2)} + v_{ttt}^{(-1)} + v_{ttt}^{(0)},
\end{equation}
where the three components are determined successively starting from the highest--order
\begin{equation}
\begin{aligned}
v_{ttt}^{(-2)}
&:= \partial_x^2( \mathfrak{c}_1^{\alpha} (u,\partial u) \partial_{\alpha} v )  ,\\[2mm]
v_{ttt}^{(-1)}
&:= \partial_x( \mathfrak{c}_2^{\alpha} (u,\partial^{\le 2}u )  \partial_\alpha v - 2 \partial_x \mathfrak{c}_1^{\alpha} (u,\partial u) \partial_{\alpha} v  ),\\[2mm]
v_{ttt}^{(0)}
&:=  \mathfrak{c}_3^{\alpha}(u, \partial^{\le 3} u) \partial_\alpha v + \mathfrak{c}_4(u, \partial^{\le 3} u)  v - \partial_x^2( \mathfrak{c}_1^{\alpha} (u,\partial u) )\partial_{\alpha} v \\
& \quad \ -  \partial_x( \mathfrak{c}_2^{\alpha} (u,\partial^{\le 2}u ) - 2 \partial_x^2 \mathfrak{c}_1^{\alpha} (u,\partial u) )\partial_{\alpha} v     .
\end{aligned}
\end{equation}

We estimate each component at the corresponding regularity level. For the highest-order part, we have

\[
\begin{aligned}
    \| \Lambda_{\ge 2} v_{ttt}^{(-2)}  \|_{H^{-2}}
    &
    \lesssim 
    \| \partial_x^2( \Lambda_{\ge 1}\mathfrak{c}_1^{\alpha} (u,\partial u) \partial_\alpha v )  \|_{H^{-2}} 
\\
&
\lesssim \| \Lambda_{\ge 1}( \mathfrak{c}_1^{\alpha} (u,\partial u)) \partial_\alpha v \|_{L^2} 
\\
&
\lesssim_{\A_0}  \A_0 \| v[t] \|_{H^1\times L^2} .
\end{aligned}
\]
The same argument yields
\[
\begin{aligned}
    \| \Lambda_{\ge 2} v_{ttt}^{(-1)} \|_{H^{-1}}
    & \lesssim 
     \
    \|   \partial_x( \Lambda_{\ge 1 }\mathfrak{c}_2^{\alpha} (u,\partial^{\le 2}u )  \partial_\alpha v - 2\Lambda_{\ge 1 } \partial_x \mathfrak{c}_1^{\alpha} (u,\partial u) \partial_{\alpha} v  ) \|_{H^{-1}} 
    \\
    & \lesssim
     \
    \| \Lambda_{\ge 1 }\mathfrak{c}_2^{\alpha} (u,\partial^{\le 2}u )  \partial_\alpha v - 2\Lambda_{\ge 1 } \partial_x \mathfrak{c}_1^{\alpha} (u,\partial u) \partial_{\alpha} v   \|_{L^2} 
    \\
    & \lesssim_{\A_0}
     \
    \A_1 \| v[t] \|_{H^1\times L^2} .
\end{aligned}
\]
The remaining term $v_{ttt}^{(0)}$ is estimated in the same way, and we conclude the desired bound.
This completes the proof.
\end{proof}

We also record a preliminary \( L^2\) estimate for the quadratic part of the nonlinearity in the linearized equation.
\begin{proposition} \label{p:lin2}
    The quadratic part of the nonlinearity \( N^{lin}_{bal}(u)v \) 
    can be estimated at fixed time as
    \begin{equation} \label{eq:lin2}
         \| N^{lin,[2]}_{bal}(u)v \|_{L^2} 
    \lesssim  \A_1 \| v[t] \|_{H^1 \times L^2},
    \end{equation}
 and
    \begin{equation} \label{eq:lin2-t}
         \| \partial_t N^{lin,[2]}_{bal}(u)v \|_{H^{-1}} 
    \lesssim_{\A_0} \A_1 \| v[t] \|_{H^1 \times L^2}.
    \end{equation}
\end{proposition}
\begin{proof}
   We recall that the quadratic contribution is given by
 \[
\begin{aligned}
    N_{bal}^{lin,[2]}(u)v  = & 2 Q^{hh}_{11}(u,v)
+ 2Q^{hh}_{00}(u_t,v_t) + Q^{hh}_{10}(u,v_t) 
+  Q^{hh}_{01}(u_t,v) 
\\
&
+  2Q^{hl}_{11}(u,v) + Q^{hl}_{10}(u,v_t)
+ Q^{hl}_{01}(u_t,v) +2Q^{hl}_{00}(u_t,v_t).
\end{aligned}
\]
It is clear that each term involves exactly three derivatives. Moreover, for both the high–high and high–low interactions we may arrange that at most one derivative falls onto  \(v\). Therefore, by standard paraproduct estimates~\eqref{CM}, we obtain \eqref{eq:lin2}.

To prove \eqref{eq:lin2-t}, it suffices to estimate the time derivative of the most delicate term, namely \(Q_{00}(u_t,v_t) \). 
We consider the high--high and high--low interactions separately. More precisely,
\[
\begin{aligned}
    \| \partial_t  Q_{00}(u_t,v_t) \|_{H^{-1}} 
    &
    \lesssim 
    \|   Q_{00}^{hh}(u_{tt},v_t) \|_{H^{-1}} + \|   Q^{hh}_{00}(u_t,v_{tt}) \|_{H^{-1}} 
    \\
    & \qquad  +  \|   Q_{00}^{hl}(u_{tt},v_t) \|_{H^{-1}} + \|   Q^{hl}_{00}(u_t,v_{tt}) \|_{H^{-1}}  .
\end{aligned}
\]
We first estimate the terms involving $u_{tt}$ using \eqref{eq:p2-utt}. Since $q^{00}\in S^{1}$, we obtain

\[
\begin{aligned}
     \|   Q_{00}^{hh}(u_{tt},v_t) \|_{H^{-1}} 
     + \|   Q_{00}^{hl}(u_{tt},v_t) \|_{H^{-1}} 
     \lesssim 
     \| u_{tt} \|_{L^{\infty}} \|v_t \|_{L^2}
     \lesssim \A_1 \| v[t] \|_{H^1\times L^2}.
\end{aligned}
\]
We then use the expansion for \( v_{tt}\) in \eqref{eq:vtt} in order to estimate the second term. 
\[
\begin{aligned}
   \|   Q^{hh}_{00}(u_{t},v_{tt}) \|_{H^{-1}} 
     \lesssim 
     \| u_t \|_{W^{1,\infty}} \| v_{tt}^{(-1)} \|_{H^{-1}} 
     + \| u_t \|_{L^{\infty}} \| v_{tt}^{(0)} \|_{L^2}
     \lesssim \A_1 \| v[t] \|_{H^1 \times L^2}.
\end{aligned}
\]
The high--low contributions are treated in the same way.

\smallskip

The remaining terms  can be bounded in the same manner.

\end{proof}

For the purposes of this article, it is sufficient to treat perturbatively the cubic nonlinearities arising in the paradifferential expansions of both the full and linearized equations:

\begin{proposition}\label{p:full3}
Let $s \geq 1$. Then the cubic part of the nonlinearity $N_{bal}$ in the equation 
\eqref{para-KG} can be estimated at fixed time as
\begin{equation}
\| N_{bal}^{[3]}(u)\|_{H^{s-1}}
\lesssim_{\A_0} 
\A_0
\A_1
\| u[t]\|_{H^{s} \times H^{s-1}}.
\end{equation}
\end{proposition}

A similar result can be proved in the case of the linearized equation:
\begin{proposition}\label{p:lin3}
The cubic part of the nonlinearity $N_{bal}^{lin}$ in the equation \eqref{para-KG-lin} can be estimated at fixed time as
\begin{equation} \label{eq:n-lin-3}
\| N_{bal}^{lin,[3]}(u)v\|_{L^2}
\lesssim_{\A_0} 
\A_0
\A_1
\| v[t]\|_{H^{1} \times L^2}.
\end{equation}

The time derivative of the cubic component of the nonlinearity  \(N_{bal}^{lin} (u)v\)  can be estimated at a fixed time by  
\begin{equation} \label{eq:n-lin-3-t}
    \| \partial_t ( N_{bal}^{lin,[3]}(u) v )\|_{H^{-1}} 
    \lesssim_{\A_0} 
\A_0
\A_2
\| v[t]\|_{H^{1} \times L^2}.
\end{equation}

\end{proposition}

The two propositions above contain all the information that will be used later in the paper and that provide bounds for these cubic nonlinearities. The remainder of this section is devoted to their proof. For brevity we omit the dependence of all implicit constants on $\A_0$ in what follows.

\begin{proof}[Proof of Proposition~\ref{p:full3}]

We estimate each term in \(N_{\mathrm{bal}}^{[3]}(u)\) by first performing a Taylor expansion and then applying Moser-type estimates \eqref{bmo-hs}. We begin by expanding the cubic balanced terms
\[
\begin{aligned}
    N_{bal}^{[3]}(u) &= \Lambda_{\geq 3}( N_{bal}(u))\\
    &=  T_{\partial_\alpha \partial_\beta u  } \Lambda_{\geq 2} (g^{\alpha \beta}(u,\partial u)  )
+ \Pi(\Lambda_{\geq 2} ( g^{\alpha \beta}(u,\partial u)),\partial_\alpha \partial_\beta u )  \\
& \quad +\Lambda_{\geq 3 }( f(u,\partial u) )-  T_{ \Lambda_{\geq 2}(\Tilde{F}^{\gamma,lin} )} \Tilde{\partial_{\gamma}} u.
\end{aligned}
\]
Let \( U:=(u,\partial u)\) . Since \( g \) is smooth, there exist coefficient
functions \( g^I(U)\) , smooth and bounded near \( U=0\) , such that, by Taylor expansion and homogeneity cutoffs, for multi-indices \( I = (i_1,i_2,i_3)\), with  \( |I|: = i_1 +i_2 + i_3\),
we have
\[
\Lambda_{\geq 2} (g^{\alpha \beta}(u,\partial u ) ) 
=\sum_{|I| = 2 }  g^I(U) U^I.
\]
Therefore
\[
 T_{\partial_\alpha \partial_\beta u } \Lambda_{\geq2} ( g^{\alpha \beta}(u,\partial u))
 = \sum_{|I| = 2} T_{\partial_\alpha \partial_\beta u } g^{I} (U) U^{I}.
\]
First, we apply the paraproduct estimates~\eqref{CM}, and then we apply the Moser estimates~\eqref{bmo-hs}.
For the balanced paraproduct \( \Pi(\Lambda_{\ge 2}(g^{\alpha\beta}),\partial_\alpha\partial_\beta u )\), the same type of
estimates yields
\[
\begin{aligned}
 \Vert  \Pi( \Lambda_{\geq2} (g(u,\partial u)), \partial_\alpha \partial_\beta u ) \Vert_{H^{s-1}} 
 &\lesssim 
 \sum_{|I| =2}
 \| \partial^2 u \|_{L^{\infty}} 
 \|  g^I (U) U^I  \|_{H^{s-1} }  \\
& \lesssim
 \| \partial^{\leq 1} u \|_{L^{\infty}} 
  \| \partial^2 u  \|_{L^{\infty}} 
 \|  u[t]   \|_{H^{s} \times H^{s-1} }   .
 \end{aligned}
\]
Since \(f \) is smooth, there exists smooth coefficient functions \(f^{I}\) uniformly bounded near \( U=0\) for all  multi-indices \( |I| = 3\), such that :
\[
\Lambda_{\geq 3} (f) = \sum_{|I| =3} f^{I} (U)   U^{I}.
\]

Since $\Lambda_{\ge3}(f)$ is cubic in $U$,
repeated applications of Moser-type estimates give
\[
\bigl\|f^{I}(U)U^{I} \bigr\|_{H^{s-1}}
 \lesssim \|U\|_{L^\infty}^2\|U\|_{H^{s-1}}
 \lesssim \|\partial^{\le1}u\|_{L^\infty}^2
          \|u[t]\|_{H^{s}\times H^{s-1}}.
\]
Finally, recall that 
\[
\tilde F^{\gamma,lin}=g^{\alpha\beta}_{p_\gamma}(U)\partial_\alpha\partial_\beta u + f_{p_\gamma}(U).
\]
In particular, $\Lambda_{\ge2}(\tilde F^{\gamma,lin})$ is at least quadratic in $U$ and contains at most one factor of $\partial^2u$. Hence, using the paraproduct estimates~\eqref{CM} and the Moser estimates~\eqref{bmo-hs} once again,
\[
\begin{aligned}
    \|  T_{ \Lambda_{\geq2} ( \Tilde{F}^{\gamma,lin})} \Tilde{\partial}_{\gamma} u \|_{H^{s-1}}
    \lesssim
    & \ 
    \|   T_{  g^{\alpha \beta}_{p_{\gamma}}(U) \partial_\alpha \partial_\beta u  }\partial_{\gamma} u \|_{H^{s-1}} 
    + \|  T_{ F_{p_{\gamma}}^{I}  (U) U^{I} }\partial_{\gamma} u \|_{H^{s-1}} \\
    \lesssim
    &\  \|  g^{\alpha \beta}_{p_{\gamma}}(U) \partial_\alpha \partial_\beta u  \|_{L^{\infty}} 
    \| u[t] \|_{H^{s}\times H^{s-1}} 
    +  \|F_{\gamma}^{I}  (U) U^{I}\|_{L^{\infty}} \| u[t] \|_{H^{s}\times H^{s-1}} 
    \\
    \lesssim
    & \ 
    \A_0
    \A_1
    \| u [t] \|_{H^{s}\times H^{s-1}} .
\end{aligned}
\]
Combining the above bounds, we obtain the desired estimate for
$\|N^{[3]}_{bal}(u)\|_{H^s}$.

\end{proof}

\smallskip

\begin{proof}[Proof of Proposition~\ref{p:lin3}]
 \
 We argue as in the proof of Proposition~\ref{p:full3}. Recall that
\[
\begin{aligned}
  N^{lin}_{bal}(u)v
  &
=  T_{ \partial_\alpha \partial_\beta v} g^{\alpha \beta }(u,\partial u ) 
    + \Pi( \partial_\alpha \partial_\beta v, g^{\alpha \beta }(u,\partial u ) )
    + T_{ \Tilde{\partial}_{\gamma} v} \Tilde{F}^{\gamma,lin}
    + \Pi(\Tilde{\partial}_{\gamma} v,\Tilde{F}^{\gamma,lin}  ).
\end{aligned}
\]

Let $U:=(u,\partial u)$, and use the same Taylor expansions as before. By the Coifman--Meyer estimates~\eqref{CM+} together with the
Moser-type estimates~\eqref{bmo-hs}, we obtain for instance
\[
\begin{aligned}
    \| T_{\partial_\alpha \partial_\beta v} \Lambda_{\geq 2}( g^{\alpha\beta}(U)) \|_{L^2} 
    &
\lesssim 
\sum_{|I| =2}
\sum_{(\alpha,\beta) \neq (0,0)}
\| \partial_{\alpha}(  g^{I} (U) U^{I} ) \|_{L^{\infty}}
\| \partial_{\beta} v \|_{L^2}
\\
&
\lesssim 
 \A_0
 \A_1
\| v[t]\|_{H^{1} \times L^2}.
\end{aligned}
\]
Note that we are allowed to use Coifman--Meyer estimates~\eqref{CM} to move one derivative because at least one derivative of $\partial_\alpha \partial _\beta$ is spatial. 
The remaining terms are estimated in a similar fashion;  we omit the details.

We now estimate the time derivative of the cubic nonlinearity in \(H^{-1}\):
\[
\begin{aligned}
    \partial_t (  N^{lin}_{bal}(u)v) 
    & = T_{ \partial_t \partial_\alpha \partial_\beta v} g^{\alpha \beta }(u,\partial u ) 
    +
     T_{ \partial_\alpha \partial_\beta v} \partial_t g^{\alpha \beta }(u,\partial u ) 
    + \Pi(\partial_t \partial_\alpha \partial_\beta v, g^{\alpha \beta }(u,\partial u ) ) 
    \\
    &
   \quad  +  \Pi(\partial_\alpha \partial_\beta v, \partial_t g^{\alpha \beta }(u,\partial u ) ) 
    + T_{ \partial_t \Tilde{\partial}_{\gamma} v} \Tilde{F}^{\gamma,lin} 
    + T_{ \Tilde{\partial}_{\gamma} v} \partial_t \Tilde{F}^{\gamma,lin} 
    \\
    &
   \quad  + \Pi(\partial_t  \Tilde{\partial}_{\gamma} v,\Tilde{F}^{\gamma,lin}  )
     + \Pi(\Tilde{\partial}_{\gamma} v, \partial_t  \Tilde{F}^{\gamma,lin}  ).
\end{aligned}
\]
Arguing as above, we obtain 
\[
\begin{aligned}
    \|  T_{ \partial_t \partial_\alpha \partial_\beta v} \Lambda_{\geq 2}g^{\alpha \beta }(u,\partial u ) \|_{H^{-1}}
    &
    \lesssim
     \|  T_{ \partial_t \partial_\alpha \partial_\beta v} \Lambda_{\geq 2}g^{\alpha \beta }(u,\partial u ) \|_{L^2}
     \\
     &
     \lesssim 
     \| T_{ \partial_x v_{tt} }\Lambda_{\geq 2} g^{01}(u,\partial u ) \|_{L^2} 
     +   \| T_{ \partial_x^2 v_{t} } \Lambda_{\geq 2}g^{11}(u,\partial u ) \|_{L^2} 
     \\
     &
      \lesssim \| v_{tt} \|_{H^{-1}} \| 
      \Lambda_{\geq 2} g^{01}(u,\partial u )  \|_{W^{2,\infty}}
       +  \| v_t \|_{L^2}\|\Lambda_{\geq 2} g^{11}(u,\partial u ) \|_{W^{2,\infty}} .
\end{aligned}
\]
Using the \( H^{-1}\) control for \(v_{tt}\) in \eqref{eq:vtt} together with Moser estimate~\eqref{bmo-hs}, we conclude
\[
\begin{aligned}
    \|  T_{ \partial_t \partial_\alpha \partial_\beta v}
    \Lambda_{\ge 2 }g^{\alpha \beta }(u,\partial u ) \|_{H^{-1}} 
  & 
  \lesssim
  \A_0 \A_2 \| v[t] \|_{H^1 \times L^2} .
\end{aligned}
\]
Next, we estimate the second term. Since the term \((\alpha,\beta) = (0,0)\) is eliminated, at least one derivative is spatial:
\[
\begin{aligned}
      \| \Lambda_{\geq 3} ( T_{ \partial_\alpha \partial_\beta v} \partial_t  g^{\alpha \beta }(u,\partial u ) ) \|_{H^{-1}}
    &
     \lesssim 
     \|  T_{\partial_\alpha v} \partial_t \partial_x \Lambda_{\geq 2 }g^{ \alpha 1}(u,\partial u) \|_{L^2} 
     \\
     & 
     \lesssim
     \| \partial v  \|_{L^2 } \| \partial_t \partial_x \Lambda_{\geq 2 }g^{ \alpha 1}(u,\partial u) \|_{L^{\infty}}
     \\
     &
     \lesssim
     \A_0 \A_2 \| v[t] \|_{H^1 \times L^2}.
\end{aligned}
\]
The remaining terms are treated in the same manner.

\end{proof}

\section{Normal forms for related equations}
\label{s:normal-forms extras}
In this section we revisit the linearized equation in paradifferential form \eqref{para-KG-lin}. Starting from this equation, we introduce a more general class of paradifferential Klein--Gordon type equations of the form
 \begin{equation}\label{quadratic-para-lin}
      L_{KG}^{para}v = H_{00}(u_t,v_t) + H_{01}(u_t,v)+  H_{10}(u,v_t) + H_{11}(u,v)  + F ,
\end{equation}
whose quadratic component on the right  contains general quadratic
high--low and low--high interactions together with a generic quadratic (and higher) source term $F$ depending 
on $u,u_t, v,v_t$ and their derivatives. Equations in this class will reappear later, both in the analysis of the paradifferential linearized equation and in the construction of higher–order modified energies.

\smallskip

Our goal is to perform a bilinear normal form transformation that eliminates the quadratic terms on the right, and produces a new variable satisfying an equation with additional, perturbative cubic (and higher) source terms.

To carry out this program,  we establish three lemmas.
Lemma~\ref{l-nfl-exist} performs the same computation to construct the normal form variables $w$ as in Lemma~\eqref{nfl}, but stated for the more general source terms. A key point is that the microlocal structure of the correction---in particular, its frequency support and the associated derivative imbalance---is inherited from the quadratic source terms in the equation \eqref{quadratic-para-lin}.
Lemma~\ref{nfl-invert} shows that these normal form transformations are invertible under appropriate assumptions on the bilinear source terms. Lemma~\ref{source} provides
bounds for the resulting cubic source terms in the transformed equation, namely for the cubic and higher order components of \(L_{KG}^{\mathrm{para}} w\).

\medskip

In the next lemma, we view the paradifferential equation~\eqref{quadratic-para-lin} as evolving on a low--frequency background solution $u$ of \eqref{kg1}. 
For the bilinear terms we will impose two additional assumptions (i) each quadratic term carries a total of three\footnote{This is a natural assumption as these are usually given by the quasilinear terms, which have three derivatives in total.} derivatives (counting time derivatives as well), and 
(ii) the balance of high and low frequency derivatives is so that the resulting corrections and cubic remainders are perturbative.

We remark that the computation below is similar to the normal form computation for the quadratic Klein–Gordon equation~\eqref{quadratic-KG} in Section~\ref{s:normal-form}. However, in Lemma~\ref{nfl}, the symbols are assumed to be polynomials, whereas here we carry out the same computation for slightly larger symbol classes.
\begin{lemma}  \label{l-nfl-exist}
   Assume that $u$ satisfies  \eqref{kg1} i.e., 
    \[
- g^{\alpha \beta}(u,\partial u) \partial_\alpha
\partial_\beta u + m u  =f(u,\partial u),
    \]
and $v$ satisfies \eqref{quadratic-para-lin} with the quadratic source terms supported  in the low-high \emph{(LH)} and the high-low  \emph{(HL)} frequency region, with symbol regularity as follows:
\[
h_{00}\in S^{1,0},\qquad h_{01}\in S^{1,1},\qquad
h_{10}\in S^{2,0},\qquad h_{11}\in S^{2,1}.
\]

Then there exists a normal form variable
\[
w = v + A(u,v) + B(u_t,v_t) +C(u,v_t) + D(u_t,v),
\]
such that the quadratic part in \( L_{KG}^{para} w\) vanishes. 
Moreover,
$A(u,v)$, $B(u_t,v_t)$, $C(u_t,v)$, and $D(u,v_t)$ are bilinear operators, whose frequency support of the normal variables is inherited from the bilinear source terms and are satisfying the following regularity
\begin{itemize}
\item in the \emph{(LH)} case their symbols lie in
$S^{3,0},\,S^{2,-1},\,S^{2,0},\,S^{3,-1}$,\\

\item in the \emph{(HL)} case their symbols lie in
$S^{1,2},\,S^{0,1},\,S^{0,2},\,S^{1,2}$, respectively. 
\end{itemize}

\end{lemma}
\begin{proof}
    For the normal form computation, it suffices to retain only the quadratic terms. 
Accordingly, we first insert the homogeneity cutoffs and  apply the normal form transformation to the quadratic part. 
Define
\[
  \begin{aligned}
  \Lambda_2 \bigr( L^{para}_{KG} v \bigl) \, :=  \Lambda_2 \bigr(L^{para}_{KG} v +  L^{para}_{KG} ( A(u, v) + B(u_t,v_t) + C(u_t, v ) +D(u,v_t)) 
  \bigl).
  \end{aligned}
\]
Matching the quadratic terms on the right-hand side yields a $4 \times 4$ system. 
Using the time symmetry, we may split this system into two \( 2 \times 2\) subsystems:
\[
\left\{
\begin{aligned}
 H_{00}(u_t,v_t) & =  2 A(u_x,v_x) - 2 \Lambda_2(B(u_{tt}, v_{tt})) +  m  A(u,v),\\
 H_{11}(u,v) & = 2B(u_{tx}, v_{tx}) + m B(u_t,v_t) - 2A(u_t,v_t),
\end{aligned}
\right.
\]
and
\[
\left\{
\begin{aligned}
H_{01}(u_t,v) & =  2D(u_t,v_{xx} - mv)-2 C(u_{tx},v_x) - m C(u_t,v),\\
H_{10}(u,v_t) & =  2C(u_{xx}-mu, v_t) -2 D(u_{x}, v_{tx}) -mD(u,v_t).
\end{aligned}
\right.
\]
Passing to symbols, we obtain the system for $a$ and $b$
\[
\left\{
\begin{aligned}
h_{11}(\xi_1,\xi_2) 
&= (m-2\xi_1 \xi_2 )a(\xi_1,\xi_2) -2(\xi_1^2 + m)(\xi_2^2 +m)b(\xi_1,\xi_2), \\
  h_{00}(\xi_1,\xi_2) & = (m-2\xi_1\xi_2) b(\xi_1,\xi_2) -2a(\xi_1,\xi_2) .
\end{aligned}
\right.
\]
Similarly, we obtain the system for $c$ and \( d\) : 
\[
\left\{
\begin{aligned}
     h_{01} (\xi_1,\xi_2) & = (-2\xi_1 \xi_2 + m  ) c(\xi_1,\xi_2) + 2 (\xi_2^2 + m)d(\xi_1,\xi_2) ,\\
     h_{10} (\xi_1,\xi_2) & = (-2\xi_1 \xi_2 + m ) d(\xi_1,\xi_2) + 2 ( \xi_1^2 + m) c (\xi_1,\xi_2) .
\end{aligned}
\right.
\]

Solving both systems yields the following expressions for $a$, $b$, $c$, $d$ 

\[
\left\{
\begin{aligned}
  a(\xi_1,\xi_2)&= \frac{(m-2\xi_1\xi_2) h_{11} + 2(\xi_1^2+m)(\xi_2^2 + m)h_{00}}{(m -2\xi_1\xi_2)^2  - 4(\xi_1^2 + m)(\xi_2^2 +m)} ,\\
  b(\xi_1,\xi_2) & =
  \frac{ 2h_{11} +(m-2\xi_1\xi_2) h_{00} }{  (m -2\xi_1\xi_2)^2  - 4(\xi_1^2 + m)(\xi_2^2 +m)},
  \\
  c(\xi_1,\xi_2) & = \frac{( m -2\xi_1 \xi_2   ) h_{01} 
     -  2 (\xi_2^2 + m) h_{10}   
    }{(m -2\xi_1\xi_2)^2  - 4(\xi_1^2 + m)(\xi_2^2 +m)} ,
    \\
  d(\xi_1,\xi_2) & = \frac{( m -2\xi_1 \xi_2   ) h_{10} 
    -  2 (\xi_1^2 + m) h_{01}  
    }{(m -2\xi_1\xi_2)^2  - 4(\xi_1^2 + m)(\xi_2^2 +m)} . 
\end{aligned}
\right.
\]
We now investigate the symbol regularity in each of the two cases described above 
 
 \smallskip
 \noindent  \emph{Low--high support}:
  we can use the expansion in Lemma~\ref{nfl}:
 \[
 \frac{1}{\Delta} = - \frac{1}{4m} \xi_2^{-2} (1 - \frac{\xi_1}{\xi_2}+ O( (1 + |\xi_1|^2) \xi_2^{-2} ) ) \in S^{0,-2}.
 \]
 
Since \( (m-2\xi_1 \xi_2) h_{11} \in S^{3,2}\) and \( (\xi_1^2+m)(\xi_2^2 + m) h_{00}\in S^{3,2}\), we have \( a \in S^{3,0}\). 
 
Since \( h_{11} \in S^{2,1}\) and  \( 2\xi_1\xi_2 \in S^{2,1}\), we have \( b \in S^{2,-1}\).

Since \( (m-2\xi_1\xi_2) h_{01} \in S^{2,2}\) and  \( 2(\xi_2^2 + m) h_{10} \in S^{2,2}\), we have \( c \in S^{2,0}\)

Since \( (m-2\xi_1\xi_2) h_{10} \in S^{3,1}\) and  \( 2(\xi_1^2 + m) h_{01} \in S^{3,1}\), we have \( d \in S^{3,-1}\).

\smallskip
\noindent  \emph{High--low support}: we use the analogous expansion,
\[
 \frac{1}{\Delta} = - \frac{1}{4m} \xi_1^{-2} (1 - \frac{\xi_2}{\xi_1}+ O( (1 + |\xi_2|^2) \xi_1^{-2} ) ) \in S^{-2,0} .
\]
Since \( (m-2\xi_1 \xi_2) h_{11} \in S^{3,2}\) 
and \( (\xi_1^2+m)(\xi_2^2 + m) h_{00}\in S^{3,2}\), we have \( a \in S^{1,2}\).

Since \( h_{11} \in S^{2,1}\) and  \( (m- 2\xi_1\xi_2)h_{00} \in S^{2,1}\), we have \( b \in S^{0,1}\).

Since \( (m-2\xi_1\xi_2) h_{01} \in S^{2,2}\) and  \( 2(\xi_2^2 + m) h_{10} \in S^{2,2}\), we have \( c \in S^{0,2}\).

Since \( (m-2\xi_1\xi_2) h_{10} \in S^{3,2}\) and  \( 2(\xi_1^2 + m) h_{01} \in S^{3,2}\), we have \( d \in S^{1,2}\).
In particular, these symbols admit polyhomogeneous expansions.

\end{proof}

We next consider invertibility type bounds for the normal form transformation.

\begin{lemma} \label{nfl-invert}
    Assume that $u$  satisfies 
    \[
        - g^{\alpha \beta}(u,\partial u) \partial_\alpha
\partial_\beta u + m u  =f(u,\partial u),
\]
and $v$ satisfies \eqref{quadratic-para-lin}. We assume that 
\begin{enumerate}[label=(\alph*)]
\item $A,B,C,D$ are bilinear operators of type $(LH)$ or $(HL)$ so that 
\begin{itemize}
\item In the \emph{(LH)} case their symbols lie in
$S^{3,0},\,S^{2,-1},\,S^{2,0},\,S^{3,-1}$, respectively;

\item In the \emph{(HL)} case their symbols lie in
$S^{1,2},\,S^{0,1},\,S^{0,2},\,S^{1,2}$, respectively.
\end{itemize}

\item The source term $F$ is linear in $v,v_t$
and satisfies the bound
\begin{equation} \label{F-sour-q}
    \|  F \|_{ L^2}
\lesssim 
\A_2
\| v[t] \|_{H^1 \times L^2} . 
\end{equation}
\end{enumerate}

Then, for the normal form variable
\[
w = v + A(u,v) + B(u_t,v_t) +C(u_t,v) + D(u,v_t),
\]
we have the fixed time invertibility estimates 
\[
\|(w,w_t)-(v,v_t)\|_{H^1\times L^2}
 \lesssim
 \A_2 \,\|(v,v_t)\|_{H^1\times L^2}.
\]
\end{lemma}

\begin{proof} We consider separately the contributions from bilinear operators corresponding to the case (LH), respectively
(HL).

\medskip
\emph{The case \emph{(LH)}.} We start from

\begin{equation} \label{eq:ivt-exn}
    \begin{aligned}
     & \| ( w , w_t) - (v,v_t) \|_{H^1 \times L^2}
    \\
    &
    \leq  \| (A(u,v) \|_{H^1} + \| B(u_t,v_t)\|_{H^1} 
    +  \| C(u_t,v) \|_{H^1} + \| D(u,v_t) \|_{H^1} 
    \\
    &
    \quad  \, +
    \| \partial_t  A(u,v) \|_{L^2} 
    + \|\partial_t B(u_t,v_t) \|_{L^2} 
    + \| \partial_t C(u_t,v) \|_{L^2}
    + \| \partial_t D(u,v_t) \|_{L^2}.
\end{aligned}
\end{equation}

There are three cases to consider: 
\begin{enumerate}[label=(\alph*)]
\item  Expressions which depend only on $u$, $u_t$, $ v$, and $ v_t$, 
which are estimated directly.

\item  Expressions which also depend on $u_{tt}$, where
we need to first use the equation \eqref{kg1}.

\item  Expressions which also depend on $v_{tt}$, where
we need to first use the equation \eqref{quadratic-para-lin}.
\end{enumerate}

\smallskip
\noindent\emph{Case (a).}
Since $A(u,v)$ has symbol in $S^{3,0}$, the paraproduct estimates yield 
\[
\begin{aligned}
    \Vert \partial_x A_{lh}(u,v) \Vert_{L^2} &\lesssim 
      \Vert A_{lh}(u_x, v) \Vert_{L^2} 
    + \Vert A_{lh}(u,v_x) \Vert_{L^2}\\
    & \lesssim \Vert u \Vert_{W^{3,\infty}} \Vert( v,v_t) \Vert_{H^1 \times L^2}.
\end{aligned}
\]
Next, since the symbol of $B$ lies in $S^{2,-1}$, we use the \( -1\) order to place one spatial derivative on  \(v_t\):
\[
\begin{aligned}
    \Vert \partial_x B_{lh}(u_t,v_t) \Vert_{L^2} 
    &
    \lesssim 
        \Vert B_{lh}(\partial_x u_t, v_t) \Vert_{L^2} 
    + 
        \Vert B_{lh}( u_t, \partial_x v_t) \Vert_{L^2} 
      \\
     &  
     \lesssim \Vert u_t \Vert_{W^{2,\infty}}  \Vert (v,v_t) \Vert_{H^1\times L^2},    
\end{aligned}
\]
as needed. The contributions of the remaining terms in case (a) are treated in the same manner.
\smallskip

\noindent\emph{Case (b).}
For the terms involving $u_{tt}$, we consider for instance the bound for 
$\Vert B_{lh}(u_{tt},v_t) \Vert_{L^2}$.
For this we
use the bounds for \(u_{tt}\) in \eqref{eq:p2-utt}  and the bilinear estimate~\eqref{BMO-shift} to get
\[
\begin{aligned}
    \Vert B_{lh}(u_{tt},v_t) \Vert_{L^2} \lesssim 
    & \
     \Vert u_{tt} \|_{W^{1,\infty}} \| v_t \Vert_{L^2} \lesssim \A_2 \| v_t \Vert_{L^2}.
\end{aligned}
\]
The other terms in case (b) are estimated in the same way.

\smallskip

\noindent\emph{Case (c).} Using the paradifferential equation for $v$, we separate 
$v_{tt}$ into two terms,
\begin{equation}\label{eq:parvtt-split}
v_{tt}=v_{tt}^{(-1)}+v_{tt}^{(0)}.
\end{equation}
The principal part is given by
\begin{equation}
v_{tt}^{(-1)}:= 2T_{g^{01}}\partial_x v_t  + T_{g^{11}}\partial_x^2 v,
\end{equation}
and the remainder by
\[
\begin{aligned}
v_{tt}^{(0)} :=\;& T_{\Tilde F^{\gamma,\mathrm{lin}}}\,\Tilde\partial_\gamma v - m v + H_{00}(u_t,v_t) + H_{01}(u_t,v) + H_{10}(u,v_t) + H_{11}(u,v) + F.
\end{aligned}
\]

We estimate $v_{tt}^{(-1)}$ and $\Lambda_{\ge 2 }v_{tt}^{(-1)}$ in $H^{-1}$ and control $v_{tt}^{(0)}$ and $\Lambda_{\ge 2 }v_{tt}^{(0)}$ in $L^2$.

We illustrate the estimate for the first term; the second term is similar. Using paraproduct estimates~\eqref{CM}, the Moser estimate~\eqref{bmo-hs}, and a standard commutator bound, we obtain
\[
\begin{aligned}
     \| \jD^{-1} T_{g^{11}}\partial_x^2 v   \|_{L^2 }  
     \lesssim 
      \|  T_{g^{11}} \jD^{-1}\partial_x^2 v   \|_{L^2 } 
      + 
       \|  [T_{g^{11}}, \jD^{-1} ]\partial_x^2 v   \|_{L^2 } 
       \lesssim   \| v[t] \|_{H^1 \times L^2},
\end{aligned}
\]
respectively,
\[
\begin{aligned}
      \| \jD^{-1} T_{ \Lambda_{\ge 1 }g^{11}}\partial_x^2 v   \|_{L^2 }  
     \lesssim 
      \|  T_{ \Lambda_{\ge 1 }g^{11}} \jD^{-1}\partial_x^2 v   \|_{L^2 } 
      + 
       \|  [T_{ \Lambda_{\ge 1 }g^{11}}, \jD^{-1} ]\partial_x^2 v   \|_{L^2 } 
       \lesssim  \A_2 \| v[t] \|_{H^1 \times L^2} .
\end{aligned}
\]

To estimate the remainder $v_{tt}^{(0)}$ in $L^2$, it suffices to treat the representative
terms \( T_{\Tilde F^{\gamma,\mathrm{lin}}}\,\Tilde\partial_\gamma v \), \( H_{00}(u_t,v_t)\), and \(F\).
The paradifferential term is estimated by \eqref{CM},
\[
\begin{aligned}
    \| T_{\Tilde F^{\gamma,\mathrm{lin}}}\,\Tilde\partial_\gamma v \|_{L^2 }
    \lesssim 
    \| \Tilde F^{\gamma,\mathrm{lin}} \|_{L^{\infty}} 
    \| v[t] \|_{H^1\times L^2} 
    \lesssim 
    \A_2
    \| v[t] \|_{H^1\times L^2} .
\end{aligned}
\]
For the bilinear term, we note that the symbol $h_{00}$ belongs to $S^{1,0}$ on both the
low--high and the high--low frequency supports. Therefore,
\[
\| H_{00}^{lh}(u_t,v_t) \|_{L^2} + \| H_{00}^{hl}(u_t,v_t) \|_{L^2} 
\lesssim \A_1 \| v[t] \|_{H^1 \times L^2}.
\] 
The remaining contributions in $v_{tt}^{(0)}$ are estimated in the same way. Finally, the
perturbative cubic terms follow directly from \eqref{F-sour-q}.
Hence, 
\begin{equation} \label{eq:pa-vtt-est}
\left\{
\begin{aligned}
  \| v_{tt}^{(-1)} \|_{H^{-1}}
    & \lesssim 
    \| v[t] \|_{H^1\times L^2} ,
     \quad
     \| \Lambda_{\ge 2 } v_{tt}^{(-1)} \|_{H^{-1}} 
     \lesssim 
    \A_2 \| v[t] \|_{H^1\times L^2} 
    \\
     \| v_{tt}^{(0)} \|_{L^2} 
    & \lesssim 
    \| v[t] \|_{H^1\times L^2}  , 
    \quad
    \| \Lambda_{\ge 2 } v_{tt}^{(0)} \|_{L^2} 
     \lesssim 
    \A_2 \| v[t] \|_{H^1\times L^2}  . 
\end{aligned}
\right.
\end{equation}

Using the expansion~\eqref{eq:parvtt-split} together with the bounds~\eqref{eq:pa-vtt-est} and recalling that
$b_{lh}\in S^{2,-1}$, we obtain
\[
\begin{aligned}
    \| B_{lh}(u_t,v_{tt}) \|_{L^2} 
   & \lesssim 
     \| B_{lh}(u_t,v_{tt}^{(-1)}) \|_{L^2} 
     + \| B_{lh}(u_t,v_{tt}^{(0)}) \|_{L^2} 
     \\
     & 
     \lesssim
     \| u_t \|_{W^{2,\infty}} 
     \| v_{tt}^{(-1)} \|_{H^{-1}}
     + 
      \| u_t \|_{W^{1,\infty}} 
     \| v_{tt}^{(0)} \|_{L^2}
     \\
     &
     \lesssim
     \A_2 
      \| v[t] \|_{H^1\times L^2} ,
\end{aligned}
\]
where we used the paraproduct estimate~\eqref{BMO-shift} to move one derivative onto the \(v \) factor in the second term (equivalently, so that the $v_{tt}^{(0)}$--factor is measured
in $L^2$).

Collecting all the bounds yields the desired inequality. The estimates for \(D(u,v_{tt})\) follow similarly. 

\medskip

\emph{The case \emph{(HL)}.}
    To prove the invertibility estimate, we first expand the left-hand side exactly as in \eqref{eq:ivt-exn}.

As before, we split the analysis into the same three sub cases as in the proof of the preceding propositions. Moreover, we repeatedly apply the paraproduct estimate~\eqref{BMO-shift} to move one derivative from the low-frequency $v$ onto $u$ whenever necessary.

\medskip

\noindent\textit{Case (a).}
 For instance, since $a_{hl} \in S^{1,2}$, we obtain
\[
\begin{aligned}
     \| A_{hl}(u,v) \|_{H^1}  &
     \lesssim
    \| \partial u   \|_{W^{2,\infty}} 
    \| v \|_{H^1} 
    \lesssim
    \A_2
    \| v[t] \|_{H^1 \times L^2} .
\end{aligned}
\]
The remaining terms that do not involve $u_{tt}$ or $v_{tt}$ are treated in the same way.

\noindent\textit{Case (b).} 
For the contributions involving $u_{tt}$, we combine  the bounds~\eqref{eq:p2-utt}   for $u_{tt}$ with paraproduct estimates~\eqref{CM} and the symbol bound for \( b_{hl} \in S^{0,1}\). This yields, for example,
\[
\begin{aligned}
    \Vert B_{hl}(u_{tt},v_t) \Vert_{L^2} \lesssim 
    &
   \| u_{tt} \|_{W^{1,\infty}} \| v_t \|_{L^2} \lesssim \A_2  \| v[t] \|_{H^1 \times L^2}   .
\end{aligned}
\]

\noindent\textit{Case (c).} 
For the terms involving \( v_{tt}\), we use the expansion~\eqref{eq:parvtt-split} together with estimates~\eqref{eq:pa-vtt-est}. Recalling that \( b_{hl} \in S^{0,1}\), we obtain
\[
\begin{aligned}
   \| B_{hl}(u_t,v_{tt}) \|_{L^2} 
    & 
    \lesssim 
   \| B_{hl}(u_t,v_{tt}^{(-1)}) \|_{L^2} 
   +
    \| B_{hl}(u_t,v_{tt}^{(0)}) \|_{L^2} 
    \\
    & 
    \lesssim 
    \| u_t \|_{W^{2,\infty}} 
    \| v_{tt}^{(-1)} \|_{H^{-1}}
    + 
    \| u_t \|_{W^{1,\infty}} 
    \| v_{tt}^{(0)} \|_{L^2} 
    \\
    & 
    \lesssim \A_2 \| v[t] \|_{H^1 \times L^2}.
\end{aligned}
\]
The estimates for \(D(u,v_{tt})\) are proved in the same way.

\end{proof}

\begin{lemma} [source term bounds] \label{source}
 Let $u$, \(v\), \( w\),
$A(u,v)$, $B(u_t,v_t)$, $C(u_t,v)$, and $D(u,v_t)$ satisfy the same assumptions as in Lemma~\ref{nfl-invert}.  We further assume the source term \(F\) satisfies the following bounds:
\begin{equation} \label{F-sour-t}
\begin{aligned}
    \| F \|_{L^2} &\lesssim \A_1  \| v[t] \|_{H^1 \times L^2},
\\
\| \Lambda_2 \partial_t F \|_{H^{-1}} &\lesssim  \A_1 \| v[t] \|_{H^1 \times L^2},
\\
\| \Lambda_{\geq 3} \partial_t F \|_{H^{-1}} &\lesssim  \A_0 \A_2 \| v[t] \|_{H^1 \times L^2}.
\end{aligned}
\end{equation}

Then the cubic and higher source term generated by the normal form transformation satisfies the estimate
   \[
  \Vert \Lambda_{\geq3} \bigr(L_{KG}^{para}  (w -v)\bigl) \Vert_{L^2} \lesssim_{\A_0} 
  \A_0
  \A_3  
  \Vert v[t] \Vert_{H^1\times L^2}.
   \]

\end{lemma}
\begin{proof} [Proof of case \emph{(LH)}]

We begin by expanding:
    \[
    \begin{aligned}
  \Lambda_{\geq 3}(L_{KG}^{para}( w -v))  
& =    
  \Lambda_{\geq 3} \bigr(L_{KG}^{para} A_{lh}(u,v)  \bigl)
+  \Lambda_{\geq 3} \bigr(L_{KG}^{para} B_{lh}(u_t,v_t)  \bigl)
+  \Lambda_{\geq 3} \bigr(L_{KG}^{para} C_{lh}(u_t,v)  \bigl) 
\\
& \quad \ +
  \Lambda_{\geq 3} \bigr(L_{KG}^{para} D_{lh}(u,v_t)  \bigl) .
\end{aligned}
      \]
The bounds for the four terms above are largely similar to each other, so for simplicity we take \(\Lambda_{\geq 3} \bigr(L_{KG}^{para} B_{lh}(u_t,v_t)  \bigl) \) as an example, since it potentially contains more time derivatives.

We begin by expanding the terms as follows
\begin{equation} \label{eq:sour-exp}
    \begin{aligned}
       \Lambda_{\ge 3} L^{para}_{KG} B_{lh}(u_t,v_t) 
       & =  \
        \Lambda_{\ge 3} \bigl(
        -T_{g^{\alpha \beta}} \partial_\alpha \partial_\beta B_{lh}(u_t,v_t) 
        - T_{ \Tilde{F}^{\gamma,lin}} \Tilde{\partial}_{\gamma} B_{lh} (u_t,v_t)  
        \bigr) 
        \\
        & = \ \Lambda_{\ge 3} \bigl(
        2\big(
         B_{lh}(u_t, T_{g^{01}}\partial_x   v_{tt} ) - T_{g^{01}} \partial_x   B_{lh}(u_t,v_{tt}) \big) 
         \\
         &
         \qquad \
         + 
         \big( 
         B_{lh}(u_t, T_{g^{11}} \partial_x^2 v_t)
         - T_{g^{11}} \partial_x^2 B_{lh}(u_t,v_t)
         \big)
        \\
        & \qquad \
       +  \big( 
       B_{lh}(u_t, \partial_t ( T_{ \Tilde{F}^{\gamma,lin}} \Tilde{\partial_{\gamma}} v  )  )  - 
     T_{ \Tilde{F}^{\gamma,lin}} \Tilde{\partial_{\gamma}} B_{lh}(u_t,v_t) 
     \big)
     \\
     &
     \qquad \
     +  2 B_{lh}(u_{tt}, v_{tt}) + B_{lh}(u_{ttt},v_t) 
     + 2 B_{lh}(u_t, T_{\partial_t g^{01}}\partial_x   v_{t} ) 
     \\
     & \qquad \ 
     - 2 T_{g^{01}} \partial_x  B_{lh}(u_{tt},v_t) 
     +  B_{lh}(u_t,  T_{\partial_t g^{11}} \partial_x^2 v) 
     \\
     & \qquad \ 
     - B_{lh}(u_t, \partial_t(L_{KG}^{para} v ))
     \bigr) .
    \end{aligned}
\end{equation}

There are four groups of terms to consider:

\begin{enumerate}[label=(\alph*)]

\item Commutator terms.
\begin{enumerate} [label=(\roman*)]
    \item A commutator structure acting on $v_{tt}$ (line 1).
    \item A commutator structure acting on $v_{t}$ (line 2).
\end{enumerate}
\item Terms containing \(v_{tt}\)  not in commutator form.
    \begin{enumerate}[label=(\roman*)] 
    \item  Expressions which depend on both $v_{tt}$ and \( u_{tt}\), where
we need to use the bounds in \eqref{eq:p2-utt} and \eqref{eq:pa-vtt-est}.
(first term on line 4)

    \item Expressions which also depend only on \( u_t \) (line 3) .
\end{enumerate}
\item Terms containing \(v_{t}\)  not in commutator form.
\begin{enumerate}[label=(\roman*)]
 \item  Expressions which also depend on $u_{tt}$ and \( u_{ttt}\), where
we need to use the bounds~\eqref{eq:p2-utt} (second term on line 4). 
\item  Expressions which depend only on $u$, $u_t$, $ v$, and $ v_t$, 
which are estimated directly (lines 4 and 5).
\end{enumerate}
\item Remainder terms.
\begin{enumerate}[label=(\roman*)]
    \item Lower-order contributions involving $\partial_t H_{ij}$ and $\partial_t F$ (line 6).

\end{enumerate}
\end{enumerate}

\medskip
\noindent \emph{Case (a)}:

\smallskip
 \emph{Sub-case (i)}: It can be estimated using the commutator bound in Lemma~\ref{para-com} and Moser estimate~\eqref{bmo-hs}. In particular, using the symbol bound \( b_{lh} \in S^{2,-1}\) together with the bounds~\eqref{eq:pa-vtt-est} for 
\( v_{tt} \), we obtain
\[
\begin{aligned}
  \|    B_{lh}(u_t, T_{g^{01}}\partial_x   v_{tt} ) - T_{g^{01}} \partial_x   B_{lh}(u_t,v_{tt})   \|_{L^2}  
  & 
  \lesssim \, 
  \| u_t \|_{W^{3,\infty} }
   \| \partial u \|_{L^{\infty}}
   \| v_{tt} \|_{H^{-1}}
   \\
   & \lesssim \,
   \| u_t \|_{W^{3,\infty} }
   \| \partial u \|_{L^{\infty}}
    \bigl( \| v_{tt}^{(-1)} \|_{H^{-1}} +  \| v_{tt}^{(0)} \|_{H^{-1}} \bigr) 
  \\
  & \, 
  \lesssim 
\A_0 \A_3
  \| v[t] \|_{H^1 \times L^2}.
\end{aligned}
\]

\smallskip
\emph{Sub-case (ii)}: We next estimate the the commutator applied to $v_{t}$ exactly as above but without using any expansion.
\[
\begin{aligned}
  \|  B_{lh}(u_t, T_{g^{11}} \partial_x^2 v_t)
         - T_{g^{11}} \partial_x^2 B_{lh}(u_t,v_t) \|_{L^2} 
         \lesssim \, \A_0 \A_3 \| v[t]\|_{H^1 \times L^2}.
\end{aligned}
\]

\medskip
\noindent \emph{Case (b)}:

\smallskip
\emph{Sub-case (i)}:
The expression $
\Lambda_{\geq3}( B_{lh}( u_{tt},v_{tt}))   $
is handled using the decomposition~\eqref{eq:parvtt-split} and the bounds~\eqref{eq:pa-vtt-est}.
In particular,
\[
\begin{aligned}
    \| \Lambda_{\geq3}( B_{lh}( u_{tt},v_{tt}))   \|_{L^2}
    \lesssim \
    \| B_{lh}( \Lambda_{\geq 2 }u_{tt}, v_{tt})  \|_{L^2}
    +
     \| B_{lh}( u_{tt}, \Lambda_{\geq 2 } v_{tt})  \|_{L^2}.
\end{aligned}
\]
For the first term, we combine the paraproduct bounds~\eqref{BMO-shift} and~\eqref{B-CM}
with the estimate for $u_{tt}$ in~\eqref{eq:p2-utt-2} and the decomposition/estimates for
$v_{tt}$ in~\eqref{eq:parvtt-split} and~\eqref{eq:pa-vtt-est}.
\[
\begin{aligned}
     \| B_{lh}( \Lambda_{\geq 2 }u_{tt}, v_{tt})  \|_{L^2} 
     & \lesssim \
     \| B_{lh}( \Lambda_{\geq 2 }u_{tt}, v_{tt}^{(-1)})  \|_{L^2} 
     +\| B_{lh}( \Lambda_{\geq 2 }u_{tt}, v_{tt}^{(0)})  \|_{L^2} 
     \\
     &
     \lesssim \
     \| \Lambda_{\geq 2 }u_{tt} \|_{W^{2,\infty}}
     \| v_{tt}^{(-1)} \|_{H^{-1}}
     + 
      \| \Lambda_{\geq 2 } u_{tt} \|_{W^{1,\infty}}
     \| v_{tt}^{(0)} \|_{L^2 }
     \\
     &
     \lesssim \ \A_0 \A_3 
     \| v[t] \|_{H^1 \times L^2}.
\end{aligned}
\]
The second term is estimated analogously. Using the decomposition
$v_{tt}=v_{tt}^{(-1)}+v_{tt}^{(0)}$, we obtain
\[
\begin{aligned}
      \| B_{lh}( u_{tt}, \Lambda_{\geq 2 } v_{tt})  \|_{L^2}
      & \lesssim 
       \| B_{lh}( u_{tt}, \Lambda_{\geq 2 } v_{tt}^{(-1)})  \|_{L^2}
       + \| B_{lh}( u_{tt}, \Lambda_{\geq 2 } v_{tt}^{(0)})  \|_{L^2}
       \\
        & \lesssim 
       \|  u_{tt}\|_{W^{2,\infty}} \| \Lambda_{\geq 2 } v_{tt}^{(-1)})  \|_{H^{-1}}
       +  \|  u_{tt}\|_{W^{3,\infty}} \| \Lambda_{\geq 2 } v_{tt}^{(0)})  \|_{L^2}
       \\
       & 
         \lesssim \A_0 \A_3 
     \| v[t] \|_{H^1 \times L^2}.
\end{aligned}
\]

\smallskip
\emph{Sub-case (ii)}:
Here we consider the expression
\[ 
B_{lh}(u_t, \partial_t ( T_{ \Tilde{F}^{\gamma,lin}} \Tilde{\partial_{\gamma}} v  )  )  - 
     T_{ \Tilde{F}^{\gamma,lin}} \Tilde{\partial_{\gamma}} B_{lh}(u_t,v_t),
\]
which has a commutator form. However, we do not need to exploit the commutator
structure, since at most one derivative falls on $v$. Instead, we estimate it directly using the bounds~\eqref{eq:pa-vtt-est} to control \( v_{tt}\)

     \[
     \begin{aligned}
      \|  B_{lh}(u_t,T_{ F^{1,lin}}   v_{tt}  )   \|_{L^2}
      & \lesssim \,
       \A_2 \| T_{ F^{1,lin}}   v_{tt}  \|_{L ^2} 
       \\
       & 
       \lesssim \,
       \A_2 \A_1 (\| v_{tt}^{(-1)} \|_{H^{-1}} + \| v_{tt}^{(0)} \|_{L^2} ) 
       \\
       & \lesssim \A_0 \A_3 
     \| v[t] \|_{H^1 \times L^2}.
     \end{aligned}
     \]
     
     We next consider the remaining terms arising from this expansion.
   The  expression \(  T_{F^{1,lin} } B_{lh}(u_t,v_{tt})\) can be handled similarly, while 
    \(T_{F^{1,lin} } B_{lh}(u_{tt},v_{t}) \) falls under Case (c)(i).
All remaining terms can be treated as in Case (c)(ii).

\medskip 
\noindent \emph{ Case (c)}:

\emph{Sub-case (i)}:
Here we estimate
\(
 B_{lh}( \Lambda_{\geq2}(\partial_t(u_{tt})),v_t)  .
\)
Using paraproduct bound~\eqref{CM+}, Moser estimate~\eqref{bmo-hs}, and the bound for \(u_{tt}\) in \eqref{eq:p2-utt-2}, we obtain
\[ 
\begin{aligned}
  \| B_{lh}( \Lambda_{\ge 2 }\partial_t u_{tt}  ,v_t) \|_{L^2}
   &
    \lesssim \,
    \| \Lambda_{\ge 2 }\partial_t u_{tt}  \|_{W^{1,\infty}}
    \|  v_t \|_{L^2} 
    \\
    &
    \lesssim \,
   \A_0 \A_3
      \| v[t] \|_{H^1 \times L^2}.
\end{aligned}
\]

\smallskip
 \emph{Sub-case (ii)}:
 We estimate these terms directly by using paraproduct estimates as well as the symbol bound \( b_{lh} \in S^{2,-1}\) and Moser estimates~\eqref{bmo-hs}:
\[
\begin{aligned}
    \|T_{ \Lambda_{\ge 1}g^{11}} B_{lh}(\partial_x u_t, \partial_x v_t) \|_{L^2}
    \lesssim 
    \A_0 \A_3
    \| v[t] \|_{H^1 \times L^2}   .
\end{aligned}
\]

The rest of the terms can be estimated similarly. 

\medskip 
\noindent \emph{Case (d)}

\emph{Sub-case (i)}:
Here  we estimate the the bilinear terms $H$ and $F$.
For the bilinear terms $H$, we treat $H_{00}(u_t,v_t)$ as a representative example, since it contains the largest number of time derivatives.
Recall also that $h_{00}\in S^{1,0}$.
Then 
\[
\begin{aligned}
     \|   B_{lh}(u_t, \partial_t  H_{00}(u_t,v_t) ) \|_{L^2} 
     &
     \lesssim \,
     \|   B_{lh}(u_t,   H_{00}(u_{tt},v_t) ) \|_{L^2} 
     +    \|   B_{lh}(u_t,   H_{00}(u_{t},v_{tt} ) \|_{L^2} 
     \\
     &
     \lesssim \,
     \| u_t \|_{W^{1,\infty}} \| H_{00} (u_{tt},v_t) ) \|_{L^2} 
     + 
      \| u_t \|_{W^{2,\infty}} \| H_{00} (u_{t},v_{tt}) ) \|_{H^{-1}} 
      \\
    & 
    \lesssim  \,
    \| u_t \|_{W^{1,\infty}} \| u_{tt} \|_{W^{1,\infty}} \| v_t  \|_{L^2} 
     + 
      \| u_t \|_{W^{2,\infty}} \| u_t\|_{W^{1,\infty}} \| v_{tt}^{(-1)} \|_{H^{-1}} 
      \\
      &
      \qquad \ +  \| u_t \|_{W^{2,\infty}} \| u_t\|_{L^{\infty}} \| v_{tt}^{(0)} \|_{L^2} 
      \\
      &
      \lesssim \
      \A_1 \A_2 \| v[t] \|_{H^1 \times L^2},
\end{aligned}
\]
where we have used \eqref{eq:pa-vtt-est} to control \( v_{tt}\) together with  the symbol bound for \( H_{00}\).

Finally, for the \( F\) term, we invoke the quadratic bound for \( \partial_t F\):
\[
\begin{aligned}
     \|  \Lambda_{\geq 3}  B_{lh}(u_t, \partial_t  F ) \|_{L^2} 
     & \lesssim \,
     \| u_t \|_{W^{2,\infty}} \| \Lambda_{2}\partial_t F  + \Lambda_{\geq 3}\partial_t F \|_{H^{-1}}
     \\
     &
     \lesssim \,
     \A_2 \A_1 \| v[t] \|_{H^1 \times L^2} + \A_2 \A_0 \A_2 \| v[t] \|_{H^1 \times L^2}.
\end{aligned}
\]

Hence the proof is complete.

\end{proof}
\begin{proof}[Proof of case \emph{(HL)}]

    We now estimate the source terms in the high--low case. The idea is similar to the previous case except that the high--low structure allows us to freely move derivatives from low-frequency \( v\) onto \( u\). As a representative (and worst) term, we consider $B^1_{hl}(u_t,v_t)$, which contain the most time derivatives.
       We can use the similar expansion in \eqref{eq:sour-exp}.
    This leads to four cases to consider:
    \begin{enumerate}[label=(\alph*)]
    \item  Expressions which also depend on $u_{ttt}$, where
we need to first use the equation \eqref{kg1}.

     \item  Expressions which also depend on $v_{ttt}$, where
we need to first use the equation \eqref{quadratic-para-lin}.
    \item  Expressions which also depend on $v_{tt}$ and \( u_{tt}\), where
we need to first use both of the equations \eqref{kg1} and \eqref{quadratic-para-lin}.
    \item  Expressions which depend only on $u$, $u_t$, $ v$, and $ v_t$, 
which are estimated directly.
\end{enumerate}

\medskip

\noindent \emph{Case (a)}:

\medskip

We estimate
\(
 B_{hl}( \Lambda_{\geq2}(\partial_t(u_{tt})),v_t)  
\)
using \eqref{eq:p2-utt} and the fact  \( b_{hl} \in S^{0,1}\):
\[
\begin{aligned}
  \| B_{hl}( \partial_t\bigr( u_{tt}) \bigl) ,v_t) \|_{L^2}
   &
    \lesssim
    \| \partial_t  u_{tt}  \|_{W^{1,\infty}}
    \|  v_t \|_{L^2} 
    \lesssim
    \A_0 \A_3
      \| v[t] \|_{H^1 \times L^2} .
\end{aligned}
\]

\medskip

\noindent \emph{Case (b)}:

We use paraproduct estimates~\eqref{CM} to move one derivative from the low-frequency factor \( v\) onto \(u\) and use Moser estimates~\eqref{bmo-hs}.
\[
\begin{aligned}
    \| B_{hl}( u_t , \partial_t\bigr( v_{tt}) \bigl)) \|_{L^2}
    & \lesssim 
    \|  B_{hl}( u_t , \partial_t( L^{para}_{KG}v -v_{tt})) \|_{L^2} 
    \\
    & 
    \quad \ 
    +  \|  B_{hl}( u_t , \partial_t( H_{00}(u_t,v_t) + H_{01}(u_t,v)+  H_{10}(u,v_t) + H_{11}(u,v)   ) \|_{L^2} 
    \\
    & \quad \ 
    + \| B_{hl}( u_t, \partial_t F) \|_{L^2} .
\end{aligned}
\]
For the terms in \( \partial_t( L^{para}_{KG}v -v_{tt})) \), we again estimate them using commutator structure:
\[
\begin{aligned}
   & \|   B_{hl}(u_t, \partial_t T_{\Lambda_{\geq 1} g^{\alpha\beta} } \partial_\alpha\partial_\beta v) -T_{ \Lambda_{\geq 1} g^{\alpha\beta} } \partial_\alpha\partial_\beta B_{hl}(u_t,v_t) \|_{L^2} 
    \\
    & \quad \
     + 
     \|  B_{hl}(u_t, \partial_t ( T_{ \Lambda_{\geq 1}\Tilde{F}^{\gamma,lin}} \Tilde{\partial_{\gamma}} v  )  )  - 
     T_{ \Lambda_{\geq 1}\Tilde{F}^{\gamma,lin}} \Tilde{\partial_{\gamma}} B_{hl}(u_t,v_t)\|_{L^2} 
     \\
    &\lesssim 
    \A_0 \A_3 \| v[t] \|_{H^1 \times L^2}.
\end{aligned}
\]
For the bilinear terms, we take \( H_{00}(u_t,v_t)\) as an example:
\[
\begin{aligned}
    \|  B_{hl}( u_t , \partial_t( H_{00}(u_t,v_t)  ) \|_{L^2}  
    & 
     \lesssim 
     \|  B_{hl}( u_t ,  H_{00}(u_{tt},v_t)   \|_{L^2}  
     +  \|  B_{hl}( u_t , H_{00}(u_t,v_{tt}) \|_{L^2}  
     \\
     &
     \lesssim 
      \|  \partial u \|_{W^{2,\infty}} \| H_{00}(u_{tt},v_t)   \|_{H^{-1}} 
      +
         \|  \partial u \|_{W^{2,\infty}} \| H_{00}(u_{t},v_{tt})   \|_{H^{-1}} 
         \\
    & 
    \lesssim \A_2 \A_1 \|v[t] \|_{H^{1} \times L^2},
\end{aligned}
\]
where, in the last line, we have used the estimates for \(v_{tt}\) in \(H^{-1}\) in \eqref{eq:vtt}.
Now we estimate the \(\partial_t F\):
\[
\begin{aligned}
    \| B_{hl}( u_t, \partial_t F) \|_{L^2}  
    &
    \lesssim 
     \| B_{hl}( u_t, \partial_t \Lambda_2 F) \|_{L^2}  
     +
     \| B_{hl}( u_t, \partial_t \Lambda_{\geq 3} F) \|_{L^2}  
     \\
     &
     \lesssim
     \A_1 \A_2  \|v[t] \|_{H^{1} \times L^2} + \A_2 \A_0 \A_2 \|v[t] \|_{H^{1} \times L^2}.
\end{aligned}
\]

\noindent \emph{Case (c)}:
The expression \(
\Lambda_{\geq3}( B_{hl}( u_{tt},v_{tt}))   
\)
is estimated using \eqref{eq:p2-utt} to control \(u_{tt}\) in \(L^{\infty}\) . For instance, they can be estimated as follows,
\[
\begin{aligned}
   \| \Lambda_{\geq3}( B_{hl}( u_{tt},v_{tt}))   \|_{L^2 } 
   \lesssim \,
   \| B_{lh}(\Lambda_{\ge 2} u_{tt}, v_{tt}) \|_{L^2 }
   +  \| B_{lh}( u_{tt}, \Lambda_{\ge 2} v_{tt}) \|_{L^2 }
\end{aligned}
\]
For the first term, we combine the paraproduct bounds~\eqref{BMO-shift} and~\eqref{B-CM}
with the estimate for $u_{tt}$ in~\eqref{eq:p2-utt-2} and the decomposition/estimates for
$v_{tt}$ in~\eqref{eq:parvtt-split} and~\eqref{eq:pa-vtt-est}.
\[
\begin{aligned}
     \| B_{hl}( \Lambda_{\geq 2 }u_{tt}, v_{tt})  \|_{L^2} 
     & \lesssim \
     \| B_{hl}( \Lambda_{\geq 2 }u_{tt}, v_{tt}^{(-1)})  \|_{L^2} 
     +\| B_{hl}( \Lambda_{\geq 2 }u_{tt}, v_{tt}^{(0)})  \|_{L^2} 
     \\
     &
     \lesssim \
     \| \Lambda_{\geq 2 } u_{tt} \|_{W^{2,\infty}}
     \| v_{tt}^{(-1)} \|_{H^{-1}}
     + 
      \| \Lambda_{\geq 2 } u_{tt} \|_{W^{1,\infty}}
     \| v_{tt}^{(0)} \|_{L^2 }
     \\
     &
     \lesssim \ \A_0 \A_3 
     \| v[t] \|_{H^1 \times L^2}.
\end{aligned}
\]
The second term is estimated analogously. Using the decomposition
$v_{tt}=v_{tt}^{(-1)}+v_{tt}^{(0)}$, we obtain
\[
\begin{aligned}
      \| B_{hl}( u_{tt}, \Lambda_{\geq 2 } v_{tt})  \|_{L^2}
      & \lesssim 
       \| B_{hl}( u_{tt}, \Lambda_{\geq 2 } v_{tt}^{(-1)})  \|_{L^2}
       + \| B_{hl}( u_{tt}, \Lambda_{\geq 2 } v_{tt}^{(0)})  \|_{L^2}
       \\
        & \lesssim 
       \|  u_{tt}\|_{W^{2,\infty}} \| \Lambda_{\geq 2 } v_{tt}^{(-1)})  \|_{H^{-1}}
       +  \|  u_{tt}\|_{W^{3,\infty}} \| \Lambda_{\geq 2 } v_{tt}^{(0)})  \|_{L^2}
       \\
       & 
         \lesssim \A_0 \A_3 
     \| v[t] \|_{H^1 \times L^2}.
\end{aligned}
\]

\smallskip

\noindent \emph{Case (d)}:

\medskip
We estimate these terms directly by using paraproduct estimates as well as the symbol bound \( B^1_{hl} \in S^{0,1}\) and Moser estimates~\eqref{bmo-hs}:
\[
\begin{aligned}
    \|T_{g^{11}} B_{hl}(\partial_x u_t, \partial_x v_t) \|_{L^2}
    \lesssim 
    \A_0 \A_3
    \| v[t] \|_{H^1 \times L^2}   .
\end{aligned}
\]

The rest of the terms can be estimated similarly. Hence the proof is complete.

\end{proof}

\section{Cubic energy estimates for the paradifferential problem}
\label{s:cubic para ee}

In the previous sections we constructed the normal form transformation
and the paradifferential reformulation of the Klein–Gordon equation.
The goal of the present section is to build modified cubic
energies for the paradifferential Klein–Gordon operator $L^{\mathrm{para}}_{KG}$
and to derive cubic estimates for solutions of the associated
linear problem. These estimates will later be applied to the 
study of both the full equation and the linearized equation.

More precisely, in this section we study the inhomogeneous linear paradifferential equation
\begin{equation}\label{paralin}
  L^{para}_{KG} w = f.
\end{equation}

To explain the construction of the energies, we begin with the case \(s=1\).
We first construct the \emph{normal form energy} by substituting the normal form variables of the linearized equation into the linear energy and retaining only the quadratic and cubic terms. The normal forms variables here can also be viewed as the linearization of the normal form transformation in Proposition~\ref{nfl}.
By construction, this energy is cubic, since the quadratic contributions are removed by the normal form transformations.
We then use the symbol expansions from Proposition~\ref{nfl} to rewrite the resulting cubic terms in a more explicit and convenient form.

Since the equation is quasilinear, the energy must also reflect the underlying quasilinear structure so as to cancel the top-order terms produced when differentiating in time.
The \emph{normal form energy} alone does not yet have this property.
We therefore introduce quasilinear corrections, choosing their coefficients by matching with the standard quasilinear energy and momentum densities.
This yields the desired \emph{modified energy} at the level \(s=1\), see \eqref{epara}.

Next, we extend the construction to an arbitrary regularity index \(s\), as carried out in Proposition~\ref{p:para-s}.
The key point is to conjugate the homogeneous paradifferential equation \(L_{KG}^{\mathrm{para}} w = 0\) by \(\jD^{\,s}\), which produces a new paradifferential equation together with lower-order commutator terms.
This falls into the class of paradifferential equations treated in Section~\ref{s:normal-forms extras}, and we use the results there to obtain the required invertibility statement and perturbative bounds.

Finally, we incorporate the source term \(f\) in the inhomogeneous paradifferential equations \(L^{para}_{KG} =f \).
The extra source terms appear when differentiating in time the general energy
at the level \(s\).
We carefully examine the resulting extra terms and thereby prove the main result of this section, Theorem~\ref{paraE-inhom}, which is stated next:

\begin{theorem}\label{paraE-inhom}
    For each  index $s \geq 1$, there exists a modified cubic energy 
    $E^{s,para}(w,w_t)$ 
for the 
    inhomogeneous  paradifferential equation~\eqref{paralin} such that we have
    \begin{enumerate}[label=(\roman*)]
\item the norm  equivalence
\begin{equation}
\label{t:e-equiv}
  E^{s,para}(w,w_t)  \approx_{\A_2}   \Vert (w,w_t)\Vert _{H^{s}\times H^{s-1}}^2,
  \end{equation}
  
\item the cubic energy estimate
\begin{equation}\label{dt-Es}
|\frac{d}{dt}   E^{s,para}(w,w_t)|\lesssim_{\A_2}  
\A_0
\A_3\,  E^{s,para}(w,w_t) + \|f\|_{H^{s-1}} \Vert (w,w_t)\Vert _{H^{s}\times H^{s-1}}.
\end{equation}

\end{enumerate}
\end{theorem}

\begin{remark}
Here, we view the energy as a bilinear form acting on \((w, w_t) \). It then becomes clear  that the energy depends on the full Cauchy data at a given time, a point that will be important later. Moreover, the bilinear formulation makes it easier to modify the energy estimates in order to incorporate the effect of source terms.
\end{remark}

The main step in the proof of the above theorem
is to construct these energies for the homogeneous paradifferential equation without the source terms, for \( s \ge 1\). This is separately stated below in
Proposition~\ref{paraE}. We will separate its proof into two parts. First, we construct the modified energy for the case $s=1$, see Proposition~\ref{p:para-1}; this represents 
the principal part. We then extend it to all $s$ in Proposition~\ref{p:para-s} by a conjugation argument.
The last step is to incorporate the contribution of the source term $f$. This is carried out in the proof of Proposition~\ref{para-source} below.

\begin{proposition} \label{paraE}
For each regularity index $s \geq 1$, there exists a modified cubic energy 
$E^{s,para}(w,w_t)$ for the homogeneous paradifferential equation \eqref{paralin} (with source term $f =0 $) so that we have
\begin{enumerate}[label=(\roman*)]
\item  the norm  equivalence

\begin{equation}\label{ne}
  E^{s,para}(w,w_t)  \approx_{\A_2}   \Vert (w,w_t)\Vert _{H^{s}\times H^{s-1}}^2,
  \end{equation}

\item the cubic energy estimate
\begin{equation}\label{cee}
|\frac{d}{dt}   E^{s,para}(w,w_t)|\lesssim_{\A_2}  \A_0 \A_3\,  E^{s,para}(w,w_t).
\end{equation}
\end{enumerate}

\end{proposition}

For some intuition we now briefly review the energy-momentum tensor 
associated to linear wave equations of the form 
\[
g^{\alpha\beta} \partial_\alpha \partial_\beta w = 0.
\]
In the constant-coefficient case one has the local conservation laws
\[
\partial_\alpha T^{\alpha\beta}(w)=0,
\]
with  the energy-momentum tensor
$T^{\alpha\beta}$ defined by  
\begin{equation}\label{eq:stress-energy}
T_{\alpha\beta} [w]
:= \partial_{\alpha}w\,\partial_{\beta}w
-\frac12\,g_{\alpha\beta}\Big(g^{\mu\nu}\partial_{\mu}w\,\partial_{\nu}w\Big),
\end{equation}
where the indices are raised with respect to the metric $g$, and in particular $\partial^{\alpha}:=g^{\alpha\gamma}\partial_{\gamma}$.
In the variable-coefficient case, this is no longer an exact conservation law, instead we have additional (bounded) source terms
involving derivatives of the coefficients  \( \nabla g\).

In particular, recalling  that we assume $g^{00}=-1$, the energy type density is given by:
\begin{align} \label{eq:e00}
    T^{00} (w) := g^{0\alpha} T_{\alpha 0 } = & \partial_0 w\partial^0 w - \frac{1}{2} \delta^0_0(g^{\alpha\beta}\partial_\alpha\partial_\beta w) 
    \\ \notag
    = & - \frac{1}{2}  w_t^2 - \frac{1}{2}g^{11}w_x^2,
\end{align}
and the momentum type density is 
\begin{align}\label{eq:e01}
    T^{01}(w) : = g^{0\alpha} T_{\alpha 1} = & \partial^0 w\partial_1 w - \frac{1}{2} \delta^0_1(g^{\alpha\beta}\partial_\alpha\partial_\beta w) 
    \\ \notag
    = & w_t w_x -g^{01} w_x^2.
\end{align}

Then it is natural that a general choice for a linear energy density should be 
a linear combination of the two, of the form
\begin{align*}
    E^{1} (w) 
    = &  \int (1+c_0) T^{00} + c_1 T^{01} \, dx,
\end{align*}
with reasonably regular coefficients $c_0$ and $c_1$.
In the context of the paradifferential equation, the above coefficients
$c_0$ and $c_1$, as well as all the coefficients in 
$T^{00}$ and $T^{01}$, should be considered with
the usual paradifferential interpretation. 
Hence, in what follows, when constructing quasilinear corrections to the normal form energy, we will use these densities as a guide for matching the appropriate coefficients.

\smallskip

We begin our analysis with the base case $s = 1$:

\begin{proposition}\label{p:para-1}
    There exists a modified cubic energy $E^{1,para}(w,w_t)$ such that the norm equivalence \eqref{ne} and the cubic energy estimate \eqref{cee} hold.
\end{proposition}
\begin{proof}

The computation below is based on the normal form for the linearized equation~\eqref{eq-lin}, which can be obtained by linearizing the normal form transformation for the full quadratic equation in Lemma~\ref{nfl}.
Concretely, the linearized normal form transform can be written
\[
w_{\NF} : = w+ A_{lh}(u,w) + B_{lh} (u_t,w_t) + C_{lh}(u_t,w) +C_{hl}(w_t,u).
\]
Substituting the normal form transformation into the linear energy and keeping only the quadratic and cubic contributions, we arrive at the paradifferential normal form energy that serves as our starting point
\begin{equation}\label{NF energy}
\begin{aligned} 
    E^{1,para}_{\NF} (w,w_t) := & \ \Lambda_{\leq 3} E(w_{\NF})
 \\   
    =  & \ \Lambda_{\leq 3} \biggr( E( w+ A_{lh}(u,w) + B_{lh} (u_t,w_t) + C_{lh}(u_t,w) +C_{hl}(w_t,u)) \biggl)
    \\ 
    = & \  \frac{1}{2} \!
    \int  \! w_t^2 + w_x^2 + m w^2 + 4 w_t \partial_t A_{lh}(u,w) + 4w_x \partial_x A_{lh}(u,w) + 4m w A_{lh}(u,w)  \!\!\!\!
    \\ 
    & \ + 4w_t \partial_tB_{lh}(u_t,w_t) 
     + 4 w_x \partial_xB_{lh}(u_t,w_t) + 4 m w B_{lh}(u_t,w_t) 
\\ 
& + 2 w_t \partial_t C_{lh}(u_t,w) + 2 w_x \partial_x C_{lh}(u_t,w) 
+ 2m w C_{lh}(u_t,w)  
 \\ 
 &+ 2 w_t \partial_t C_{hl}(w_t,u) + 2 w_x \partial_x C_{hl}(w_t,u) 
+ 2m w C_{hl}(w_t,u)
   \, dx    .
\end{aligned}
\end{equation}

By construction, the time derivative of the above normal form energy contains no cubic contributions,
\begin{equation}
\label{dt-enf}
\begin{aligned}
    \Lambda_{\leq 3} \frac{d}{dt}E^{1,para}_{\NF}(w,w_t)  
    &\  = 
     \ \Lambda_{\leq 3} \frac{d}{dt}E(w_{\NF}) \\
    &\ = \ \Lambda_{\leq 3} \int w_{\NF} L_{KG} w_{\NF} \,dx \\
    & \ = \int w_{\NF} \Lambda_{\leq 2}( L_{KG} w_{\NF} ) \,dx
      =0.
\end{aligned}
\end{equation}

Next, we further expand the energy. Using Lemma~\ref{nfl}, we decompose
$A_{lh}(u,w)$, $B_{lh}(u_t,w_t)$, $C_{lh}(u_t,w)$, and $C_{hl}(w_t,u)$ as follows: 

\[
\left\{
\begin{aligned}
    &A_{lh}(u,w) =  T_{a_0(D) u} Dw + T_{a_1(D) u} w + \widetilde{A}_{lh}(u,w), \\
   & B_{lh}(u_t,w_t) =  T_{b_0(D) u_t} w_t + T_{b_1(D)u_t} D^{-1} w_t + \widetilde{B}_{lh}(u_t,w_t) , \\
  &  C_{lh} (u_t,w) =  T_{c_0^1(D) u_t} D w + T_{c_1^1(D) u_t} w 
    + \widetilde{C}_{lh}(u_t,w) , 
    \\
  &  C_{hl} (w_t,u) =  T_{c_0^2(D) u} w_t + T_{c_1^2(D) u} D^{-1} w_t 
    + \widetilde{C}_{hl}(w_t,u) , 
\end{aligned}
\right.
\]
where we carefully note that all terms in the above expansions are real valued functions. Here $\Tilde A_{lh}(u,w)$, $\Tilde B_{lh}(u_t,w_t)$, $\Tilde C_{lh}(u_t,w)$, and
$\Tilde C_{hl}(w_t,u)$ denote remainder operators whose bilinear symbols satisfy, respectively, in the low--high regions,
\[
\begin{aligned}
\widetilde a_{lh} &= O\big((1+|\xi_1|^4)\xi_2^{-1}\big),\qquad
\widetilde b_{lh} &= O\big((1+|\xi_1|^3)\xi_2^{-2}\big),\\
\widetilde c_{lh} &= O\big((1+|\xi_1|^3)\xi_2^{-1}\big),\qquad
\widetilde c_{hl} &= O\big((1+|\xi_2|^4)\xi_1^{-2}\big),
\end{aligned}
\]
or in terms of symbol classes,
\[
\ \widetilde a_{lh} \in S^{4,-1},
\quad \widetilde b_{lh} \in S^{3,-2},
\quad \widetilde a_{lh} \in S^{3,-1},
\quad \widetilde b_{lh} \in S^{4,-2}.
\]
     We begin by extracting the paradifferential coefficient of the $w_t\cdot w_t$ terms. We start  with the term
 
 \[
4\int \Big( w_t A_{lh}(u,w_t) + w_t B_{lh}(u_{xx}-mu,w_t) \Big)\,dx .
\]
Using the decompositions in Lemma~\ref{nfl}  and substituting $D=-i \partial_x$ we obtain
\[
\begin{aligned}
&4\int \Big( -T_{ia_0(D)u}\partial_x w_t \cdot w_t + T_{a_1(D)u}w_t\cdot w_t
+ T_{b_0(D)(u_{xx}-mu)}w_t\cdot w_t 
\\
&\qquad \ + T_{i b_1(D)(u_{xx}-mu)}\partial_x^{-1}w_t\cdot w_t \Big)\,dx 
 + 4\int \Big( w_t \widetilde A_{lh}(u,w_t) + w_t \widetilde B_{lh}(u_{xx}-mu,w_t) \Big)\,dx .
\end{aligned}
\]
Here the para-coefficient $ia_0(D)u$ is real
valued, so  by integrating by parts (or equivalently, using the self-adjoint property of $T_{ia_0(D)u}$) we may rewrite
\[
\int - T_{ia_0(D)u}\partial_x w_t\cdot w_t\,dx
= \frac12\int T_{i \partial_x a_0(D)u}w_t\cdot w_t\,dx
= - \frac12\int T_{D a_0(D)u}w_t\cdot w_t\,dx,
\]
and hence the above contribution can be expressed as a bounded bilinear form in $(w_t,w_t)$ plus lower-order remainders.

Next we extract the $w_t\cdot w_t$ contribution coming from $C_{lh}(u_t,w)$ and $C_{hl}(w_t,u)$:

\[
2\int \Big( w_t C_{lh}(u_t,w_t) + w_t C_{hl}(w_t,u_t) \Big)\,dx .
\]
Expanding again by Lemma~\ref{nfl} yields
\[
\begin{aligned}
&2\int w_t\Big( -T_{ic_0^1(D)u_t}\partial_x w_t + T_{c_1^1(D)u_t}w_t + \widetilde C_{lh}(u_t,w_t)\Big)\,dx 
\\
&
+2\int w_t\Big( T_{c_0^2(D)u_t}w_t + T_{i c_1^2(D)u_t}\partial_x^{-1}w_t + \widetilde C_{hl}(w_t,u_t)\Big)\,dx .
\end{aligned}
\]
In the first term the paracoefficient $ic_0^1(D)u_t$ is real valued and the corresponding paraproduct is selfadjoint, therefore we can 
integrate by parts as before, obtaining
\[
2\int - w_t T_{ic_0^1(D)u_t}\partial_x w_t \, dx
= - \int w_t T_{Dc_0^1(D)u_t} w_t \, dx.
\]

Collecting  all the  $w_t\cdot w_t$ contributions, we may summarize the above computations as
\[
\int T_{\kappa_0(u)}w_t\cdot w_t\,dx + I_1,
\]
where the real paracoefficient $\kappa_0(u)$ is given by
\[
\begin{aligned}
    \kappa_0(u)
&:= 4\Big(-\frac12 a_0(D)Du_t + a_1(D)u + b_0(D)(u_{xx}-mu)\Big)
\\
&\quad \ + 2\Big(-\frac{1}{2} Dc_0^1(D)u_{t} + c_1^1(D)u_t + c_0^2(D)u_t\Big),
\end{aligned}
\]
and $I_1$ collects the remainder terms, namely
\[
\begin{aligned}
I_1 :=\;& 4\int \Big( w_t\widetilde A_{lh}(u,w_t) + w_t\widetilde B_{lh}(u_{xx}-mu,w_t)
+ T_{b_1(D)(u_{xx}-mu)}D^{-1}w_t\cdot w_t \Big)\,dx \\
&\quad + 2\int \Big( w_t\widetilde C_{lh}(u_t,w_t) + w_t\big(T_{c_1^2(D)u_t}D^{-1}w_t + \widetilde C_{hl}(w_t,u_t)\big)\Big)\,dx,
\end{aligned}
\]
which has the form 
\[
I_1 = \int w_t L_{lh}( \partial^{\le 3}_x \partial u, \partial_x^{-1} w_t)\, dx.
\]

\bigskip

Next, we extract the paradifferential coefficient of the $w_t\cdot w_x$ terms. 
To this end, we consider the contribution
\[
\begin{aligned}
  &  4\int \Big(
w_t A_{lh}(u_t,w)
+ w_t B_{lh}(u_t,w_{xx}-mw)
+ w_x B_{lh}(u_{tx},w_t) 
\\
&
\quad \quad + w_x B_{lh}(u_t,w_{tx}) 
+ m w\, B_{lh}(u_t,w_t)
\Big)\,dx .
\end{aligned}
\]
Applying the expansions from Lemma~\ref{nfl} yields
 
  \[
 \begin{aligned}
  = & 
     4 \int T_{a_0(D) u_t} Dw \cdot w_t 
     + T_{a_1(D) u_t} w \cdot w_t 
     + T_{b_0(D) u_t} (w_{xx}-mw) \cdot w_t \\
     &
    \qquad  + T_{b_1(D) u_t}D^{-1}(w_{xx} -mw) \cdot w_t  
     + T_{b_0(D) u_{tx}} w_t \cdot w_x  
      + T_{b_1(D) u_{tx}}D^{-1} w_t \cdot w_x  
      \\
      &
      \qquad   + T_{b_0(D) u_t} w_{tx} \cdot w_x 
      + T_{b_1(D) u_t}D^{-1} w_{tx} \cdot w_x  \\
    &
     \qquad   + mT_{b_0(D) u_t} w_t \cdot w 
      + mT_{b_1(D) u_t} D^{-1} w_t \cdot w
    \, dx\\
    & \qquad   + 4 \int  w_t (\widetilde{A}_{lh}(u_t,w) 
      + \widetilde{B}_{lh}(u_t,w_{xx}-mw) )
      + w_x (\widetilde{B}_{lh}(u_{tx},w_t) 
      \\
      &
      \qquad   + \widetilde{B}_{lh} (u_t,w_{tx})) 
      + m w \widetilde{B}_{lh}(u_t,w_t)
    \, dx . \\
 \end{aligned}
 \]
 
 We now compute the contribution arising from the \(C_{lh}(u_t,w)\) and \( C_{hl}(w_t,u)\), namely
\[
2\int \Big(
w_t C_{lh}(u_{xx}-mu,w)
+ w_t C_{hl}(w_{xx}-mu,u)
+ w_x\partial_x C_{hl}(w_t,u)
+ m w\, C_{hl}(w_t,u)
\Big)\,dx .
\]
Expanding $C_{lh}$ and $C_{hl}$ according to Lemma~\ref{nfl}, we obtain
 \[
 \begin{aligned}
 & 2 \int w_t \bigr( T_{c_0^1(D)(u_{xx} - mu)} D w + T_{c_1^1(D) (u_{xx} -mu)} w 
    + \widetilde{C}_{lh}(u_{xx} -mu ,w)  \bigl)  
    \\
    &
    \qquad   + w_t \bigr( T_{c_0^2(D) u} (w_{xx} - mw) + T_{c_1^2(D) u} D^{-1} (w_{xx} - mw)
    + \widetilde{C}_{hl}(w_{xx} - mw,u)  \bigl) 
     \\
     & 
     \qquad   + w_x \partial_x \bigr( T_{c_0^2(D) u} w_t + T_{c_1^2(D) u} D^{-1} w_t 
    + \widetilde{C}_{hl}(w_t,u)    \bigl)
    \\
    & 
     \qquad   + 
      m  w \bigr( T_{c_0^2(D) u} w_t + T_{c_1^2(D) u} D^{-1} w_t 
    + \widetilde{C}_{hl}(w_t,u)  \bigl) \, dx.
 \end{aligned}
 \]
Collecting the principal $w_t\cdot w_x$ terms, we can write the above contribution as
 \[
     \int T_{\kappa_1(u)} w_t \cdot w_x \, dx 
     + I_2,
 \]
 where
 \[
 \begin{aligned}
     \kappa_1(u) := &  4\bigr(- i a_0(D)u_t 
     + 2 i b_1(D)u_t \bigl) 
     + 2\bigr( 
     - i c_0^1(D) (u_{xx}-mu)
     + 2i c_1^2(D) u 
     \bigl) ,
 \end{aligned}
 \]
 and $I_2$ collects the remaining terms displayed above:
 \[
 \begin{aligned}
   I_2 : = & \
   4 \int T_{a_1(D) u_t} w \cdot w_t -m T_{b_0(D) u_t} w \cdot w_t
   -m T_{b_1(D) u_t} D^{-1}(w) \cdot w_t
   \\
   &
   \qquad  + T_{b_1(D) u_{tx}}D^{-1} w_t \cdot w_x  
    + mT_{b_0(D) u_t} w_t \cdot w 
    + mT_{b_1(D) u_t} D^{-1} w_t \cdot w \, dx \\
   &
    \qquad   + 2\int  w_t (\widetilde{A}_{lh}(u_t,w) 
      + \widetilde{B}_{lh}(u_t,w_{xx}-mw) )
      + w_x (\widetilde{B}_{lh}(u_{tx},w_t) 
      + \widetilde{B}_{lh} (u_t,w_{tx}))  
      \\
      &
     \qquad   + m w \widetilde{B}_{lh}(u_t,w_t)  \, dx  
    +  2 \int w_t \bigr(  T_{c_1^1(D) (u_{xx} -mu)} w 
    + \widetilde{C}_{lh}(u_{xx} -mu ,w)  \bigl) 
    \\
    &
     \qquad   + w_t \bigr( T_{c_0^2(D) u} ( - mw) 
     + T_{c_1^2(D) u} D^{-1} ( - mw)
    + \widetilde{C}_{hl}(w_{xx} - mw,u)  \bigl) 
    \\
    &
    \qquad    + w_x \partial_x  \widetilde{C}_{hl}(w_t,u)   
      + m  w \bigr( T_{c_0^2(D) u} w_t + T_{c_1^2(D) u} D^{-1} w_t 
    + \widetilde{C}_{hl}(w_t,u)  \bigl) \, dx,
 \end{aligned}
 \]
  which has the form
 \[
 I_2 = \int w_t L_{lh} ( \partial^{\le 3}_x \partial u, w)  \, dx.
 \]
Finally, we extract the paradifferential coefficient of the $w_x\cdot w_x$ terms. We begin with 
\[
4\int \Big( w_x\partial_x A_{lh}(u,w) + mw\,A_{lh}(u,w)\Big)\,dx,
\]
and expand $A_{lh}$ using Lemma~\ref{nfl}, which yields the contributions displayed above,
together with remainder terms involving $\widetilde A_{lh}$:
  \[
  \begin{aligned}
      & 4 \int   T_{a_0(D) u_x} Dw\cdot w_x 
     + T_{a_0(D) u} D w_x \cdot w_x 
     + T_{a_1(D)u_x} w_x \cdot w 
     + T_{a_1(D)u} w_x \cdot w_x \\
     &
    \qquad  \qquad    + m T_{a_0(D) u} Dw \cdot w 
     + m T_{a_1(D)u} w \cdot w \, dx \\
     & 
      + 4 \int w_x \widetilde{A}_{lh}(u_x,w) 
      + w_x \widetilde{A}_{lh}(u,w_x)
      + m w \widetilde{A}_{lh}(u,w) \, dx .\\ 
  \end{aligned}
  \]
     We also include the contributions coming from $C_{lh}(u_t,w)$:
   \[
\begin{aligned}
   2  \int w_x \partial_x C_{lh}(u_t,w) + m w C_{lh}(u_t,w)  \, dx 
  &
  = 2 \int w_x \partial_x \bigr(  T_{c_0^1(D) u_t} D w + T_{c_1^1(D) u_t} w 
    + \widetilde{C}_{lh}(u_t,w) \bigl)   \\
    &
    \qquad + m w \bigr(   T_{c_0^1(D) u_t} D w + T_{c_1^1(D) u_t} w 
    + \widetilde{C}_{lh}(u_t,w)\bigl) 
    \, dx.
\end{aligned}
\]
Collecting the principal $w_x\cdot w_x$ terms, we may summarize the above computations as 
\[
\int T_{\kappa_2(u)} w_x \cdot w_x \, dx 
      + I_3 ,
\]
  where 
  \[
  \begin{aligned}
    \kappa_2(u) :& =  4\bigr(-\frac{i}{2} a_0(D) u_x +a_1(D) u \bigl) 
    + 2 \bigr( - \frac{i}{2} c_0^1(D) \partial_x u_t + c_1^1(D)  u_t
    \bigl) ,
\end{aligned}
   \]
and $I_3$ collects the remaining lower-order terms listed above.
\begin{align*}
    I_3 : = & 4 \int  w_x \widetilde{A}_{lh}(u_x,w) 
      + w_x \widetilde{A}_{lh}(u,w_x)
      + m w \widetilde{A}_{lh}(u,w)  \\
      & 
     \qquad   + m T_{a_0(D) u} Dw \cdot w 
      + m T_{a_1(D)u} w \cdot w   \, dx 
      \\
      &
      + 2 \int w_x \partial_x \widetilde{C}_{lh}(u_t,w)  + m w \bigr(   T_{c_0^1(D) u_t} D w + T_{c_1^1(D) u_t} w 
    + \widetilde{C}_{lh}(u_t,w)\bigl)  \,dx,
\end{align*}
which has the form 
\[
I_3 = \int \partial_x^{\leq 1}w L_{lh} (\partial^{\le 3}_x \partial u, w) \, dx.
\]
At this stage, we collect all lower-order contributions involving at most one derivative of $w$ and define
\[
E^{1,para}_{\NF,lot} :=I_1+I_2+I_3,
\]
and the real-valued paracoefficients
\begin{equation} \label{eq:para-coef}
    \begin{aligned}
        \kappa_0(u) 
     = & \ 4( -\frac{1}{2}a_0(D) D u +a_1(D) u + b_0(D)( u_{xx} -mu) ) \\
       &
       \qquad   + 2 \bigr(  -\frac{1}{2} c_0^1(D) D u_t + c_1^1(D)u_t +c_0^2(D) u_t \bigl),
   \\
   \kappa_1(u) 
   = &  \ 4\bigr(- i a_0(D)u_t 
     + 2 i b_1(D)u_t \bigl) 
     + 2\bigr( 
     - i c_0^1(D) (u_{xx}-mu)
     + 2i c_1^2(D) u 
     \bigl),
    \\
   \kappa_2(u)
    = & \ 4\bigr(-\frac{1}{2} a_0(D)D u +a_1(D) u \bigl) 
    + 2 \bigr( - \frac{1}{2} c_0^1(D) D u_t + c_1^1(D)  u_t
    \bigl) .
    \end{aligned}
\end{equation}

Combining the above computations, we obtain the following representation for the paradifferential normal form energy:
\begin{equation}
  E^{1,para}_{N\!F}(w,w_t) =   E^{1,para}_{N\!F,main}(w,w_t) +   E^{1,para}_{\NF,lot}(w,w_t) ,
\end{equation}
where the leading part is
\begin{equation}
\begin{aligned}
  E^{1,para}_{\NF,main}(w,w_t)  :=  \frac{1}{2} \! \int\! w_t^2 + w_x^2 + m w^2 
    +   T_{\kappa_0(u)}  w_t \cdot w_t 
    + T_{\kappa_1(u)}  w_t \cdot w_x +  T_{\kappa_2(u)}  w_x \cdot w_x 
   \,  dx\! ,
\end{aligned}
\end{equation}
while the lower-order terms arise from low--high paraproducts and involve at most one derivative on $w$, i.e. they are cubic expressions which have the form
\begin{equation}\label{E-para-lot}
\begin{aligned}
E^{1,para}_{\NF,lot}(w,w_t) = \int &L_{lhh}( \partial^{\le 3}_x \partial u ,w_t,w)
+ L_{lhh}( \partial^{\le 3}_x \partial u ,w_t, \partial^{-1}w_t) \\
&+ L_{lhh}(\partial^{\le 3}_x \partial u ,\partial^{\leq 1}_xw,w)\, dx.
\end{aligned}
\end{equation}
Ideally, at this point one might hope to show that the norm equivalence 
\eqref{ne} and the cubic energy estimate \eqref{cee} hold for 
$E^{1,para}_{\NF}$. While the norm equivalence is not a problem, 
the cubic energy estimate \eqref{cee} does not hold in general; this can be seen as a consequence of the
fact that our problem is quasilinear.
To divide and conquer, we take advantage of \eqref{dt-enf} to write
\[
 \frac{d}{dt} E^{1,para}_{\NF}(w,w_t) = \Lambda_{\geq 4} \frac{d}{dt} E^{1,para}_{\NF}(w,w_t) = 
 \Lambda_{\geq 4} \frac{d}{dt}
 E^{1,para}_{\NF,main}(w,w_t) +  \Lambda_{\geq 4} \frac{d}{dt} E^{1,para}_{\NF,lot}(w,w_t).
\]
Here the last term can be estimated directly, so we can leave 
$E^{1,para}_{\NF,lot}$ as is. 
However, differentiating \( E^{1,para}_{\NF}\) in time produces unbounded quasilinear terms coming from both energy type density $w_t^2$ and momentum type density $w_t w_x$. These terms will carry two derivatives on \( w\) and cannot be eliminated by integration by parts. The remaining task will then be 
to find a quartic and higher correction to $E^{1,para}_{\NF,main}(w)$
so that the next to last term above is also favorable.

To dispense with the contribution of the lower order terms it suffices to show that
\begin{lemma}
\label{l:lot}
The lower order part  $E^{1,para}_{\NF,lot}$ of the energy satisfies the 
following bounds:
\begin{equation} \label{eq:e-lot-norm}
  |E^{1,para}_{\NF,lot}(w,w_t)|  \lesssim 
  \A_2
  \Vert (w,w_t)\Vert _{H^{1}\times L^2}^2,
  \end{equation}
respectively
  \begin{equation} \label{eq:e-lot-time}
|\Lambda_{\geq 4} \frac{d}{dt}   E^{1,para}_{\NF,lot}(w,w_t)|\lesssim 
\A_0 \A_3
\,  \| (w,w_t)\|_{H^1 \times L^2}.
\end{equation}

\end{lemma}

\begin{proof}
First, we prove the estimate~\eqref{eq:e-lot-norm}. By the definition in \eqref{E-para-lot}, it suffices to verify one of the terms; for instance,
\[
\begin{aligned}
    \left| \int  L_{lhh}(\partial^{\le 3}_x \partial u ,w_t, \partial^{-1}w_t)\, dx \right| 
    \lesssim
    \A_2 \Vert (w,w_t)\Vert _{H^{1}\times L^2}^2.
\end{aligned}
\]
 First, we use the  H\"older's inequality:
\[
\begin{aligned}
     \left|\int  L_{lhh}( \partial^{\le 3}_x \partial u ,w_t, \partial^{-1}w_t)\, dx \right| 
     & = \bigg| \int w_t L_{lh} (\partial^{\le 3}_x \partial u , \partial^{-1}w_t) \, dx \bigg| 
     \\
     & \lesssim 
     \| w_t \|_{L^2}  
     \|L_{lh} (\partial^{\le 3}_x \partial u , \partial^{-1}w_t) \|_{L^2}
     \\
     &
     \lesssim \A_2 \Vert (w,w_t)\Vert _{H^{1}\times L^2}^2.
\end{aligned}
\]
As these are the lower order terms, we have room in \( w\) to move one derivative from the low-frequency factor \( u\) to \(w\), using the paraproduct estimates~\eqref{CM}.

\bigskip

Then we estimates the time derivative of the lower order terms to prove inequality~\eqref{eq:e-lot-time}. Again, we take \(L_{lhh}(\partial_x^{\leq 3} \partial u,w_t, \partial^{-1}w_t)\) as an example,
\[
\begin{aligned}
  \Lambda_{\geq 4} \frac{d}{dt} L_{lhh}( \partial^{\le 3}_x \partial u,w_t, \partial^{-1}w_t)  =
    &  \ L_{lhh}(  \partial^{\le 3}_x  \Lambda_{\ge 2}u_{tt},w_t, \partial^{-1}w_t)  
    + L_{lhh}(  \partial^{\le 3}_x  \partial u , \Lambda_{\ge 2}w_{tt}, \partial^{-1}w_{t})  
    \\
    &\quad \ + L_{lhh}(  \partial^{\le 3}_x  \partial u ,w_t, \partial^{-1}\Lambda_{\ge 2}w_{tt})  .  
\end{aligned}
\]

We first use H\"older's inequality together with the paraproduct estimates~\eqref{CM} and the Moser estimates~\eqref{bmo-hs}. As before, this allows us to transfer one derivative from the low-frequency factor u to w. We also invoke~\eqref{eq:p2-utt-2} to control \(\Lambda_{\ge 2} u_{tt}\) in \( W^{2,\infty}\).
\[
\begin{aligned}
\left| \int L_{lhh}(\partial^{\le 3}_x \Lambda_{\ge 2}u_{tt},w_t, \partial^{-1}w_t) \, dx  \right| 
& 
= \ \bigg| \int w_t L_{lh} (\partial^{\le 3}_x \Lambda_{\ge 2}u_{tt},\partial^{-1}w_t)\, dx\bigg| \\
&
 \lesssim \
 \| w_t \|_{L^2}  \| L_{lh} (\partial^{\le 3}_x\Lambda_{\ge 2}u_{tt},\partial^{-1}w_t) \|_{L^2}
 \\
 &
 \lesssim \
  \| w_t \|_{L^2}  \| \Lambda_{\ge 2}u_{tt}\|_{W^{2,\infty }}
  \| w_t \|_{L^2}
     \\ 
    &
    \lesssim \
\A_0
\A_3
 \| w[t] \|_{H^1 \times L^2}^2.  
\end{aligned}
\]

\smallskip  

For the second term, we begin by introducing a decomposition for \(w_{tt}\) into two components $w_{tt}^{(-1)}$ and $w_{tt}^{(0)}$,
where the former will be estimated in $H^{-1}$ and the latter in $L^2$:
\[
\begin{aligned}
    w_{tt}^{(-1)} & := \  \partial_x \big( 2 T_{g^{01}}   w_t 
+
T_{g^{11}} \partial_x w \big) ,
\\
w_{tt}^{(0)}  & : = \  T_{\Tilde{F}^{\gamma,lin}} \Tilde{\partial}_{\gamma} w - mw - 2 T_{\partial_x  g^{01}}   w_t 
- T_{\partial_x g^{11}} \partial_x w .
\end{aligned}
\] 
In particular, by the paraproduct estimates~\eqref{CM}, we have
\[
\begin{aligned}
    \| \Lambda_{\ge 2} w_{tt}^{(-1)} \|_{H^{-1}}
    & 
    \lesssim \ 
    \A_0 \| w[t] \|_{H^1 \times L^2 },
    \\
    \| \Lambda_{\ge 2}w_{tt}^{(0)} \|_{L^2 }
     & 
    \lesssim \ 
    \A_1  \| w[t] \|_{H^1 \times L^2 }.
\end{aligned}
\]

Using the above decomposition and duality, we obtain
\[
\begin{aligned}
    & \bigg| \int \Lambda_{\ge 2}w_{tt}  L_{lh} ( \partial^{\le 3}_x  \partial   u,\partial^{-1}w_t)\, dx\bigg|  
    \\
    &\lesssim  \
       \bigg| \int \Lambda_{\ge 2}w_{tt}^{(-1)}  L_{lh} ( \partial^{\le 3}_x  \partial   u,\partial^{-1}w_t)\, dx\bigg| 
       + 
          \bigg| \int \Lambda_{\ge 2}w_{tt}^{(0)}  L_{lh} ( \partial^{\le 3}_x  \partial   u,\partial^{-1}w_t)\, dx\bigg| 
    \\
    & 
    \lesssim  \ 
    \| \Lambda_{\ge 2}w_{tt}^{(-1)} \|_{H^{-1}} \| L_{lh} ( \partial^{\le 3}_x  \partial   u,\partial^{-1}w_t) \|_{H^1}
    + 
    \|\Lambda_{\ge 2}w_{tt}^{(0)}  \|_{L^2 }
    \| L_{lh} ( \partial^{\le 3}_x  \partial   u,\partial^{-1}w_t) \|_{L^2}
    \\
    & 
    \lesssim \ 
    \A_0 \A_3 \|  w[t] \|_{H^1 \times L^2}^2
    + \A_1 \A_2 \|  w[t] \|_{H^1 \times L^2}^2.
\end{aligned}
\]
 The rest of the terms are similar; we omit the details.
\end{proof}

It remains to resolve the problem of modifying the leading 
part $E^{s,para}_{\NF,main}$, by adding a correction which is at least quartic.

We choose the corrections so as to match the coefficients with the quasilinear energy-momentum tensor densities in \eqref{eq:e00} and \eqref{eq:e01}, which we recall here in paradifferential form:

\[
T^{00} = \frac{1}{2} w_t^2 
+ \frac{1}{2} (w_x^2 + 
w_x \cdot T_{g^{11}} w_x ), \qquad 
T^{01} =  
  w_t w_x 
   -  w_x \cdot T_{g^{01}} w_x.
\]
To compare $E^{s,para}_{\NF,main}$ with the above expressions, it is useful to compare 
the paracoefficients $\kappa_0$ and $\kappa_2$. Recalling \eqref{lead-symbol}, we obtain
\begin{equation} \label{coefficient}
    \kappa_2 - \kappa_0 = \frac{1}{2} g^{11}_{p_\gamma} \partial_\gamma u + \frac{1}{2} g_u^{11} u 
 = \frac{1}{2}\Lambda_1 (g^{11}).
\end{equation}

This allows us to view the main term of the paradifferential normal form energy as a linear combination, up to cubic level, of the above densities for the energy and momentum, 
\begin{align*}
    E^{1,para}_{\NF} (w, w_t)  &  = \frac{1}{2}\int \Lambda_{\leq 3} 
    \bigr(
      w_t \cdot T_{1+ \kappa_0} w_t 
      + T_{1+ \kappa_0} w_x \cdot T_{g^{11}} w_ x 
    \bigl)  
    \\
    &
    +
   \Lambda_{\leq 3} \bigr( w_t \cdot T_{\kappa_1} w_x 
    - T_{\kappa_1} w_x \cdot T_{g^{01}} w_x
    \bigl)
    \, dx 
    + l.o.t. \\  
\end{align*}
The first bracket corresponds to the energy-type part together with its quasilinear corrections, whereas the second bracket corresponds to the momentum-type part and its corrections.

This suggests that the corrected energy can be obtained simply by suppressing the homogeneity cutoffs $\Lambda_{\le 3}$ in the last formula above. Hence we define the corrected leading energy as
\begin{equation}
\begin{aligned} \label{epara}
    E^{1,para}_{main} (w, w_t ) 
    &  = \frac{1}{2}\!\int \!
      w_t \cdot T_{1+ \kappa_0} w_t 
      + T_{1+ \kappa_0} w_x \cdot T_{g^{11}} w_ x
     +  w_t \cdot T_{\kappa_1} w_x 
    - T_{\kappa_1} w_x \cdot T_{g^{01}} w_x
    \, dx , \!\!\!\!\!\!\!\!
\end{aligned}
\end{equation}
and the full corrected energy as 
\begin{equation}
E^{1,para} (w,w_t) := E^{1,para}_{main} (w,w_t) + E^{1,para}_{l.o.t.} (w,w_t) .
\end{equation}

Since our quasilinear corrections only contribute terms of quartic (and higher) order, the corrected energy has, by construction, no cubic contribution in its time derivative; in particular, we still have
\[
\Lambda_{\le 3}\Big(\frac{d}{dt}E^{1,para}\Big)=0.
\]
Hence, it remains to show that $E^{1,para}_{main}$ has the following properties:
\begin{lemma}
\label{l:main}
The leading  part  $E^{1,para}_{main}$ of the energy satisfies the 
following bounds:
\begin{equation}
\label{l:e-equiv}
  |E^{1,para}_{main}(w,w_t) - E^1(w) |  \lesssim 
 \A_2
  \Vert (w,w_t)\Vert _{H^{1}\times L^2}^2,
  \end{equation}
respectively
  \begin{equation}
  \label{l:e-est}
|\Lambda_{\geq 4} \frac{d}{dt}   E^{1,para}_{main}(w,w_t)|\lesssim  \A_0 \A_3\,  E^{1,para}(w).
\end{equation}
\end{lemma}

\begin{proof}

We recall that \(E^1\) represents energy  associated with the corresponding  constant coefficient linear equation.
For the first bound \eqref{l:e-equiv}, by the definition in \eqref{epara}, 
\[
\begin{aligned}
   E^{1,para}_{main}(w, w_t) - E^1(w, w_t)  
  & = \frac{1}{2} \int 
    \bigr( w_t \cdot T_{1+ \kappa_0} w_t  - w_t^2 \bigl)
      + \bigr(  T_{1+ \kappa_0} w_x \cdot T_{g^{11}} w_ x 
      - \Lambda_0(g^{11}) w_x^2
      \bigl)
      \\
      & \qquad \qquad 
     +  w_t \cdot T_{\kappa_1} w_x 
    - T_{\kappa_1} w_x \cdot T_{g^{01}} w_x\, dx .
\end{aligned}
\]
After the subtraction, all the remaining terms contain at least one low-frequency factor depending on \((u,\partial u )\), hence each term is at least cubic. 
We first apply H\"older's inequality, for instance,
\[
\begin{aligned}
  \big| \int  w_t \cdot T_{\kappa_0} w_t \, dx \big|  
  &
  \lesssim \| w_t\|_{L^2} \|  T_{\kappa_0} w_t \|_{L^2} \lesssim \| w_t\|_{L^2}^2 \|  \kappa_0  \|_{L^\infty}
  \\
  &
  \lesssim \ 
   \A_2 \Vert (w,w_t)\Vert _{H^{1}\times L^2}^2.
\end{aligned}
\]

The last estimate follows from using paraproduct estimates~\eqref{CM} together with Moser estimates~\eqref{bmo-hs}
and the properties of \( \kappa_0 \) defined in \eqref{eq:para-coef}.
The remaining terms are treated similarly and are omitted.

For the second bound \eqref{l:e-est}, we first estimate the contribution of the energy-type term. As usual, using the equation for \( w\), we replace \( w_{tt}\) by \( T_{g^{\alpha\beta}} \partial_{\alpha} \partial_{\beta} w\)
with \( (\alpha,\beta) \neq (0,0)\),
 plus  lower-order terms 
 \( T_{\Tilde{F}^{\gamma,lin}}\Tilde{\partial}_{\gamma}w + mw\). We arrive at
 \[
\begin{aligned}
    & 
    \Lambda_{\geq 4}
   \int   
  \partial_t \bigr(  w_t \cdot T_{1+ \kappa_0} w_t 
      + T_{1+ \kappa_0} w_x \cdot T_{g^{11}} w_ x
       \bigl) \, dx
        \\
        &=  \Lambda_{\geq 4} \int 
        2\bigr( \sum_{(\alpha,\beta)\neq (0,0)}  T_{g^{\alpha \beta}}\partial_\alpha\partial_\beta w 
        \bigl)
        \cdot  T_{1+ \kappa_0} w_t 
        +  T_{1+ \kappa_0} w_{tx} \cdot T_{g^{11}} w_{x} + T_{1+ \kappa_0} w_{x} \cdot T_{g^{11}} w_{tx}
        \, dx  
        \\
        &\quad   + \Lambda_{\geq 4} \int   (    2 T_{\Tilde{F}^{\gamma,lin}}\Tilde{\partial}_{\gamma}w + mw) T_{1+\kappa_0} w_t + w_t T_{\partial_t \kappa_0} w_t  + T_{\partial_t \kappa_0} w_x \cdot T_{g^{11}} w_ x  + T_{1+\kappa_0} w_x \cdot T_{\partial_t g^{11}} w_ x \,  dx  .     \end{aligned}
\]
Here we have separated the contributions into those with three derivatives on \(w\), on the second line, and those with at most two, on the third line. We refer to the latter terms, in which at most two derivatives fall on $w$, as \emph{good terms}. These good terms can be estimated directly via  H\"older's inequality and Moser estimates~\eqref{bmo-hs}. For example,
\[
\begin{aligned}
    \left| \int  w_t \cdot T_{
 \Lambda_{\geq 2}(\partial_t \kappa_1 ) } w_t  \, dx \right| 
 \leq \
 &
 \|   T_{\Lambda_{\geq 2}(\partial_t \kappa_1 ) } w_t \|_{L^2}
 \| w_t \|_{L^2} 
 \lesssim \
\A_0
\A_3 .
\end{aligned}
\]

For the terms with three derivatives on \(w\), 
on the other hand, we take advantage of its structure to integrate by parts in order to shift one extra derivative onto the low-frequency factor.
For the $(w_{tx},w_t)$ terms we get 
\[
\Lambda_{\geq 4}\int 2 T_{g^{01}} w_{xt} T_{1+\kappa_0} w_t \, dx 
= \ \Lambda_{\geq 4}\int -  T_{\partial_x g^{01}} w_{t} T_{1+\kappa_0} w_t - T_{g^{01}} w_{t} T_{\partial_x \kappa_0} w_t + [ T_{\kappa_0},T_{g^{01}}] w_{xt} w_t
\, dx,
\]
where the first two terms are good terms 
while commuting paraproducts yields good terms by 
\eqref{para-com}.

Now we compute the \((w_{tx},w_x)\) terms, for which the leading terms cancel with each other by construction.
\[
\begin{aligned}
  & \Lambda_{\geq 4} \int 2 T_{g^{11}} \partial_x^2 w\cdot T_{1+ \kappa_0}w_t + 2 T_{1+ \kappa_0} w_{tx} \cdot T_{g^{11}} w_{x} \, dx \\
  & = \ \Lambda_{\geq 4} \int - 2 T_{\partial_x g^{11}} w_x\cdot T_{1+ \kappa_0}w_t 
  -2 T_{ g^{11}} w_x\cdot T_{\partial_x \kappa_0}w_t  \, dx .
\end{aligned}
\]
Both of these terms are good terms and so they can be treated perturbatively.

\bigskip

We next estimate the momentum-type term. Again, using the equation for \(w\), we replace \( w_{tt}\) by \( T_{g^{\alpha\beta}} \partial_{\alpha} \partial_{\beta} w\)
with \( (\alpha,\beta) \neq (0,0)\),
 plus  lower-order terms 
 \( T_{\Tilde{F}^{\gamma,lin}}\Tilde{\partial}_{\gamma}w + mw\). By the argument above, all lower-order terms in which at most two derivatives fall on \( w\) are good terms. Hence,
\[
\begin{aligned}
   & \Lambda_{\geq 4}   
   \int   
  \partial_t \bigr( w_t \cdot T_{\kappa_1} w_x 
    - T_{\kappa_1} w_x \cdot T_{g^{01}} w_x \bigl)
        \, dx
        \\
        & \quad = \Lambda_{\geq 4}\int
        \sum_{(\alpha,\beta)\neq (0,0)} T_{g^{\alpha\beta}} \partial_\alpha \partial_\beta w \cdot T_{\kappa_1} w_x 
        -   T_{\kappa_1} w_{tx} \cdot T_{g^{01}} w_x 
        - T_{\kappa_1} w_{x} \cdot T_{g^{01}} w_{tx} 
        \bigl)
        \, dx + \text{ good terms} \\
        & \quad = \Lambda_{\geq 4}\int T_{g^{11}} w_{xx} \cdot T_{\kappa_1} w_x  \,dx
        +  \int T_{g^{01}} w_{tx} \cdot  T_{\kappa_1} w_x 
        -   T_{\kappa_1} w_{tx} \cdot T_{g^{01}} w_x  \, dx  + \text{ good terms} .
\end{aligned}
\]
For the terms with three derivatives on \( w\), we exploit the structure and integrate by parts to move the extra derivative onto the low-frequency factor. For \( (w_{xx},w_x)\) terms we obtain:
\[
\begin{aligned}
\Lambda_{\geq 4}\!\!\int T_{g^{11}} w_{xx} \cdot T_{\kappa_1} w_x  \, dx
= \frac{1}{2} \Lambda_{\geq 4} \!\!\int - T_{\partial_x g^{11}} w_{x} \cdot T_{\kappa_1} w_x - T_{g^{11}} w_{x} \cdot T_{\partial_x \kappa_1} w_x + [  T_{g^{11}} , T_{\kappa_1} ] w_{xx}w_x\, dx,
\end{aligned}
\]
where the first two terms are good terms 
while commuting paraproducts yields good terms by \eqref{para-com}.

For the \( (w_{tx},w_x)\) terms, we obtain a commutator structure, 
\[
\begin{aligned}
    & \int T_{g^{01}} w_{tx} \cdot  T_{\kappa_1} w_x 
        -  T_{\kappa_1} w_{tx} \cdot T_{g^{01}} w_x  \, dx 
         =   \int [T_{\kappa_1}  , T_{g^{01}}] w_{tx} w_x  \, dx ,
\end{aligned}
\]
which can be bounded using \eqref{para-com}.
This completes the proof of the lemma.

\end{proof}
The proof of Proposition~\ref{p:para-1} is now concluded.
\end{proof}

To extend the energy estimates to all $s \ge 1$, we will employ a conjugation argument. Carrying out this argument requires the \(s=1 \) case of the
inhomogeneous energy estimates in Theorem~\ref{paraE-inhom}. We therefore establish that case next.
\begin{proof}[Proof of Theorem~\ref{paraE-inhom}, case $s=1$] 
Adding the source terms does not affect the norm equivalence, so it suffices to track the additional contribution coming from the source terms in the time derivative of the energy.
Recall the definition of the energy:
\[
\begin{aligned}
    \frac{d}{dt} E^{1,para}(w,w_t)  = \ \frac{d}{dt} E^{1,para}_{main} (w,w_t) + \frac{d}{dt} E^{1,para}_{lot} (w,w_t).
\end{aligned}
\]
For the main part, using the definition in \eqref{epara}, the new contributions arising from the source term in \( \frac{d}{dt} E^{1,para}_{main} (w,w_t)  \) are
\[
\begin{aligned}  
    \int  w_t \cdot T_{1+\kappa_0} f+ f \cdot T_{\kappa_1} w_x \, dx.
\end{aligned}
\]
By H\"older's inequality and paraproduct estimates~\eqref{CM}, we obtain 
\[
\begin{aligned}
    \big|  \int  w_t \cdot T_{1+\kappa_0} f+ f \cdot T_{\kappa_1} w_x \, dx \big| 
    &
    \lesssim \ \| w_t \|_{L^2} \|  T_{1+\kappa_0} f\|_{L^2} + \| f \|_{L^2} \|  T_{\kappa_1} w_x \|_{L^2} 
    \\
   & 
   \lesssim \ \|f \|_{L^2} \| (w,w_t) \|_{H^1 \times L^2}.
\end{aligned}
\]
Next, we consider the time derivative of the lower-order part in the energy in \( \frac{d}{dt} E^{1,para}_{lot} (w,w_t)  \). Again, we record  the additional contributions coming from the source term $f$:
\[
\begin{aligned}
\int L_{lhh}( \partial^{\le 3}_x \partial  u ,f,w)
+ L_{lhh}(\partial^{\le 3}_x \partial  u,f , \partial^{-1}w_t) 
+ L_{lhh}( \partial^{\le 3}_x \partial  u,w_t , \partial^{-1}f) 
\, dx.
\end{aligned}
\]
Using H\"older's inequality together with paraproduct estimates~\eqref{CM}, these terms satisfy the bound:
\[
\lesssim \ \A_2 \| f\|_{L^2} \| (w,w_t)\|_{H^1 \times L^2}.
\]
This closes the \( s=1\) case.

\end{proof}

Now that we have a cubic $H^1\times L^2$ energy for the paradifferential equation,  our next objective is to identify a suitable energy functional which is equivalent to the $H^{s} \times H^{s-1}$ norm.
One advantage of working at the paradifferential level is that the precise choice of Sobolev exponent $s$ is not essential: the same arguments apply uniformly to $H^{s}\times H^{s-1}$ for all $s$. The result is as follows:

\begin{proposition} \label{p:para-s}
    Let $s \in \mathbb{R}$. Given $w$ solving the homogeneous paradifferential equation \eqref{para-KG}, there exist normalized variables $\Tilde{w}^{s-1}$ solving 
    \[
    L_{KG}^{para} \Tilde{w}^{s-1} = \ \tilde G^{s-1},
    \] 
    such that 
    \begin{equation}\label{bdd-nf}
    \Vert
    (\Tilde{w}^{s-1}, \Tilde{w}_t^{s-1} ) - (\jD^{s-1} w, 
    \jD^{s-1}w_t ) 
    \Vert_{H^{1} \times L^2} 
    \lesssim \
   \A_2
    \Vert w[t] \Vert_{H^{1} \times L^2},
    \end{equation}
    \begin{equation}\label{good-source}
    \Vert \tilde G^{s-1 }\Vert_{L^2}   
    \lesssim \ 
\A_0 \A_3
    \Vert w[t] \Vert_{H^{s} \times H^{s-1}}.
    \end{equation}
  
\end{proposition}

\begin{proof}
To obtain the desired normalized variables we first conjugate the equation with $\jD^{s-1}$. For convenience, we set \( \sigma := s-1\).

We begin with $w^{\sigma} : = \jD^{\sigma} w$, for which we compute a corresponding paradifferential equation, 
\[
L_{KG}^{para} w^{\sigma} =\  G^{\sigma},
\]
where the source term is given by
\[
\begin{aligned}
    G^{\sigma} & = \ L_{KG}^{para} \jD^{\sigma} w \\
    & =\ \jD^{\sigma} L_{KG}^{para}w + [L_{KG}^{para} ,\jD^{\sigma} ] w \\
    & = \ [L_{KG}^{para} ,\jD^{\sigma} ] w \\
    & = \ -
   \jD^{\sigma} [T_{g^{\alpha \beta}} \partial_\alpha \partial_\beta
    +T_{\Tilde{F}^{\gamma,lin}} \Tilde{\partial_{\gamma}}
    ,\jD^{-\sigma}] w^{\sigma} .
\end{aligned}
\] 
We now need to carry out normal form computations in order to remove the quadratic terms, so we further separate the source term $G^{\sigma}$ into quadratic parts and cubic and higher parts
\[
\begin{aligned}
      G^{\sigma,[2]}  = & \ 
      - \Lambda_2(
    \jD^{\sigma} [T_{g^{\alpha \beta}} \partial_\alpha \partial_\beta
    +
    T_{\Tilde{F}^{\gamma,lin}} \Tilde{\partial_{\gamma}}
    ,\jD^{-\sigma}] w^{\sigma} ) ,
    \\
 G^{\sigma,[3]} =&  \ 
    - \Lambda_{\ge 3} \bigr(  
    \jD^{\sigma} [T_{g^{\alpha \beta}} \partial_\alpha \partial_\beta
    + 
    T_{\Tilde{F}^{\gamma,lin}} \Tilde{\partial_{\gamma}}
    ,\jD^{-\sigma}] w^{\sigma} \bigl).
\end{aligned}
\]
We claim that the cubic and higher parts are perturbative using a commutator bound and Moser estimates~\eqref{bmo-hs}, noting that the 
normalization $g^{00}= -1$ eliminates the $(\alpha,\beta) = (0,0)$ term:
\[
\begin{aligned}
\| G^{\sigma,[3]} \|_{L^2} =    & \  \|  \bigr(  
    \jD^{\sigma} [T_{\Lambda_{\geq 2} g^{\alpha \beta}} \partial_\alpha \partial_\beta
    + 
    T_{\Lambda_{\geq 2}\Tilde{F}^{\gamma,lin}} \Tilde{\partial_{\gamma}}
    ,\jD^{-\sigma}] w^{\sigma} \bigl) \|_{L^2}
    \\
    \lesssim & \
    \| \Lambda_{\geq2} (\partial_x g^{\alpha \beta}
    + 
    \partial_x \Tilde{F}^{\gamma,lin}  ) 
    \|_{L^{\infty}}
    \| w[t]\|_{H^{s} \times H^{s-1} }
    \\
    \lesssim & \
  \A_0
  \A_2
    \| w[t]\|_{H^{s} \times H^{s-1} }.
\end{aligned}
\]
 Then it satisfies the requirement for the \(F\)  in the Lemma~\eqref{nfl-invert}.
We now consider the quadratic part. In what follows, all derivatives of \( f(u,\partial u)\) and \( g^{\alpha\beta}(u,\partial u)\) are evaluated at  \((0,0) \).
To facilitate the application of the results in Section~\ref{s:normal-forms extras}, we express it in the form
\[
\begin{aligned}
G^{\sigma,[2]}  = & \ 
      - 
    \jD^{\sigma} [T_{\Lambda_1 g^{\alpha \beta}} \partial_\alpha \partial_\beta
    +
    T_{\Lambda_1 \Tilde{F}^{\gamma,lin}} \Tilde{\partial_{\gamma}}
    ,\jD^{-\sigma}] w^{\sigma} ) 
    \\
    := & \  H_{00}(u_t,w_t) + H_{00}(u_t,w)+
    H_{10}(u,w_t)+ H_{11}(u,w).
\end{aligned}
\]
To describe the four bilinear forms in the last 
expression, for simplicity we consider first the bilinear form
$\jD^{\sigma} [T_{f},\jD^{-\sigma} ] \partial_x w$,
which can be expressed as 
\[
  \jD^{\sigma} [T_{f},\jD^{-\sigma} ] \partial_x w
:= \ Q^{\sigma}(f_x,w),
\]
where the bilinear form $Q^{\sigma}$ has symbol
\begin{equation}\label{eq:s-com}
    q^\sigma(\xi_1,\xi_2) =\ (\langle \xi_1+\xi_2 \rangle^\sigma  \langle \xi_2\rangle^{-\sigma} - 1) \chi_{lh}(\xi_1,\xi_2+\frac{\xi_1}2) \frac{\xi_2}{\xi_1} \in S^0.
\end{equation}

Using the notation introduced above, we may write these four bilinear forms as follows:
\[
\begin{aligned}
    H_{00}(u_t,w_t) 
    &:= \ Q^{\sigma}(  \partial_x ( g^{11}_{u_t} u_t ) ,w_t) 
    + Q^{\sigma} (  \partial_x \bigr( 2 g^{01}_{u_x} \partial_x u_t + f_{u_x u_t} u_t \bigl) , \partial^{-1}w_t),
    \\
    H_{10}(u,w_t) 
    &:=  \ Q^{\sigma}(  \partial_x (2  (g^{01}_{u} u + g^{01}_{u_x} u_x) ) ,w_t) 
    +  Q^{\sigma}(  \partial_x (g^{11}_{u_t}\partial_x^2u + f_{u_t u} u ) , \partial^{-1}w_t) ,
    \\
    H_{01}(u_t,w)
    &:= \ Q^{\sigma}(  \partial_x ( g^{11}_{u_t} u_t ) , w_x) 
    +  Q^{\sigma}(  \partial_x ( 2 g^{01}_{u_x} \partial_x u_t + f_{u_x u_t} u_t) , w) ,
    \\
     & \quad \ + Q^{\sigma}(  \partial_x ( 2 g^{01}_{u} \partial_x u_t + f_{u u_t} u_t) , \partial^{-1}w) ,
    \\
    H_{11}(u,w) 
    &:= \ Q^{\sigma}(  \partial_x (g^{11}_{u} u +g^{11}_{u_x} u_x ) , w_x)  
    +Q^{\sigma}(  \partial_x ( g_{u_x}^{11}\partial_x^2 u+ f_{u_x u_x} u_x+ f_{u_x u} u ) , w) ,
    \\
    &  \quad \
     +  Q^{\sigma}(  \partial_x (g_{u}^{11} \partial_x^2u+ f_{uu}u+ f_{u_x u} u_x ) , \partial^{-1}w).
\end{aligned}
\]
Combining the above expansions with the commutator symbol computation in \eqref{eq:s-com}, we conclude that the associated bilinear symbols satisfy
\[
\begin{aligned}
    & h_{00} \in S^{1,0} , \qquad h_{01} \in S^{1,1}.\\
    & h_{10} \in S^{2,0} , \qquad h_{11} \in S^{2,1}.
\end{aligned}
\]
In particular, these symbols admit polyhomogeneous expansions.
Consequently, the resulting paradifferential equation fits into the class of equations considered in Section~\ref{s:normal-forms extras}.

To eliminate these quadratic terms, we invoke the normal form in Lemma~\ref{l-nfl-exist} 
which guarantees the existence of a corresponding normal form transformation 
\[
\Tilde{w}^{\sigma}: = \ w^{\sigma} + A_1(u,w^{\sigma}) + B_1(u_t,w_t^{\sigma}) + C_1(u_t,w^{\sigma}) + D_1(u,w_t^{\sigma}).
\]
 The normal form variable $\Tilde{w}^{\sigma}$  solves a linear paradifferential equation with source term, 
\[
L_{KG}^{para} \Tilde{w}^{\sigma} = \ \tilde{G}^{\sigma}
\]
with the source term given by:
\[
\begin{aligned}
    \tilde{G}^{\sigma} &:= \Lambda_{\ge 3} \bigg(L_{KG}^{para} w^{\sigma} +  L_{KG}^{para} \bigr( A_1(u,w^{\sigma}) + B(u_t,w_t^{\sigma}) + C(u_t,w^{\sigma}) + D(u,w_t^{\sigma}) \bigl) \bigg) .
\end{aligned} 
\]
Now Lemma~\ref{nfl-invert} yields the invertibility of the normal form transformation:
\[
\Vert (\Tilde{w}^{\sigma},\Tilde{w}^{\sigma}_t ) - (w^{\sigma}, w^{\sigma}_t) \Vert_{H^1 \times L^2}
\lesssim \
\A_2 \Vert w^{\sigma}[t] \|_{H^1 \times L^2}.
\]
Here we emphasize that this situation is different from the one in Section~\ref{s:normal-form}, where we treat the quasilinear equation. In that setting, a direct normal form correction would lead to unbounded terms. In the present case, however, the source terms may be viewed as perturbative: they involve at most one derivative on \(w\), and the associated normal form variable likewise contains at most one derivative on \(w\).

 Finally, Lemma~\ref{source} provides the source term bounds:
\[
\Vert \tilde G^{\sigma} \Vert_{L^2}
= \ \Vert L^{para}_{KG} \Tilde{w}^{\sigma} \Vert_{L^2}
\lesssim \
\A_0
\A_3
\Vert \Tilde{w}^{\sigma}[t] \Vert_{H^1 \times L^2}.
\]
This concludes the proof of the proposition.
\end{proof}

At this point, we  have an explicit expression for  the $H^s \times H^{s-1}$ energy functional $E^{s,para} (w)$ in Proposition~\ref{p:para-s},
namely 
\begin{equation}\label{para-energy}
\begin{aligned}
     E^{s,para} (w,w_t) & = \ E^{1,para} (\Tilde{w}^{s-1},\Tilde{w}^{s-1}_t) .
    \\ \notag
     \end{aligned} 
\end{equation} 
This works well for solutions to the homogeneous paradifferential equation, but applying it directly to the corresponding inhomogeneous 
problem is less than ideal, because the expression in the second line above also involves $w_{tt}$, so its time derivative would require also bounds for 
$\partial_t f$. The way out of this difficulty is to remark that for solutions to the homogeneous paradifferential equation we can directly substitute
$w_{tt}$ from the equation in terms of $w$ and $w_t$. In other words, if we think of $E^{1,para}$
defined above as a bilinear form  $E^{1,para}(\Tilde{w}^{s-1}, \Tilde{w}^{s-1}_t)$ acting on the 
pair $(\Tilde{w}^{s-1}, \Tilde{w}^{s-1}_t)$, 
then for solutions to the homogeneous equation this is equal to  $E^{1,para}(\Tilde{w}^{s-1}, \hat{\Tilde{w}}^{s-1}_t)$,
where the modified second variable is 
\[
\begin{aligned}
\hat{\Tilde{w}}^{s-1}_t = & \ \jD^{s-1} w_t + \partial_t A_1(  u, \jD^{s-1} w) + B_1( u_{tt},  \jD^{s-1} w_t ) 
    \\
    & \quad +\partial_t C_1(u_t, \jD^{s-1}w) +D_1(u_t, \jD^{s-1}w_t)
\\
& \quad + B_1( u_{t},  \jD^{s-1} w_{[tt]} )+ D_1(u, \jD^{s-1}w_{[tt]}).
\end{aligned}
\]
Here by $w_{[tt]}$ we denote the expression 
for $w_{tt}$ from the equation $L_{KG}^{para} w = 0$,
\[
w_{[tt]} = \sum_{(\alpha,\beta) \neq (0,0)} T_{g^{\alpha\beta}}\partial_{\alpha} \partial_{\beta} w + T_{\tilde F^{\gamma,lin}} \tilde \partial_{\gamma}   - m .
\]

So we redefine
\begin{equation} \label{epara-s}
 E^{s,para} (w,w_t) := \ E^{1,para}(\Tilde{w}^{s-1}, \hat{\Tilde{w}}^{s-1}_t),  
\end{equation}
and use this expression for solutions to the 
inhomogeneous equation $L_{KG}^{para} w = f$.
Now $E^{s,para}$ is defined as a bilinear form 
in $(w,w_t)$, for which we will prove the following:

\begin{proposition} \label{para-source}
For each regularity index \( s \ge 1\), the modified cubic energy \( E^{s,para} (w,w_t)\) constructed above satisfies
   the following cubic energy estimate
   for solutions $w$ to the inhomogeneous paradifferential equation \( L_{KG}^{para}w = f\):
    \[
    \big|\frac{d}{dt} E^{s,para} (w,w_t) \big|
    \lesssim \
     \A_0
    \A_3
     E^{s,para}(w,w_t)
     + \| f\|_{H^{s-1}} \| (w,w_t) \|_{H^s \times H^{s-1}}.
    \]
\end{proposition} 
\begin{proof}
The terms independent of \( f\)  were already treated in Proposition~\ref{p:para-s}. Here we only need to estimate the contributions involving \(f\). 
 Such contributions can only arise from the time derivative  $ \partial_t \hat{\Tilde{w}}^{s-1}_t$. 
 Here $\hat{\Tilde{w}}^{s-1}_t$ depends linearly on $w,w_t$ with coefficients depending on $u$,
\[
\hat{\Tilde{w}}^{s-1}_t = \ K^{s-1}_0 w_t + K^{s-1}_1 w.
\]
From the definition, we collect the operator applied to \(w_t\) as:
\[
\begin{aligned}
     K^{s-1}_0 w_t := & \ \jD^{s-1} w_t 
     + A_1(  u, \jD^{s-1} w_t) 
     + B_1( u_{tt},  \jD^{s-1} w_t ) 
     +  C_1(u_t, \jD^{s-1} w_t ) 
     \\
     &
     + D_1(u_t, \jD^{s-1}w_t)
  + B_1( u_{t},  \jD^{s-1} ( T_{2g^{01}} \partial_x w_t + T_{F^{0,lin} }w_t  )) \\
  & + D_1(u, \jD^{s-1}( T_{2g^{01}} \partial_x w_t + T_{F^{0,lin} }w_t  )).
\end{aligned}
\]
Then the $f$ term in $ \partial_t \hat{\Tilde{w}}^{s-1}_t$ is $K^{s-1}_0 f$.

By the definition of the modified energy, 
\[
\begin{aligned}
    E^{s,para}(w,w_t)  = & \ E^{1,para} (\tilde w^{s-1}, \hat{ \tilde w }^{s-1}_t) 
\end{aligned}
\]
the additional contribution from the source terms in \( \frac{d}{dt} E^{s,para}(w,w_t)\) is:
\[
\begin{aligned}
   \int  K_0^{s-1} w_t \cdot T_{1+\kappa_0}  K_0^{s-1} f \, dx.
\end{aligned}
\]
Thus it suffices to prove the bound
\[
\| K^{s-1}_0 f\|_{L^2} \lesssim \| f\|_{H^{s-1}}.
\]
We recall that the symbols of \( A_1\), \( B_1\), \( C_1\) , and \( D_1\) are in \(S^{3,0}\), \(S^{2,-1}\), \(S^{3,-1}\), \(S^{2,0}\), respectively (see Lemma~\ref{l-nfl-exist}).
The leading term \( \jD^{s-1} f\) is immediate. We take the following two representative terms as example,
\[
\begin{aligned}
    \| A_1(u, \jD^{s-1} f) \|_{L^2}
    &
    \lesssim \ \A_2 \| f \|_{H^{s-1}}
\end{aligned}
\]
The second example is the one using the expression for \(w_{[tt]}\):
\[
\begin{aligned}
    &\| B_1(u_t, \jD^{s-1}  T_{2g^{01}} \partial_x f + T_{F^{0,lin} }f) \|_{L^2} \\
    &
    \lesssim \ 
    \| \partial^{2} u_t \|_{L^{\infty}}  \| \partial^{-1} \jD^{s-1} \bigr( T_{2g^{01}} \partial_x f + T_{F^{0,lin} }f) \bigl)  \|_{L^2} 
    \\
    & \lesssim \
   \A_2  \| f \|_{H^{s-1}}.
\end{aligned}
\]
The estimates follow from paraproduct estimates~\eqref{CM} and commutator estimates. The rest of the terms will be bounded similarly; we omit the details.

\end{proof}

\section{Cubic energy estimates for the full equations}
\label{s:cubic ee}
\par

In the previous sections we have constructed a normal form transformation for
the nonlinear Klein–Gordon equation, developed a paradifferential
reduction, and built modified cubic energies for the
paradifferential problem together with bounds for the source terms. The goal of the present section is to use these ideas in the study of the full nonlinear equation in order to construct a cubic 
energy functional in $H^s \times H^{s-1}$, which 
will prove Theorem~\ref{t:eepropagation}. 

\begin{proof}[Proof of Theorem~\ref{t:eepropagation}]
We start with the full equation \eqref{kg1} written in paradifferential form \eqref{para-KG}, which
we recall here:
\begin{equation}
  L^{\mathrm{para}}_{KG} u = \ N_{\mathrm{bal}}(u),
\end{equation}
where the paradifferential operators are
\[
L^{para}_{KG} :=\ -T_{g^{\alpha \beta}} \partial_\alpha \partial_\beta   - T_{F^{\gamma,lin}} \partial_{\gamma}  - T_{F^{lin} }  + m.
\]
Using the following notation 
\[
T_{\Tilde{F}^{\gamma,lin}} \Tilde{\partial_{\gamma}} v  
:=  
\ T_{F^{\gamma,lin}} \partial_{\gamma} v  + T_{F^{lin} } v,
\] 
the balanced nonlinearity has the form
\[
N_{bal}(u) = T_{\partial_\alpha \partial_\beta u  } g^{\alpha \beta}(u,\partial u) + \Pi( g^{\alpha \beta}(u,\partial u),\partial_\alpha \partial_\beta u ) + f(u,\partial u) -  T_{\Tilde{F}^{\gamma,lin}} \Tilde{\partial_{\gamma}} u,
\]
which we have separated into a quadratic part and a cubic and higher order part, 
\[
N_{bal}(u) = \ N_{bal}^{[2]}(u) + N_{bal}
^{[3]}(u).
\]
The former plays the leading role 
while the latter will be treated perturbatively,
using Proposition~\ref{p:full3}.
The quadratic component can be described as in \eqref{quadratic-KG}, 
\[
N_{bal}^{[2]}(u) = \ Q_{11}^{hh}(u,u)
+ Q_{00}^{hh}(u_t,u_t) + Q^{hh}_{01}(u_t,u).
\]

Our goal is to apply Proposition~\ref{p:para-s} to a 
normal form correction of the solution $u$.
For the correction, we contend  that the high--high portions of the normal form in Lemma \ref{nfl} would satisfy the requirements. 
We define
\begin{equation}
u_{\NF} := \ u + A_{hh}(u,u) + B_{hh}(u_t,u_t) + C_{hh}(u_t,u) .
\end{equation}
We recall the properties of normal forms here:
for the high--high interactions, we have
\begin{equation} \label{hh-symbol}
    a_{hh}\in (\xi_1+\xi_2)S^{2}+S^{2},\qquad
b_{hh}\in (\xi_1+\xi_2)S^{0}+S^{0},\qquad
c_{hh}\in (\xi_1+\xi_2)S^{1}+S^{1}.
\end{equation}

Then we rewrite the equation for $u_{\NF}$ as a paradifferential equation with bounded, cubic and higher source terms:
\begin{equation}
       L^{para}_{KG} u_{\NF} = \ N_{bal}^{[3]}(u) + G,
\end{equation}  
where $G$ is given by
\begin{equation}\label{what-G}
G= \ \Lambda_{\geq 3} (L^{para}_{KG}(u_{\NF}-u)).
\end{equation}

\smallskip
We claim that the following two properties hold:

\begin{enumerate}
    \item[(i)] Invertibility: 
     \begin{equation}
     \label{e:invert}
     \Vert
     (u_{\NF}-u, \partial_t (u_{\NF}-u)) 
     \Vert_{H^{s} \times H^{s-1}}
     \lesssim \
   \A_2
     \Vert  u[t] \Vert_{H^{s} \times H^{s-1}},
     \end{equation}
    \item[(ii)] Perturbative source terms: 
     \begin{equation}
     \label{e:pert}
     \Vert 
      G
     \Vert_{H^{s-1}}
     \lesssim  \
     \A_0
     \A_3
     \Vert  u[t] \Vert_{H^{s} \times H^{s-1}}.
     \end{equation}
\end{enumerate}

Once these two properties are established, Proposition~\eqref{para-source} yields the
desired cubic energy estimates for the energy functional
\begin{equation}
E^s(u,u_t) = \ E^{s,para}(u_{\NF}, \Tilde{u}_{NF}),
\end{equation}
where
\[
\Tilde{u}_{NF} := \  u_t + 2 A_{hh} (u_t,u_t) + 2 B_{hh}(u_{[tt]}, u_t)
+ C_{hh}(u_t,u_t) + C_{hh}(u_{[tt]},u_t),
\]
and
\[
u_{[tt]} : = \ \sum_{(\alpha,\beta) \neq (0,0)} g^{\alpha\beta} \partial_\alpha \partial_\beta u  -m u+  f(u,\partial u) .
\]
To motivate the $u_{[tt]}$ substitution above  we note that, without it, the expression for $E^s(u,u_t)$ would also contain second time derivatives of $u$.
Eliminating these using the equation \eqref{kg1} allows us to view our energies as nonlinear functionals on the state space $H^s \times H^{s-1}$, without reference to the 
evolution equation.
The same comment  applies in the context of the
invertibility bound \eqref{e:invert} and in the expression for $G$.

\bigskip

\noindent (i) \underline{ Invertibility:} for the difference we have
\[
\begin{aligned}
     \Vert 
    (u_{\NF}, \partial_t u_{\NF}) - (u,u_t)
    \Vert_{H^{s} \times H^{s-1}}
    \leq & \
      \Vert A_{hh}(u,u) \Vert_{H^{s}}
    + \Vert B_{hh} (u_t,u_t) \Vert_{H^{s}}
    + \Vert C_{hh} (u_t,u) \Vert_{H^{s}}
 \\
       &+  \Vert \partial_t A_{hh}(u,u) \Vert_{H^{s-1}}+ \Vert \partial_t B_{hh} (u_t,u_t)  \Vert_{H^{s-1}}\\
       & +
     \Vert \partial_t C_{hh} (u_t,u) \Vert_{H^{s-1}}.
\end{aligned}
\]
The first four terms are estimated using the paraproduct estimates~\eqref{CM} and the high-high symbol properties~\eqref{hh-symbol}. 
For high-high bilinear forms we  freely redistribute the derivatives between the two inputs, so for instance for $A_{hh}$, which has order three, we have
\[
\begin{aligned}
\| A_{hh}(u,u) \|_{H^s} 
& 
\lesssim \
\| \partial_x \widetilde{A}_{hh}(u,u) \|_{H^s} 
+ \| A^{(0)}_{hh}(u,u) \|_{H^s}
\\
&
\lesssim \
\| \partial_x^{\leq3} u\|_{L^{\infty}} \|u \|_{H^s} 
+\| \partial_x^{\leq2} u\|_{L^{\infty}} \|u \|_{H^s}
\\
&
     \lesssim  \
     \A_2
     \Vert   u[t]  \Vert_{H^{s} \times H^{s-1}}.
\end{aligned}
\]

The last two terms can be estimated using the symbol bound~\eqref{hh-symbol} and the bounds for \( u_{tt}\) in \eqref{eq:p2-utt}. For instance,
treating the two inputs of \( B\) symmetrically, we obtain:
\[
\begin{aligned}
   \Vert  B_{hh} ( u_{tt} ,u_t) \Vert_{H^{s-1}}  
   & 
   \lesssim \
   \| u_{tt} \|_{W^{1,\infty}} \| u_t \|_{H^{s-1}}
   \lesssim \ 
   \A_2
     \Vert  u[t]  \Vert_{H^{s} \times H^{s-1}}.
\end{aligned}
\]

The bounds for the other terms follow via similar arguments. 
\medskip

\noindent (ii) \underline{The source term bound:}
By the definition of $u_{\NF}$ we have
\[
  \Lambda_2\bigl(L^{\mathrm{para}}_{KG}u_{\NF}\bigr)=0,
  \qquad
  \Lambda_{\ge3}\bigl(L^{\mathrm{para}}_{KG}u_{\NF}\bigr)=G.
\]
Indeed
\[
\begin{aligned}
    \Lambda_2( L_{KG}^{para} u_{\NF}   )  
    = & 
     L_{KG} (u + A_{hh}(u,u) +B_{hh}(u_t,u_t) + C_{hh}(u_t,u))
     \\
     &
    - Q^{hh}_{00}(u_t,u_t) 
- Q^{hh}_{10}(u_t,u)
- Q_{11}^{hh}(u,u),
\end{aligned}
\] 
which vanishes by construction.

By \eqref{para-KG} and \eqref{Nbal} we have 
\[
\Lambda_{\geq 3} ( L_{KG}^{para} u) = N_{bal}
^{[3]}(u),
\]
therefore we are left with 
\[
G= \Lambda_{\geq 3} (L^{para}_{KG}(u_{\NF}-u)),
\]
justifying \eqref{def:G}. 

\smallskip

Then we expand $G$ as follows:
\begin{equation}
\label{def:G}
\begin{aligned}
    G =
    &\ 
    \Lambda_{\geq 3} \bigr(L_{KG}^{para} \bigr(A_{hh}(u,u) + B_{hh}(u_t,u_t) + C_{hh}(u_t,u)\bigl)  \bigl)
    \\ 
    = &  \
     \Lambda_{\geq 3} \bigr[
     -T_{g^{\alpha \beta}} \partial_\alpha \partial_\beta A_{hh}(u,u)  
    - T_{\Tilde{F}^{\gamma,lin}}  \Tilde{\partial_{\gamma}} A_{hh}(u,u) + m  A_{hh}(u,u)  \\
    &
   \ \qquad  -T_{g^{\alpha \beta}} \partial_\alpha \partial_\beta B_{hh}(u_t,u_t)   - T_{\Tilde{F}^{\gamma,lin}}  \Tilde{\partial_{\gamma}} B_{hh}(u_t,u_t) + m  B_{hh}(u_t,u_t)  \\
     &
  \  \qquad   -T_{g^{\alpha \beta}} \partial_\alpha \partial_\beta C_{hh}(u_t,u)   - T_{\Tilde{F}^{\gamma,lin}} \Tilde{\partial_{\gamma}} C_{hh}(u_t,u)  + m  C_{hh}(u_t,u)
     \bigl].
\end{aligned}
\end{equation}
The $m$ terms are quadratic and can be removed. For 
the remaining terms, distributing derivatives 
by Leibniz rule, we have two scenarios:

\begin{enumerate}[label=(\alph*)]
\item  Expressions which depend only on $u$ and $u_t$,
which are estimated directly.

\item  Expressions which also depend on $u_{tt}$, where
we need to first use the equation \eqref{kg1}.
\item  Expressions which also depend on $u_{ttt}$, where
we also need to first use the equation \eqref{kg1}, possibly twice.
\end{enumerate}

This can easily generates a large number of cases, but they are all similar so for simplicity we 
consider the worst term, $T_{g^{\alpha \beta}} \partial_\alpha \partial_\beta B_{hh}(u_t,u_t)$,
which potentially contains the largest number of time derivatives:
\[
\begin{aligned}
     \Lambda_{\ge 3}\bigr( T_{g^{\alpha \beta}} \partial_\alpha \partial_\beta B_{hh}(u_t,u_t) \bigl) =
   & \
   \Lambda_{\ge 3}\bigr( T_{g^{11}} \partial_x^2 B_{hh}(u_t,u_t) \bigl) 
   +
    \Lambda_{\ge 3}\bigr( T_{g^{0 1}}  \partial_x B_{hh}(u_t,u_{tt}) \bigl) 
   \\ & \qquad \ +
     \Lambda_{\ge 3}  B_{hh}(u_{tt},  u_{tt}) 
     \ + \Lambda_{\ge 3}  B_{hh}(u_{t},  u_{ttt}) .
\end{aligned}
\]
Here the first term corresponds to case (a), the next two to case (b),
and the last to case (c).

Using the symbol bound~\eqref{hh-symbol},
 we obtain the desired estimates for case (a):
 \[
 \begin{aligned}
       \| \Lambda_{\ge 3}\bigr( T_{g^{11}} \partial_x^2 B_{hh}(u_t,u_t) \bigl) \|_{H^{s-1}} 
       &
       \lesssim \
       \|   \Lambda_{\geq 1}(g^{11} ) \|_{L^{\infty}}
       \|  \partial_x^2 B_{hh}(u_t,u_t)  \|_{H^{s-1}} 
       \\
       &
       \lesssim \
      \A_0 \A_3
       \|  u_t  \|_{H^{s-1}} .
 \end{aligned}
 \]

 For case (b), we use  the high--high localization of the bilinear form as well as the control for \( u_{tt}\) in \(L^{\infty}\) in \eqref{eq:p2-utt-2}.
 In the second term we can freely remove the 
 $\Lambda_{\geq 3}$ truncation, so we have
 \[
 \begin{aligned}
       \| \Lambda_{\ge 3}\bigr( T_{g^{0 1}}  \partial_x B_{hh}(u_t,u_{tt}) \bigl) \|_{H^{s-1}}  
       & \lesssim \
       \| g^{01} \|_{L^{\infty}} 
       \| u_t \|_{W^{3,\infty}}
       \| u_{tt} \|_{H^{s-2}} 
       \\
       &  \lesssim \ \A_0 \A_3   \|u[t]\|_{H^{s}\times H^{s-1}}.
 \end{aligned}
 \]
 For the third term, we estimate by \eqref{eq:p2-utt-2} 
 \[
 \begin{aligned}
      \| \Lambda_{\ge 3}  B_{hh}(u_{tt},  u_{tt}) \|_{H^{s-1}} 
      &
    \lesssim  \
    \| B_{hh}( \Lambda_{\ge 2}(u_{tt}), u_{tt}) 
    \|_{H^{s-1}} 
    \\
    &
    \lesssim \
    \| \Lambda_{\ge 2} u_{tt} \|_{W^{2,\infty}}
    \| u_{tt}
    \|_{H^{s-2}} 
    \\
    &\lesssim \
    \A_0 \A_3
            \|u[t]\|_{H^{s}\times H^{s-1}}.
 \end{aligned}
 \]

 We now consider the case (c) by using the high--high symbol localization property and the control for \(\partial^2 u_{tt}\) in \eqref{eq:p2-utt-2}
\[
\begin{aligned}
    \| \Lambda_{\ge 3}  B_{hh}(u_{t},  u_{ttt}) \|_{H^{s-1}}  
    &
    \lesssim \
    \|  B_{hh}(u_{t}, \Lambda_{\ge 2} u_{ttt}) \|_{H^{s-1}}
    \\
    &
    \lesssim \
    \|  \Lambda_{\ge 2} u_{ttt} \|_{W^{1,\infty}} \| u_t \|_{H^{s-1}}
    \\
    &\lesssim \ \A_0 \A_3  \|u[t]\|_{H^{s}\times H^{s-1}} .
\end{aligned}
\]

Thus, we conclude that
   \[
\begin{aligned}
    \| \Lambda_{\ge 3}\bigr( T_{g^{\alpha \beta}} \partial_\alpha \partial_\beta B_{hh}(u_t,u_t) \bigl) \|_{H^{s-1}} 
   & \lesssim  \
   \A_0 \A_3
            \|u[t]\|_{H^{s}\times H^{s-1}}.
\end{aligned}
\]

   \
Finally, estimating each of the above terms using the bounds for $A_{hh}$, $B_{hh}$,  and \( C_{hh}\) from Lemma~\ref{nfl} together with the Coifman–Meyer type estimates,
we arrive at
\[
  \|G\|_{H^{s-1}}
   \lesssim \
   \A_0
   \A_3
    \|u[t]\|_{H^{s}\times H^{s-1}},
\]
which proves the desired source term bound for \( G\).

This completes the proof.

\end{proof}

\section{Cubic energy estimates for the linearized equations}
\label{s:cubic linearized}

In the previous sections we derived cubic energy estimates for the full
nonlinear equation.  In this section we turn to the paradifferential
linearized equation~\eqref{para-KG-lin} and obtain corresponding estimates for the linearized
flow.  These will be used to prove Theorem~\ref{t:eepropagation-lin}.

\begin{proof}[Proof of Theorem~\ref{t:eepropagation-lin}]
We use the  paradifferential form of the linearized equations in \eqref{para-KG-lin}, which we recall here for convenience:
\begin{equation}\label{para-KG-lin-re}
  L^{\mathrm{para}}_{KG} v = N^{lin}_{\mathrm{bal}}(u) v,
\end{equation}
where the source term is
\[
\begin{aligned}
    N^{lin}_{\mathrm{bal}}(u) v
    = \ T_{ \partial_\alpha \partial_\beta v} g^{\alpha \beta }(u,\partial u ) 
    + \Pi( \partial_\alpha \partial_\beta v, g^{\alpha \beta }(u,\partial u ) )
    + T_{ \Tilde{\partial}_{\gamma} v} \Tilde{F}^{\gamma,lin}
    + \Pi(\Tilde{\partial}_{\gamma} v,\Tilde{F}^{\gamma,lin}  ).
\end{aligned}
\]
 We again separate the nonlinearity into a quadratic part and a cubic and higher order part, 
\[
N^{lin}_{bal}(u)v = \ N_{bal}^{lin,[2]}(u)v + N_{bal}
^{lin,[3]}(u)v.
\]
For the cubic part we already have the favourable estimates in Proposition~\ref{p:lin3}. As in the case of the full equation in the previous section,  our strategy will be to eliminate the quadratic 
terms using a normal form transformation.
However, unlike in the case of the full equation, now the quadratic part contains also unbalanced terms, namely
\[
\begin{aligned}
    N_{bal}^{lin,[2]}(u)v  = & \ 2 Q^{hh}_{11}(u,v)
+ 2Q^{hh}_{00}(u_t,v_t) + Q^{hh}_{10}(u,v_t) 
+  Q^{hh}_{01}(u_t,v) 
\\
& 
\quad \
+  2Q^{hl}_{11}(u,v) + Q^{hl}_{10}(u,v_t)
+ Q^{hl}_{01}(u_t,v) +2Q^{hl}_{00}(u_t,v_t).
\end{aligned}
\]

To obtain the cubic energy estimate, we first apply a normal form
transformation in order to remove the quadratic terms on the right–hand
side of  \eqref{para-KG-lin}.
Let 
\[
\begin{aligned}
    & v_{\NF}  := \ v + A^1(u,v) + B^1(u_t,v_t) + C^1(u_t,v) + D^1(u,v_t),  \\
    & G_v     : = \ L_{KG}^{para} v_{\NF}.
\end{aligned}
\]

  We determine the bilinear expressions $A^1(u,v)$, $B^1(u_t,v_t)$,
  \(C^1(u_t,v)\), and \( D^1(u,v_t)\) by applying Lemma~\ref{nfl}.
  Here, these normal form variables may be viewed as the linearization of the normal form transformation constructed in Lemma~\ref{nfl}.
  In that lemma, we denote the corresponding normal form variables by \(A(u,u)\), \( B(u_t,u_t)\), and \(C(u_t,u)\).
  Since the linearized normal form variables arises from the normal form variables for the full equation, we can describe  \(A^1\), \(B^1\), \(C^1\), \(D^1\) in terms of the normal form variables introduced in Section~\ref{s:normal-form}, which simplifies the computations.
 \begin{equation} \label{eq:nfl-linear}
     \left\{
 \begin{aligned}
    & A^1(u,v)     := \ 2 A_{hl}(u,v) + 2 A_{hh}(u,v), \\
    & B^1(u_t,v_t)  := \ 2 B_{hl}(u_t,v_t) +2 B_{hh}(u_t,v_t), \\
    & C^1(u_t,v)     := \ C_{hl}(u_t,v)  +C_{hh}(u_t,v),\\
    & D^1(u,v_t)    := \ C_{hl}(u,v_t) + C_{hh}(u,v_t). 
 \end{aligned}
 \right.
 \end{equation}
 
  We then define the cubic energy for the linearized variables by
  \[
  E_{lin} (v,v_t) := E^{1,para} (v_{NF}, v_{NF,t}).
  \]
A-priori the expression for $E_{lin} (v,v_t)$
also involves $v_{tt}$. However, as $v$ is assumed to solve the linearized equation, we can replace every instance of $v_{tt}$ with a linear expression in $v$ and $v_t$, with coefficients depending on $u$ and its derivatives. Thus,
$E_{lin} (v,v_t)$ can be unambiguously and uniquely identified with a bilinear expression
in $(v,v_t)$.

For our normal form transformation we will show invertibility, as well as the perturbative source terms bounds
\begin{equation}\label{equiv-lin}
\begin{aligned}
    \| (v_{\NF} ,\partial_t v_{\NF}) - (v,v_t) \|_{H^1 \times L^2}
    \lesssim \
    \A_2
    \| v[t] \|_{H^1 \times L^2},
\end{aligned}
\end{equation}
\begin{equation}\label{nf-source-lin}
\begin{aligned}
    \| G_v \|_{L^2}
    \lesssim \
     \A_0
     \A_3
    \| v[t] \|_{H^1 \times L^2}.
\end{aligned}
\end{equation}
Then, by Theorem~\eqref{paraE-inhom} for \( s=1\) applied to $v_{NF}$, the cubic energy estimates for $v$ follow. 
It remains to prove the two bounds above, \eqref{equiv-lin}  and \eqref{nf-source-lin}.

\medskip

We would like to split the analysis into two parts, corresponding to the two components of the 
normal form transformation, associated to the high--low interactions,  respectively the high--high interactions.

This separation is directly linear in the case of 
\eqref{equiv-lin}. On the other hand for $G_v$
we write
\[
\begin{aligned}
G_v = & \ L_{KG}^{para} v_{NF} = \Lambda_{\geq 3} L_{KG}^{para} v_{NF}
\\
= & \ \Lambda_{\geq 3} N^{[3]}_{lin}(u) v 
\\
& + \Lambda_{\geq 3} L_{KG}^{para} ( A^1_{hl}(u,v) + B^1_{hl}(u_t,v_t) + C^1_{hl}(u_t,v) + D^1_{hl}(u,v_t))
\\
& + \Lambda_{\geq 3} L_{KG}^{para} ( A^1_{hh}(u,v) + B^1_{hh}(u_t,v_t) + C^1_{hh}(u_t,v) + D^1_{hh}(u,v_t)),
\end{aligned}
\]
where we carefully note that, to avoid ambiguities, the projections
$\Lambda_{\geq 3}$ apply to the above expressions 
viewed as multilinear in $(u,u_t)$ and $(v,v_t)$.
The bound for the first term in $G_v$ has been proved in Proposition~\ref{p:lin3}. So it remains 
to separately consider the contributions of the high--low and high--high terms.

\medskip

 \emph{High--low interactions:}
 The bounds corresponding to the high--low interactions follow from Lemma~\ref{nfl-invert} and Lemma~\ref{source}. 

We begin with the high-low component of the normal forms, which by \eqref{eq:nfl-linear}
have the form
  \[
 \begin{aligned}
    & A_{hl}^1(u,v)  := 2 A_{hl}(u,v) , &   B^1_{hl}(u_t,v_t)  := 2 B_{hl}(u_t,v_t)\\
    & C^1_{hl}(u_t,v) := C_{hl}(u_t,v) ,&  D^1_{hl}(u,v_t)  := C_{hl}(u,v_t) .
 \end{aligned}
 \]
By Lemma~\ref{nfl}, these symbols have  regularity
 \[
 \begin{aligned}
     &a^{1}_{hl} \in S^{2,1},  & \ \quad b_{hl}^1 \in S^{1,0} , \\
     &c^1_{hl} \in S^{1,1} ,  & \ \quad  d_{hl}^1 \in S^{2,0}.
 \end{aligned}
 \]

 To apply  Lemma~\ref{nfl-invert}, we need to write the equation for 
 $v$ in the form \eqref{quadratic-para-lin}, with a suitable source term $F$. Precisely, \(F\) consists of all terms in  \( N^{bal}_{lin}(u)v\) except for the bilinear high--low interactions:
 \[
 F: = 2 Q^{hh}_{11}(u,v)
+ 2Q^{hh}_{00}(u_t,v_t) + Q^{hh}_{10}(u,v_t) 
+  Q^{hh}_{01}(u_t,v)  + N^{lin,[3]}_{bal} (u)v .
 \]

Then, by Proposition~\ref{p:lin2} and \eqref{eq:n-lin-3},  \(F\) satisfies
\[
\| F \|_{L^2} \lesssim  \ \A_1 \| v[t] \|_{H^1 \times L^2},
\]
which yields the desired invertibility estimates.

 We now turn to the source term
 \[
 \begin{aligned}
     G_{v,hl} = & \ \Lambda_{\ge 3} L_{KG}^{para} \bigr(  A^1_{hl}(u,v) + B^1_{hl}(u_t,v_t) + C^1_{hl}(u_t,v) + D^1_{hl}(u,v_t) \bigl)
     ,
 \end{aligned}
 \]
for which we use  Lemma~\ref{source}. 
Then it remains to verify that \(F\)
satisfies the bounds~\eqref{F-sour-t}.

 The quadratic bound follows from using the high--high property in \(Q\). As an example, consider \(Q_{00}^{hh}(u_t,v_t)\):
 \[
 \begin{aligned}
     \|  \Lambda_2 \partial_t(Q_{00}^{hh}(u_t,v_t)   \|_{H^{-1}} 
     &
     \lesssim  \ \| \jD^{-1}  (Q^{hh}_{00}(\Lambda_1u_{tt},v_t) )  \|_{L^2} +\| \jD^{-1} (Q_{00}^{hh}(u_t, \Lambda_1 v_{tt}))   \|_{L^2}  
     \\
     &
     \lesssim \
     \A_1 \| v[t]\|_{H^1 \times  L^2} .
 \end{aligned}
 \]
 The cubic and higher bounds follow from \eqref{eq:n-lin-3-t}.

\medskip
\emph{High--High interactions:} 
Again, we use the normal form variables in Lemma~\ref{nfl} to describe the linearized normal form variables:
 \[
 \begin{aligned}
     A_{hh}^1(u,v)   &  := \ 2 A_{hh}(u,v) ,
      \qquad  B^1_{hh}(u_t,v_t)  := \ 2 B_{hh}(u_t,v_t)  \\
     C^1_{hh}(u_t,v)    & := \ C_{hh}(u_t,v) ,
     \qquad   D^1_{hh}(u,v_t)   := \ C_{hh}(u,v_t) .
 \end{aligned}
 \]
 Recalling the high--high properties in Lemma~\ref{nfl}, we can describe their regularities:
 \[
 \begin{aligned}
&a^1_{hh}\in (\xi_1+\xi_2)S^{2}+S^{2},
\qquad
b^1_{hh}\in (\xi_1+\xi_2)S^{0}+S^{0}, 
\\
& c^1_{hh}\in (\xi_1+\xi_2)S^{1}+S^{1},
\qquad
d^1_{hh}\in (\xi_1+\xi_2)S^{1}+S^{1}.
 \end{aligned}
\]

 For the high--high analysis it suffices to repeat  the argument in the proof of Theorem~\ref{t:eepropagation}. Specifically, we use the equation for both $v$ and $u$ and paraproduct estimates~\eqref{BMO-shift} to move one derivative from $v$ to $u$ and apply the  Moser estimates~\eqref{bmo-hs}.

We take the most difficult term as an example; we chose $B_1(u_t,v_t)$ because it contains the most time derivatives. We use \( B^1_{hh}(u_t,v_t)\) to denote the high-high portions of the bilinear operators $B_1(u_t,v_t)$.
\medskip

\noindent (i) \underline{ The invertibility bound }: 
We first begin by expanding
\[
\begin{aligned}
    \|  \partial_t B^1_{hh} (u_t,v_t)   \|_{L^2}
    & \lesssim \ 
    \Vert B_{hh}^1(u_{tt},v_t) \|_{L^2} + \Vert B_{hh}^1(u_{t},v_{tt}) \|_{L^2}  
    \\ 
    & \lesssim  \ 
    \| u_{tt}\|_{W^{1,\infty}} \|v_t \|_{L^2} + 
    \| u_{t}\|_{W^{2,\infty}} \|v_{tt} \|_{H^{-1}}
    \\
     & \lesssim  \ 
     \A_2 \| v[t] \|_{H^1 \times L^2} +  \| u_{t}\|_{W^{2,\infty}} \bigr(\|v_{tt}^{(-1)} \|_{H^{-1}} + 
     \| v_{tt}^{(0)} \|_{L^2} \bigl)
      \\
     & \lesssim_{\A_0} \ 
     \A_2 \| v[t] \|_{H^1 \times L^2}
    .
\end{aligned}
\]

The estimate for the first term follows by using \eqref{eq:p2-utt}.
For the second term, we use the expansion in \eqref{eq:vtt} to control \(v_{tt} \).

\medskip

\noindent(ii) \underline{The source term bound}:
We  expand the expression \(G_{v,hh}\) as follows:
 \[
 \begin{aligned}
     G_{v,hh} =
    &\ 
    \Lambda_{\geq 3} L_{KG}^{para} \bigr(A^1_{hh}(u,v) + B^1_{hh}(u_t,v_t) + C^1_{hh}(u_t,v) + D^1_{hh}(u,v_t)\bigl)  
    \\ 
    = &  \
     \Lambda_{\geq 3} \bigr[
     -T_{g^{\alpha \beta}} \partial_\alpha \partial_\beta A^1_{hh}(u,v)  
    - T_{\Tilde{F}^{\gamma,lin}}  \Tilde{\partial_{\gamma}} A^1_{hh}(u,v) + m  A^1_{hh}(u,v)  \\
    &
   \ \qquad  - T_{g^{\alpha \beta}} \partial_\alpha \partial_\beta B^1_{hh}(u_t,v_t)   - T_{\Tilde{F}^{\gamma,lin}}  \Tilde{\partial_{\gamma}} B^1_{hh}(u_t,v_t) + m  B^1_{hh}(u_t,v_t)  \\
     &
  \  \qquad   -T_{g^{\alpha \beta}} \partial_\alpha \partial_\beta C^1_{hh}(u_t,v)   - T_{\Tilde{F}^{\gamma,lin}} \Tilde{\partial_{\gamma}} C^1_{hh}(u_t,v)  + m  C^1_{hh}(u_t,v)
  \\
   &
  \  \qquad   -T_{g^{\alpha \beta}} \partial_\alpha \partial_\beta D^1_{hh}(u,v_t)   - T_{\Tilde{F}^{\gamma,lin}} \Tilde{\partial_{\gamma}} D^1_{hh}(u,v_t)  + m  D^1_{hh}(u,v_t)
     \bigl].
 \end{aligned}
 \]

 The $m$ terms are quadratic and can be removed. For 
the remaining terms, distributing derivatives 
by Leibniz rule, we have three scenarios:

\begin{enumerate}[label=(\alph*)]
\item  Expressions which depend only on $u$, $u_t$, \( v\), and \(v_t\),
which are estimated directly.

\item  Expressions which also depend on $u_{tt}$ and \( v_{tt}\), where
we need to first use the equations \eqref{kg1} and \eqref{para-KG-lin-re}.
\item  Expressions which also depend on $u_{ttt}$ and \( v_{ttt}\), where
we need to first use the equations \eqref{kg1} and \eqref{para-KG-lin-re}, possibly twice.
\end{enumerate}

  This can easily generate a large number of cases, but they are all similar so for simplicity we 
consider the worst term, $T_{g^{\alpha \beta}} \partial_\alpha \partial_\beta B^1_{hh}(u_t,v_t)$,
which potentially contains the largest number of time derivatives:
 \[
 \begin{aligned}
     \| \Lambda_{\ge 3 } \bigr(T_{g^{\alpha\beta}}
     \partial_\alpha \partial_\beta B_{hh}^1(u_t,v_t)\bigl) \|_{L^2}
     &
     \lesssim \
     \| \Lambda_{\ge 3 }
      B_{hh}^1(u_{tt},v_{tt})\|_{L^2} 
       +  \| \Lambda_{\ge 3 }
      B_{hh}^1(u_{t},v_{ttt})\|_{L^2} \\
      & \quad \ 
         +  
         \| \Lambda_{\ge 3 }
      B_{hh}^1(u_{ttt},v_{t})\|_{L^2} 
      +
    \| \Lambda_{\ge 3}\bigr( T_{g^{0 1}}  \partial_x B^1_{hh}(u_t,v_{tt}) \bigl) \|_{L^2}  
     \\
     & \quad \ +
     \| \Lambda_{\ge 3}\bigr( T_{g^{0 1}}  \partial_x B^1_{hh}(u_{tt},v_{t}) \bigl) \|_{L^2} 
 +
    \| \Lambda_{\ge 3}\bigr( T_{g^{11}} \partial_x^2 B^1_{hh}(u_t,v_t) \bigl) \|_{L^2}.
 \end{aligned}
 \]
 Here only the last term corresponds to case (a);
 the second and the third term are in case (c);
and the rest are all in case (b).
Using the high--high property of the paraproduct and the Moser estimates~\eqref{bmo-hs},
 we obtain the desired estimates for case (a):
 \[
 \begin{aligned}
     \| \Lambda_{\ge 3}\bigr( T_{g^{11}} \partial_x^2 B^1_{hh}(u_t,v_t) \bigl) \|_{L^2}
     &
     \ \lesssim  \ \|  \Lambda_{\ge 1}g^{11}\|_{L^{\infty}}
     \| \partial_x^2 B^1_{hh}(u_t,v_t)  \|_{L^2} 
     \\
     &\  \lesssim \ \A_0 \A_3 \| v[t] \|_{H^1 \times L^2}.
 \end{aligned}
 \]
 For case (b), there are three sub-cases: only \(u_{tt}\), only \(v_{tt}\), and both \(u_{tt}\) and \( v_{tt}\) show up. 

 The first two sub-cases are exactly the same  as in the proof of Theorem~\ref{t:eepropagation}. So, we only describe the sub-case where we use both the full and the linearized equation. Again, we separate them by homogeneity:
 \[
 \begin{aligned}
       \| \Lambda_{\ge 3}  B_{hh}(u_{tt},  v_{tt}) \|_{L^2} 
      &
    \lesssim \
    \| B_{hh}( \Lambda_{\ge 2}(u_{tt}), v_{tt}) 
    \|_{L^2} 
    + \| B_{hh}( u_{tt}, \Lambda_{\ge 2}(v_{tt})) 
    \|_{L^2} .
    \\
 \end{aligned}
 \]
For the first term, we invoke  \eqref{eq:p2-utt-2} and \eqref{eq:vtt} to bound it by
\[
\begin{aligned}
    \| \Lambda_{\ge 2}(u_{tt}) \|_{W^{2,\infty}}
    \| v_{tt} \|_{H^{-1}}  
    \lesssim \
   \A_0 \A_3
            \|v[t]\|_{H^1 \times L^2}.
\end{aligned}
\]
For the second term, we use the para-associativity in Lemma~\ref{l:para-assoc} together with bounds~\eqref{eq:p2-utt}:
\[
\begin{aligned}
   & \| B_{hh}( u_{tt}, \Lambda_{\ge 2}( L_{KG}^{para} v - v_{tt})  
    \|_{L^2} 
    + 
      \| B_{hh}( u_{tt},  N_{bal}^{lin} (u)v )  
    \|_{L^2}  \\
    &
    \lesssim \
    \| B_{hh}( u_{tt}, T_{\Lambda_{\geq 1} g^{\alpha 1} }\partial_x \partial_\alpha v  + T_{\Lambda_{\geq 1} \Tilde{F}^{\gamma,lin}}\Tilde{\partial}^{\gamma} v )
    \|_{L^2} 
    + \| u_{tt} \|_{W^{1,\infty}} \|  N_{bal}^{lin} (u)v\|_{L^2}
    \\
     &
    \lesssim \ \A_0  \A_3 \| v[t] \|_{H^1 \times L^2} .
\end{aligned}
\]
It remains to consider case (c), which we 
split into two subcases.

\bigskip
\noindent  \emph{The case (c).Sub-case: terms involving \(v_{ttt}\).}
For this case, we use the expansion in Proposition~\ref{p:infty} and then use the bounds~\eqref{eq:vttt} to estimate
\[
\begin{aligned}
   \| \Lambda_{\ge 3}  B_{hh}(u_{t},  v_{ttt}) \|_{L^2} 
      &  
      =  \|   B_{hh}(u_{t},  \Lambda_{\ge 2}v_{ttt}) \|_{L^2} 
      \\
      & 
      \lesssim
        \|   B_{hh}(u_{t},  \Lambda_{\ge 2}v_{ttt}^{(-2)}) \|_{L^2} 
        + \|   B_{hh}(u_{t},  \Lambda_{\ge 2}v_{ttt}^{(-1)}) \|_{L^2}  
        + \|   B_{hh}(u_{t},  \Lambda_{\ge 2}v_{ttt}^{(0)}) \|_{L^2}  
      \\
      &
       \lesssim \| u_t \|_{W^{3,\infty}} \| \Lambda_{\ge 2}v_{ttt}^{(-2)} \|_{H^{-2}}
       + 
         \| u_t \|_{W^{2,\infty}} \| \Lambda_{\ge 2}v_{ttt}^{(-1)} \|_{H^{-1}}
          + 
         \| u_t \|_{W^{1,\infty}} \| \Lambda_{\ge 2}v_{ttt}^{(0)} \|_{L^2}
        \\
      & \lesssim \A_3 \A_0 \| v[t] \|_{H^1 \times L^2} + \A_2 \A_1  \| v[t] \|_{H^1 \times L^2} + \A_1 \A_2 \| v[t] \|_{H^1 \times L^2}.
\end{aligned} 
\]
Using the interpolation inequality~\eqref{eq:n-inter}, we obtain the desired bounds.

\medskip

\noindent \emph{Sub-case: terms involving \(u_{ttt}\)}.
To bound the terms containing \(u_{ttt}\),  we use the estimates for \( u_{tt}\) in \(L^{\infty}\) in \eqref{eq:p2-utt-2}:
\[
\begin{aligned}
      \| \Lambda_{\ge 3}  B_{hh}(u_{ttt},  v_{t}) \|_{L^2} 
      & 
      = \| B_{hh}( \partial_t \Lambda_{\ge 2} u_{tt} ,v_t)  \|_{L^2}
      \\
      &
      \lesssim 
       \|  \partial_t \Lambda_{\ge 2} u_{tt}   \|_{W^{1,\infty}} \|
       v_t \|_{L^2} 
       \\
       &
       \lesssim \A_0 \A_3 \| v[t] \|_{H^1 \times L^2}.
\end{aligned}
\]
The proof of the theorem is concluded.

\end{proof}

\section{Enhanced \texorpdfstring{$\epsilon^{-2}$}{} lifespan via modified energy functionals }
\label{s:cubic lifespan}

Having established the cubic energy estimates, we now use them to derive the corresponding $\epsilon^{-2}$ lifespan bounds in Theorem~\ref{t:enhanced} and Theorem~\ref{t:weak-Lip}.
Although the argument is fairly standard, we include it here for the sake of completeness. Our argument relies on a bootstrap argument, using the cubic modified energy functionals and Sobolev embeddings.

\begin{proof}[Proof of Theorem~\ref{t:enhanced}]
For $s > \frac92$ we consider the equation~\eqref{kg1}
with initial data satisfying 
\begin{equation}
\|u[0]\|_{H^{s} \times H^{s-1}} \leq \epsilon \ll 1.    
\end{equation}
The local theory insures that a local solution $u$ exists for a short time. We need to expand the
interval of existence of the solutions.
We argue by a bootstrap argument using the cubic energy estimates.  
\medskip

Fix $T>0$ and suppose that on $[0,T]$ we have the apriori bound
\begin{equation}\label{boot}
  \|u[t]\|_{H^s \times H^{s-1}} \leq \ C\epsilon
  \qquad\text{for all } t\in[0,T],
\end{equation}
for some large constant $C\ge 1$ to be chosen later.
Then we want to show that the same bounds hold with better constants:
\[
\|u[t]\|_{H^s \times H^{s-1}} \leq  \ \frac{1}{2}C \epsilon 
\]
under the assumption $T \ll \epsilon^{-2}$.

By  part (i) of Theorem~\ref{t:eepropagation}
we have the norm equivalence
\[
E^{s,para}(u[t])\approx \ \|u[t]\|_{H^{s}\times H^{s-1}}^2.
\]
By  part (ii) of Theorem~\ref{t:eepropagation}, we have the cubic energy estimate 
\[
\frac{d}{dt}E^{s,para}(u[t])
  \lesssim \
\A_0 \A_3
           \,E^{s,para}(u[t]).
\]
Since $s>\tfrac92$, Sobolev embeddings yield
\[
\|\partial u(t)\|_{L^\infty}
  + \|\partial_x^3 \partial u(t)\|_{L^\infty}
  \lesssim \
  \|u[t]\|_{H^{s}\times H^{s-1}}^2. 
\]
Using the bootstrap assumption \eqref{boot}, we conclude that
\[
\A_0 \A_3
\leq \ C_0 C^2 \epsilon^2
\qquad\text{for } t\in[0,T]
\]
with a universal constant $C_0$ independent of $\epsilon$ and $T$.

 Gr\"onwall's lemma then gives
\[
E^{s,para}(u[T])
  \lesssim \ E^{s,para}(u[0])
          \exp\!\Big( C_0 C^2\epsilon^2 T \Big).
\]
 Using the smallness of the initial data,
\[
E^{s,para}(u[0])^{\frac{1}{2}} \lesssim \|u[0]\|_{H^{s}\times H^{s-1}} \lesssim \ \epsilon,
\]
we obtain
\[
\|u[t]\|_{H^{s}\times H^{s-1}} \lesssim E^{s,para}(u[T])^{\frac{1}{2}}
\lesssim \ \epsilon\,\exp\!\Big( \frac{1}{2}C_0 C^2\epsilon^2 T\Big).
\]

If we now restrict to times $T$ such that $ C_0 C^2 \epsilon^2 T\leq 1 $, that is
$T\ll C_0^{-1} C^{-2} \epsilon^{-2}$, we find
\[
E^{s,para}(u[T])^{\frac{1}{2}}
\lesssim \ \epsilon, 
\]
with a universal implicit constant.

Hence we obtain 
\[
E^{s,para}(u[T])^{\frac{1}{2}}
\leq \ \tfrac12 C\epsilon,
\]
provided $C$ is chosen large enough.  
Thus the bootstrap bound \eqref{boot} improves from $C\epsilon$ to
$\tfrac12 C\epsilon$, and a standard continuity argument closes the
bootstrap and yields the desired bound on $[0,T]$ with $T\ll \epsilon^{-2}$.
\end{proof}
Now we use the cubic energy estimates for the linearized equation to prove the Theorem~\ref{t:weak-Lip}.

\begin{proof}[Proof of Theorem~\ref{t:weak-Lip}]
    For \( s > \frac{9}{2}\), and let $u^1,u^2$ be two solutions to  equation~\eqref{kg1}, with initial data \( u^j[0] \) for \(j =1,2\).

    We can represent the difference of the two solutions as 
    \[
    u_2 - u_1 = \ \int_1^2 \frac{d}{dh} u_h \, dh,
    \]
    where \(\{ u_h\}_{h \in [1,2] }\) is a one-parameter family of solutions connecting \( u_1\) and \( u_2\). 
    In particular, \(\frac{d}{dh} u_h \) solves the linearized equation.
    It then suffices to obtain uniform estimates for the linearized flow.

By part (1) of Theorem~\ref{t:eepropagation-lin}, we have the norm equivalence
\begin{equation}
    E_{lin}(v[t]) \approx_{\A_2}  \Vert v[t]\Vert _{H^{1}\times L^2}^2.
\end{equation}
By part (2) of Theorem~\ref{t:eepropagation-lin}, we have the cubic energy estimates
\begin{equation}
\left|\frac{d}{dt}   E_{lin}(v[t])\right|
\lesssim_{ \ \A_2}  \
\A_0
\A_3 \,  E_{lin}(v[t]).
\end{equation}
Here, by a slight abuse of notations, the control parameters \( \A_k\) are taken uniformly over $h\in[1,2]$; namely, they are computed from
the family $\{u_h\}_{h\in[1,2]}$.
Using the apriori bounds from the  previous proof, we obtain
\[
\A_0 \A_3 \leq \ C_0 \epsilon^2
\qquad \text{for all }  t \in [0,T],
\]
with universal constant \( C_0 > 0\) independent of 
\(\epsilon\) and \(T \ll \epsilon^{-2}\).

Applying Gr\"onwall's lemma yields
\[
E_{lin}(v[T]) \leq \ E_{lin}(v[0]) \exp{(C_0\epsilon^2T)}.
\]
Hence, using  \(T \ll \epsilon^{-2}\),
 we conclude that
\[
\|(u^2-u^1,\partial_t(u^2-u^1))\|_{L_t^\infty(H^1\times L^2)}
\;\lesssim\; \|(u_0^2-u_0^1,u_1^2-u_1^1)\|_{H^1\times L^2},
\]
as needed.
\end{proof}


\section{\!\!\! Strichartz estimates and the \texorpdfstring{$\epsilon^{-4}$}{}  lifespan
}
\label{s:Strichartz}

The $\epsilon^{-2}$ lifespan bound proved in the previous section applies for the quasilinear Klein-Gordon models on both $\T$ and $\R$. This is likely sharp generically in the periodic setting, but it can be improved on $\R$. 
In this section we combine the cubic energy estimates in Theorem~\ref{t:eepropagation} with the Strichartz estimates associated to the linear problem  in order to obtain a lifespan of order $\epsilon^{-4}$ in $\R$. 
The
argument is intentionally crude: our use of Strichartz estimates is far
from optimal and forces us to control several additional derivatives of the solution
through the energy estimates.  The purpose of this section is therefore not to
reach the sharp Sobolev exponent, but rather to provide a model argument which
can serve as a prototype for more refined results in future work.

We begin by recalling standard Strichartz estimates for the
one–dimensional linear Klein--Gordon flow; see, for instance,
Keel--Tao~\cite{KeelTao1998-AJM} and Tao's monograph~\cite{Tao2006-CBMS}.
A straightforward application of normal forms could convert the quadratic source terms into the cubic and higher order ones.
We then estimate these cubic and higher-order terms using a suitable Strichartz norm of the solution $u$ and its derivatives, namely \(L_t^{4}L_x^{\infty}\), which pairs well with the cubic energy estimates. This yields an existence time of order \( \epsilon^{-4} \) for the quasilinear problem. Moreover, building on this result, we can run an additional bootstrap argument for the linearized equation and obtain the same \(\epsilon^{-4}\) time scale. In particular, this leads to weak-Lipschitz bounds on the \( \epsilon^{-4} \) time scale.

For convenience we state the relevant Strichartz estimates in the following lemma.

\begin{lemma} [Strichartz estimates]
Assume $u$ solves the linear Klein-Gordon equation in $[0,T ] \times \R$
\[
 L_{KG} u  = \  f.
\]
Then the following estimates hold
\[
\Vert \langle D \rangle^{-\frac{1}{4}}\partial  u \Vert_{L^{4}_tL^{\infty}_x}  
\leq  \
C (\Vert   u[0]\Vert_{H^1 \times L^2} 
+
\Vert f \Vert_{L^{1}_t L^2_x}
).
\]
    
\end{lemma}
The key point is that the Strichartz norm \( L^4_t L^{\infty}_x\) can be combined with the cubic energy estimates.
We will work with the equation~\eqref{eq-full},
\begin{equation*}
- g^{\alpha \beta}(u,\partial u) \partial_\alpha
\partial_\beta u + m u  = \ f(u,\partial u).
\end{equation*} 
 Motivated by the normal form in Proposition~\ref{nfl}, we proceed as follows. We apply the full normal form transformation in Proposition~\ref{nfl},
 \[
 \mathbf{u}  = \ u +A(u,u) + B(u_t,u_t) +C(u_t,u).
 \]
 By construction, the quadratic terms are eliminated so this transforms the equation for $u$ into one  for $v$ with only cubic and higher-order nonlinearities,
  \[
L_{KG} \mathbf{u}=\  Q^3(u).
  \]
  To identify $Q^3$ we begin by   recalling the symbols of $A$ , $B$, and $C$. At leading order, $a \approx \xi_1^2 \xi_2 + \xi_1 \xi_2^2$ , $b \approx \xi_1 + \xi_2$, and $c \approx \xi_1\xi_2 + \xi_2^2$. Accordingly, we may write the three bilinear forms schematically as 
 \[
 L(\partial u,\partial u_x),
 \]
 where $L$ denotes a translation invariant operator of order zero, and  \( \partial \) stands for both \( \partial_x\) and \( \partial_t\). 
 
Then we obtain the cubic (and higher-order) source terms as 
\[
\begin{aligned}
   Q^3: = \ \Lambda_{\geq 3}(L_{KG} u)
   +
   \Lambda_{\geq 3}  L_{KG}  L(\partial u ,\partial_x \partial u).
\end{aligned}
\]
It remains to bound these source terms using the energy norm and the Strichartz norms.

\begin{lemma} [Source term estimates] \label{stri-sour}
The following estimates hold in $[0,T ] \times \mathbb{R}$, for $s \geq 7 \frac{1}{4}$
     \[
     \Vert Q^3 \Vert_{L^1_tH_x^{3\frac{1}{4}}} 
     \leq \ 
      C T^{1/2} \Vert \partial^{\leq4} u \Vert_{L^{4}_tL^{\infty}_x}^2 
      \sup_{t \in [0,T]} \| u[t]\|_{H^s\times H^{s-1}}.
     \]
    
\end{lemma}
\begin{proof}
To start with, we expand the nonlinearities:
\[
\begin{aligned}
   Q^3 &= \ \Lambda_{\geq 3}(L_{KG} u)
   +
   \Lambda_{\geq 3}  L_{KG}  L(\partial u ,\partial_x \partial u) \\
    & = \
     \Lambda_{\ge 3} L( \partial u, L_{KG}\partial_x \partial u )
    + \Lambda_{\ge 3}L(L_{KG} \partial u,\partial_x \partial u ) \\
    & 
    \quad \ +  2 \Lambda_{\ge 3}L(\partial_t \partial u,\partial_t \partial \partial_xu )
    + \Lambda_{\ge 3}(f+ g^{\alpha\beta} \partial_{\alpha} \partial_{\beta} u)  
    \\
    & = \ L( \partial u, \partial_x \partial \Lambda_{\ge 2}  u_{tt} )
    + L( \Lambda_{\ge 2} \partial u_{tt},\partial_x \partial u ) \\
    & 
    \quad \ 
    +  2 \Lambda_{\ge 3}L(\partial_t \partial u,\partial_t \partial \partial_xu )
    + \Lambda_{\ge 3}(f+ g^{\alpha\beta} \partial_{\alpha} \partial_{\beta} u) .
\end{aligned}
\]

The most delicate contribution comes from the term with the most unbalanced distribution of derivatives, namely, \( L(\partial u, \partial_x \partial  \Lambda_{\geq 2} u_{tt})\).  
 Applying the fractional Leibniz’ rule together with the Moser-type estimates~\eqref{bmo-hs} and the bounds for  \( u_{tt}\) in \eqref{eq:p2-utt-2}, we obtain, noting that the normalization \( g^{00}=-1\) eliminates the \( (\alpha,\beta) =(0,0)\) terms:
\[
\begin{aligned}
    \Vert   L( \partial u, \partial_x \partial \Lambda_{\ge 2}  u_{tt} )\Vert_{H_x^{3\frac14}}
    \lesssim \
    & \| \partial u \|_{L^{\infty}}
    \| \partial_x \partial \Lambda_{\geq 2} (g^{\alpha\beta} \partial_\alpha \partial_\beta u +f) \|_{H^{3 \frac{1}{4}}}
    \\
    & \quad \
    + 
    \|\partial_x \partial \Lambda_{\geq 2} u_{tt} \|_{L^{\infty}}
    \| \partial u \|_{H^{3 \frac{1}{4}}} 
    \\
    \lesssim \
    & \| \partial u \|_{L^{\infty}}
    \| \partial u \|_{L^{\infty}}
    \| \partial^{4} u \|_{H^{3 \frac{1}{4}}}
    + 
    \| \partial^{4} u \|_{L^{\infty}}
    \| \partial u \|_{L^{\infty}}
    \| \partial u \|_{H^{3 \frac{1}{4}}}
     \\
     \lesssim \
    &
     \Vert  \partial^{\leq 4} u \Vert_{L^{\infty}_x}^2
    \sup_{t \in [0,T]} \| u[t]\|_{H^s\times H^{s-1}} ,  \text{ for } s \ge 7\frac{1}{4}.
\end{aligned}
\]
We emphasize that this requires \(7\frac14\) derivatives from the energy.
The rest of the terms enjoy better bounds. 
For the second and the third terms, we estimate them similarly by using algebra properties together with Moser type estimates~\eqref{bmo-hs} and the bounds for  \( u_{tt}\) in \eqref{eq:p2-utt-2}, noting that the normalization \( g^{00}=-1\) eliminates the \( (\alpha,\beta) =(0,0)\) terms:
\[
\begin{aligned}
      \|L( \partial \Lambda_{\ge 2 } u_{tt},\partial_x \partial u ) \|_{H^{3\frac{1}{4}}_x}  
      & 
      \lesssim \
      \| \partial_x \partial u \|_{L^{\infty}} 
      \| \partial \Lambda_{\ge 2 } u_{tt} \|_{H^{3\frac{1}{4}}_x}  
      +     \| \partial \Lambda_{\ge 2 } u_{tt} \|_{L^{\infty}} 
      \|  \partial_x \partial u  \|_{H^{3\frac{1}{4}}_x} 
      \\
      &
       \lesssim \
        \| \partial^{\le 4 } u \|_{L^{\infty}} 
      \| \partial \Lambda_{\ge 2 }  \bigr( g^{\alpha\beta} \partial_\alpha \partial_\beta u +f \bigl)  \|_{H^{3\frac{1}{4}}_x}  
      \\
      & \quad \
      +     \| \partial \Lambda_{\ge 2 } u_{tt}\|_{L^{\infty}} 
      \|   \partial u  \|_{H^{4\frac{1}{4}}_x} 
      \\
        \lesssim \
    &
     \Vert  \partial^{\leq 4} u \Vert_{L^{\infty}_x}^2
    \sup_{t \in [0,T]} \| u[t]\|_{H^s\times H^{s-1}} ,  \text{ for } s \ge 7\frac{1}{4}.
\end{aligned}
\]

For the terms coming from source terms of the equation for \(u\), a simple Moser type estimate yields the bound:
\[
\begin{aligned}
   \|  \Lambda_{\ge 3}(f(u,\partial u) + g^{\alpha\beta} \partial_{\alpha} \partial_{\beta} u) \|_{H^{3\frac
    14}_x}
      \lesssim \
    &
     \Vert  \partial^{\leq 4} u \Vert_{L^{\infty}_x}^2
    \sup_{t \in [0,T]} \| u[t]\|_{H^s\times H^{s-1}} ,  \text{ for } s \ge 7\frac{1}{4}.
\end{aligned}
\]
To integrate in time, we apply Hölder’s inequality to gain a factor of \(T^{1/2}\):
\[
\begin{aligned}
     \Vert Q^3\Vert_{L_t^1H_x^{3\frac{1}{4}}}
 &
    = \ \int_0^T  \Vert Q^3 \Vert_{H_x^3} \, dt \\
 &
    \leq \ \int_0^T  C \Vert  \partial^{\leq 4} u \Vert_{L^{\infty}_x}^2
    \sup_{t \in [0,T]} \| u[t]\|_{H^s\times H^{s-1}}
  \, dt 
    \\
 &
   \leq \  C \sup_{t \in [0,T]} \| u[t]\|_{H^s\times H^{s-1}}
   \int_0^T   
   \Vert  \partial^{\leq 4} u \Vert_{L^{\infty}_x}^2
   \, dt 
   \\
   &
   \leq \ C \sup_{t \in [0,T]} \| u[t]\|_{H^s\times H^{s-1}}
  T^{1/2} 
    \Vert  \partial^{\leq 4} u \Vert_{L^4_t L^{\infty}_x}^2.
\end{aligned}
\]

\end{proof}

\begin{proof} [Proof of Theorem~\ref{stri}(i)]
    
We now perform a bootstrap argument to obtain a lifespan of order
$\epsilon^{-4}$.  We will make essential use of the cubic structure of
the energy.  Fix $s \ge 7\frac{1}{4}$ and assume that the initial data is small
$\| u[0] \|_{H^s \times H^{s-1}}\le\epsilon$.

Let $T>0$ and suppose that on $[0,T]\times\R$ we have the bootstrap
bounds, for some large constant $C_1$,
\begin{enumerate}
    \item Strichartz estimates,
    \begin{equation*}
          \Vert \partial^{\leq 4} u \Vert_{L^{4}_tL^{\infty}_x} 
     \leq \ C_1 \epsilon ,
    \end{equation*}
   \item energy estimates,
   \begin{equation*}
       \| u[t] \|_{H^s \times H^{s-1}}
     \leq \ C_1 \epsilon  \quad \text{ for  }  s \ge  7\frac{1}{4} \quad \text{ and all }   t \in [0,T] .
   \end{equation*}
\end{enumerate}

Our goal is to show that these bounds can be improved to

\begin{equation}\label{boot-improve}
  \left\{
    \begin{aligned}
         \Vert \partial^{\leq 4} u \Vert_{L^{4}_tL^{\infty}_x} 
    & \leq \ \frac{1}{2} C_1 \epsilon , \\
    \| u[t] \|_{H^s \times H^{s-1}}
    & \leq \ \frac{1}{2} C_1 \epsilon \quad \text{ \ \ \ for \ \  } s \ \ge \ 7\frac{1}{4} .
    \end{aligned}
    \right.
\end{equation}
   
Once \eqref{boot-improve} is established, a standard continuity argument
closes the bootstrap and yields the desired lifespan.

By part (i) of Theorem~\ref{t:eepropagation} we  have the norm equivalence:
\[
\| u[t] \|_{H^s \times H^{s-1}}^2 \approx_{\A_2}  \ E^{s,para}(u[t]).
\]
By part (ii) of Theorem~\ref{t:eepropagation} we  have the cubic energy estimates:
\[
\frac{d}{dt}E^{s,para}(u[t])
  \lesssim_{\A_2} \
  \A_0 \A_3
           \,E^{s,para}(u[t]).
\]

The first inequality in \eqref{boot-improve} follows from the Strichartz
estimate and the source term bounds in Lemma~\ref{stri-sour}. 
Indeed, we have

\[
\begin{aligned}
     \Vert \partial^{\leq 4} u  \Vert_{L^{4}_tL^{\infty}_x}
    &  \leq \
     \Vert \partial^{\leq 4} \mathbf{u}  \Vert_{L^{4}_tL^{\infty}_x}
     + \| \partial^{\leq 4}  L(\partial u , \partial u_x) \|_{L^{4}_tL^{\infty}_x},
\end{aligned}
\]
We bound the two terms separately.
For the first term,  we directly use the Strichartz estimates, combined with the source term bound in Lemma~\ref{stri-sour}. This gives
\[
\begin{aligned} 
     \Vert \partial^{\leq 4} \mathbf{u}  \Vert_{L^{4}_tL^{\infty}_x}
     & 
     \leq \
     C (\epsilon + \Vert \partial^{\leq 3} Q^3\Vert_{L^1_t L^2_x}
     ) 
     \\
     &
     \leq  \ C(\epsilon + C T^{1/2}\Vert \partial^{\leq4} u \Vert_{L^{4}_tL^{\infty}_x}^2 \sup_{t\in [0,T]} \| u[t]\|_{H^1 \times L^2} ),
     \end{aligned}
\]
where \( C > 0 \) is a universal constant independent of 
\( \epsilon \) and \(T\).

For the second term, we estimate using Leibniz' rule and Sobolev embeddings 
\[
\begin{aligned}
     \| \partial^{\leq 4}  L(\partial u , \partial u_x) \|_{L^{4}_tL^{\infty}_x} 
     &
     \leq \ C \big( \int_0^T \| \partial u \|_{W^{5,\infty}_x}^8 \, dt \bigl)^{1/4} 
     \\
     & 
     \leq \ C C_1^2 \epsilon^{2} T^{1/4},
\end{aligned}
\]
where in the last inequality we have used the bootstrap assumption for the energy.
Putting the above bounds together, we obtain
\[
  \|\partial^{\le4}u\|_{L_t^4L_x^\infty}
    \le \
    C\bigl(\epsilon + C T^{1/2} C_1^3 \epsilon^3 + 
     C_1^2 \epsilon^{2} T^{1/4}\bigr).
\]

The right–hand side is bounded by $\tfrac12 C_1\epsilon$ provided $C_1$
is chosen sufficiently large and $T\ll \epsilon^{-4}$.
For the energy bound we use the cubic energy inequality.  Denoting by
$C_2$ the constant in the energy estimate, we have
\[
\begin{aligned}
    \frac{d}{dt} E^{s,para} (u[t])
    &
    \leq  \
    C_2
     \A_0 \A_3
     E^{s,para}(u[t] ).
\end{aligned}
\]
By Gr\"onwall's inequality, we have
\[
\begin{aligned}
     E^{s,para}(u[T]) 
      \leq \
      &
      E^{s,para}(u[0]) \exp\!\Big( C_2  \int_0^T \Vert \partial^{\leq 4} u \Vert_{L^{\infty}}^2 \, dt \Big)  .
\end{aligned}
\]
To estimate the function on the exponent, we use Hölder's inequality,
\[
\begin{aligned}
    C_2  \int_0^T
    \Vert \partial^{\leq 4} u \Vert_{L^{\infty}}^2 \, dt 
    &\leq \
    C_2 T^{1/2}   \Vert \partial^{\leq 4} u \Vert_{L^4_tL^{\infty}_x}^2
    \\
&
    \leq \ C_2 T^{1/2} C^2_1 \epsilon^2.
\end{aligned}
\]
Hence, we have $T \lesssim \epsilon^{-4}$.

\end{proof}
Following the same strategy, we can carry out an analogous argument for the linearized equation. This yields Lipschitz difference estimates on the longer time scale \(\epsilon^{-4}\). For convenience, we recall the linearized equation~\eqref{eq-lin}: 
\begin{equation*} 
    - g^{\alpha \beta}(u,\partial u ) \partial_\alpha
\partial_\beta v  + m v
= \
F^{\gamma,lin} \partial_\gamma v + F^{lin} v,
\end{equation*}
where the coefficients on the right-hand side are given by
\[
\begin{aligned}
F^{\gamma,lin} =  \ g^{\alpha \beta}_{p_{\gamma}}(u,\partial u) \partial_\alpha \partial_\beta u + f_{p_{\gamma}}(u,\partial u),
\quad
 F^{lin} = \  g_{u}^{\alpha\beta}(u,\partial u) \partial_\alpha \partial_\beta  u + f_{u}(u,\partial u).
\end{aligned}
\]

\begin{proof} [Proof of Theorem~\ref{stri}(ii)]
Using the same argument as in the proof of Theorem~\ref{t:weak-Lip} reduces to estimate the bounds for the linearized flow.
Therefore, we show that the cubic energy estimate of Theorem~\ref{t:eepropagation-lin} for the linearized equation remains valid on the longer interval \( [ -c\epsilon^{-4},\,c\epsilon^{-4}]\).

Fix $s \ge 7\frac{1}{4}$ and assume that the initial data for $v$ satisfies
$\| v[0]\|_{H^1 \times L^2} \le\epsilon$.
By part (i) of Theorem~\ref{t:eepropagation-lin}, we have the norm equivalence:
\[
\begin{aligned}
    E_{lin}(v[t]) \approx_{\A_2} \  \Vert v[t]\Vert _{H^{1}\times L^2}^2 .
\end{aligned}
\]
By the cubic energy estimates in Theorem~\eqref{t:eepropagation-lin},
\[
\begin{aligned}
    \frac{d}{dt}   E_{lin}(v[t]) 
    \lesssim_{\A_2} \
    \A_0 \A_3
    E_{lin}(v[t]) .
\end{aligned}
\]
By Gr\"onwall's inequality, we have
\[
\begin{aligned}
     E_{lin}(v[T]) 
      \lesssim \
      &\;
       E_{lin}(v[0]) \exp\!\Big( \int_0^T  \A_0 \A_3 \, dt \Big) .
\end{aligned}
\]
The integral in the exponent is the same one estimated in the second part of the previous proof. In particular, for \(T\lesssim \epsilon^{-4}\) we obtain the bound
\[
 E_{lin}(v[t]) 
      \lesssim \ \epsilon,
\]
holding over the longer interval \( [ -c\epsilon^{-4},\,c\epsilon^{-4}]\), with $c>0$ a universal constant.
\end{proof}

\printbibliography
\end{document}